%% file: Exp1185-Aka.tex
\documentclass[brochure,12pt,reqno,english]{smfbourbaki}

\usepackage[T1]{fontenc}
\usepackage{lmodern}
\usepackage[colorlinks=true, pdfstartview=FitV,
linkcolor=blue, citecolor=red, urlcolor=blue]{hyperref}
\usepackage{graphicx}
\usepackage{float}
 \usepackage{microtype}

\usepackage[english]{babel}

\usepackage{pgf,pgfplots}
\pgfplotsset{compat=1.17}
\usepackage{tikz,tikz-3dplot}
\usetikzlibrary{math}

\usepackage[utf8]{inputenc}

\usepackage[
backend=biber,
style=authoryear, 
citestyle=authoryear-comp,
maxnames=7
]{biblatex}
\usepackage{csquotes}

\DeclareNameAlias{sortname}{given-family}

\renewcommand{\bibnamedash}{\leavevmode\raise3pt\hbox to3em{\hrulefill}\space}

\AtEveryBibitem{%
  \clearfield{issn} 
  \clearfield{doi} 

  \ifentrytype{online}{}{
    \clearfield{url}
  }
}

\renewbibmacro{in:}{%
    \ifentrytype{article}{}{\printtext{\bibstring{in}\intitlepunct}}}

\DeclareFieldFormat[article,periodical,inreference]{number}{\mkbibparens{#1}}
\DeclareFieldFormat[article,periodical,inreference]{volume}{\mkbibbold{#1}}
\renewbibmacro*{volume+number+eid}{%
    \printfield{volume}%
    \setunit*{\addthinspace}
    \printfield{number}%
    \setunit{\addcomma\space}%
    \printfield{eid}}

\DeclareFieldFormat[article,inbook,incollection]{title}{\enquote{#1}\addcomma} 

\date{Novembre 2021}
\bbkannee{74\textsuperscript{e} ann\'ee, 2021--2022}  
\bbknumero{1185}                                

\title{Joinings classification and applications}
 \subtitle{after Einsiedler and Lindenstrauss}

 \author{Menny Aka}
\address{ ETH Zürich - DMATH \\ 
Rämistrasse 101 \\ 8092 Zürich, Switzerland
}
\email{menashea@math.ethz.ch}
\input{Macros.tex}

\begin{document}

\maketitle
\stepcounter{section}
\section*{Introduction}
Fix a group $A$ and consider a set of measure-preserving actions of $A$ on $\mathrm{X_i}=\pa{X_i,\cB_i,\mu_i}$, $i=1,\dots,r$,  where $X_i$ is a Borel probability space with a measure $\mu_i$ and a $\sigma$-algebra~$\cB_i$. Consider the joint action (also called the diagonal action) of~$A$ on 
$$
\mathrm{X}=\pa{X_1\times\dots\times X_r,\cB_1\times\dots\times\cB_r}
$$ 
given by $a.(x_1,\dots,x_r)=(a.x_1,\dots,a.x_r).$ A ($r$-fold) \emph{joining} of the systems $\set{\mathrm{X_i}}_{i=1}^r$ is an $A$-invariant probability measure $\mu$ on $X$ with $\pa{\pi_i}_{*}\mu=\mu_i$ for $i=1,\dots,r$ where $\pi_i\colon X\to X_i$ is the natural projection map. 

There always exists at least one joining, namely the \emph{trivial joining}, which is the product measure $\mu_1\otimes\dots\otimes\mu_r$. When this is the only possible joining of the systems $\set{\mathrm{X_i}}_{i=1}^r$ one says that these systems are \emph{disjoint}. The systematic study of joinings stems from Furstenberg's seminal  paper \parencite{FurstenbergDisjoint}. Furstenberg marked an analogy between joinings and the arithmetic of integers: saying that two measure-preserving systems are disjoint is analogous to saying that their least common multiple is their product. The analogy works in one direction; measure-preserving systems admitting a non-trivial common factor are never disjoint: recall that a \emph{factor} of a measure-preserving system $\mathrm{X}=(X,\cB,\mu,A)$ is a measure-preserving system $\mathrm{Y}=(Y,\cC,\nu,A)$ and a measure-preserving map $\phi\colon X\to Y$ which intertwines the action of $A$, that is, for all $a\in A$ we have $a.\phi(x)=\phi(a.x)$ for $\mu$-almost everywhere. Like integers, any measure-preserving system has itself and the trivial system (one-point system) as factors. Moreover, like integers, as stated above,  if two measure-preserving systems have a common factor, they have a non-trivial joining, called the relatively independent joining over a common factor (see e.g.,~\cite[\S 6.5]{EWBOOK}). \textcite{FurstenbergDisjoint} asked if this analogy also works in the other direction: if two systems do not have any common factor, must they be disjoint? \textcite{RudolphCounterExample} answered negatively, providing the first counterexamples. Joinings are nonetheless a strong tool in ergodic theory, as exemplified by  \textcite{GlasnerJoiningsBook} which gives a complete treatment of ergodic theory via joinings. The broad applicability of the classification of possible joinings of certain systems was already visible in the work of \textcite{FurstenbergDisjoint}, where he solves a question in Diophantine approximation using joinings. We refer the reader also to the recent survey of \textcite{delaRue2020} about the broad use of  joinings in ergodic theory.

Roughly said, the study of joinings is the study of all possible ways two systems (or $r$ systems) can be embedded as factors of another system, which is in turn spanned by them. When two systems are \emph{not} disjoint, this is a sign that there is strong relation between them. The main topic of this survey is a very good example of this principle.  This is a survey of the work of Einsiedler and Lindenstrauss on joinings of higher-rank torus actions on $S$-arithmetic homogeneous spaces \parencite{EL_Joinings2019}, which extends their previous paper \parencite{EL_Joinings2007}. They consider torus actions on two (or $r$) homogeneous spaces which are quotients of $S$-arithmetic points of perfect algebraic groups, equipped with the uniform Haar probability measure on each quotient. They show in particular that if such systems are not disjoint, there must be a strong \emph{algebraic} relation between the corresponding perfect algebraic groups, exemplifying the principle stated above. This may remind the reader of the folklore Goursat's Lemma from group theory; while the latter is a natural structural theorem about subgroups of a product, the joining theorem of Einsiedler and Lindenstrauss is a striking instance of measure rigidity, where the existence of non-trivial joinings in this setting can only be due to a strong algebraic relation.

The main result of   \textcite[Theorem 1.7]{EL_Joinings2019} classifies joinings on higher-rank torus actions on a product of two (or $r$) homogenous spaces of the form 
$$
\Ga_1\backslash\bG_1(\bQ_S)\times \Ga_2\backslash \bG_2(\bQ_S) 
$$
as we now state after recalling the necessary definitions. The measure spaces we consider  are $S$-arithmetic homogeneous quotients of perfect groups. More precisely, let $\bG$ be a perfect Zariski-connected linear algebraic group defined over $\bQ$ and let $S$ be a finite set of places of~$\bQ$. Let $\bQ_S$ denote $\prod_{s\in S}\bQ_s$ (with $\bQ_\infty=\bR)$. An $S$-arithmetic quotient is a quotient space of the form $\Gamma\backslash G$ with $G$ being a finite-index subgroup of $\bG(\bQ_S)$ and $\Ga$ is an irreducible arithmetic lattice commensurable to $\bG(\cO_S)$. Here, $\cO_S$ denotes the ring of $S$-adic integers. Such an $S$-arithmetic quotient is said to be \emph{saturated by unipotent} if the group generated by all unipotent elements of $G$ acts ergodically on $\Ga\backslash G$. For example, for $\bG=\SL_n$ (or more generally simply-connected algebraic groups) the quotient $\Ga\backslash \bG(\bQ_S)$ is saturated by unipotents.

A probability measure $\mu$ on an $S$-arithmetic quotient $\Ga\backslash G$ is called \emph{algebraic over $\bQ$} if there exists an algebraic group $\bH$ defined over $\bQ$ and a finite-index subgroup $H<\bH(\bQ_S)$ such that $\mu=m_{\Ga Hg}$ where $g\in G$ and $m_{\Ga Hg}$ denotes the normalized Haar measure on a single (necessarily closed, by the finiteness of $\mu$) orbit - see \S \ref{subsec:settings} for a detailed definition.  

The joinings we aim to classify are joinings of $S$-arithmetic quotients $X_i=\Ga_i\backslash G_i$ which are saturated by unipotents, equipped with Haar probability measures $m_{X_i}=m_{\Ga \backslash G}$, and a torus action which we now define. Following the notation of \textcite{EL_Joinings2019} we say that a subgroup $A<G$ is of class-$\cA'$ if it is simultaneously diagonalizable and the projection of $a\in A$ to $\bG(\bQ_s)$  for any $s\in S$ satisfies the following: for $s=\infty$ it has only positive real eigenvalues, and for $s$ equal to a finite prime $p$, we assume that all the eigenvalues are powers of $\theta_p$ for some $\theta_p\in \bQ_p^\times$ with $\av{\theta_p}_p\neq 1$ chosen independently of $a\in A$. A homomorphism $\phi\colon\bZ^d\to G$ is said to be of class $\cA'$ if it is proper and $\phi(\bZ^d)$ is of class-$\cA'$. The term \emph{higher-rank torus action} refers to such a homomorphism with $d\geq 2$. We are ready to state the main theorem of \textcite[Theorem 1.7]{EL_Joinings2019}:

\begin{theo}[Einsiedler--Lindenstrauss, 2019]\label{thm:mainThm}
Let $r,d\geq 2$ and let $\bG_1,\dots, \bG_r$ be perfect algebraic groups defined over $\bQ$, $\bG= 
\prod \bG_i$, and $S$ be a finite set of places of~$\bQ$. Let $X_i= \Gamma_i\backslash G_i$ be  $S$-arithmetic quotients for $G_i< \bG_i(\bQ_S)$ which are  saturated by unipotents and set $G=\prod_{i=1}^r G_i$ and $X =\prod_{i=1}^r X_i$. Let $\phi_i\colon\bZ^d\to G_i$ be homomorphisms such that $\phi=(\phi_1,\dots,\phi_r)\colon \bZ^d \ra G$ is of class-$\cA'$, and such that the projection of $\phi_i$ to every $\bQ$-almost simple factor of $\bG_i(\bQ_S)$ is proper. Let $A=\phi(\bZ^d)$ and suppose $\mu$ is an $A$-invariant and ergodic joining of the actions of $A_i = \phi_i(\bZ^d)$ on $X_i$ equipped with the Haar measure $m_{X_i}$. Then, $\mu$ is an algebraic measure defined over $\bQ$.
\end{theo}

This theorem exemplifies the above principle concerning disjointness: let $\bH<\bG$ be the group showing the algebraicity of $\mu$. If $\bH=\bG$ then $\mu$ is the trivial joining. Otherwise, $\bH$ arises from a very strong relation between the algebraic groups $\bG_i$. Indeed, certain of their $\bQ$-simple factors need to be isogenous \emph{over $\bQ$}. In particular, if $\bG_i$ are pairwise non-$\bQ$-isogenous almost simple groups, any joining must be the trivial one. This situation strongly echoes Goursat's Lemma from group theory.

Taking again the broader viewpoint of measure rigidity for torus action (or $\bZ^d$-actions) on homogeneous spaces, Theorem \ref{thm:mainThm} is the most complete result in this context. Such rigidity results are currently only possible under a positive entropy assumption. In our context, the positive entropy assumption is hidden in the assumption that we join homogeneous spaces equipped with the Haar probably measure on each quotient (we give more details below). Moreover, the assumption that the groups are perfect is essential: considering more general groups in both factors would allow to recast the classification of $\bZ^k$-actions on solenoids (including the zero entropy case - a notoriously difficult problem), as a classification problem of joinings.

Theorem \ref{thm:mainThm} is already interesting when $r=2$ and $\bG_1=\bG_2=\SL_n$   for $n\geq 3$  and $d\geq 2$, or for $\bG_1=\bG_2=\SL_2\times \SL_2$ and $d=2$. While reading this survey, the reader is advised to concentrate on these cases. Indeed, the  techniques used and the main steps of the proofs of \textcite{EL_Joinings2007} and \textcite{EL_Joinings2019} are already visible when one considers the case where~$\bG_1$ and~$\bG_2$ are equal to $\SL_n$ for $n\geq 3$ or to $\SL_2\times\SL_2$, and where $S=\set{\infty}$, that is, where we consider real Lie groups. Therefore, apart from describing the main result of \textcite{EL_Joinings2019} in this introduction, we will reduce this survey to these cases.

To end this introduction we present a few images of the following arithmetic application \parencite{AES2016} which appeared at the same time as \parencite{EL_Joinings2019}. We discuss further applications in \S \ref{sec:Applicaitions}.

For $D\in \bN$ write 
$$
\bS^2(D)=\set{\pa{x,y,z}\in \bZ^3: x^2+y^2+z^2=D,\gcd\pa{x,y,z}=1}.
$$
By Legendre and Gauss we have $\bS^2(D)\neq\emptyset$ if and only if $D\neq 0,4,7\mod 8$. Consider 
\begin{equation}\label{eq:EquiOnSpheres}
P_D:=\frac{1}{\sqrt{D}}\cdot \bS^2(D)\subset \bS^2:=\set{\pa{x,y,z}\in \bR^3: x^2+y^2+z^2=1}.    
\end{equation}
By a celebrated theorem of \textcite{Duke88}, based on a breakthrough of  \textcite{IwaniecBreak},   $P_D$ equidistribute on $\bS^2$ when $D\to \infty$ along $D\neq 0,4,7\mod 8$. That is, the following weak-* convergence  
$$
\mu_D:=\frac{1}{\av{\bS^2(D)}}\sum_{v\in \bS^2(D)}\delta_{\frac{v}{\sqrt{D}}}\longrightarrow m_{\bS^2}
$$
holds, where $m_{\bS^2}$ is the uniform (cone) measure on $\bS^2$.

We wish to join this equidistribution problem with another equidistribution problem in a natural way. For each $v\in \bS^2(D)$ we consider the two-dimensional lattice $\Lambda_v:=v^\perp\cap\bZ^3$ which we can consider up to rotation as lying in  a fixed plane of $\bQ^3$. We denote it by $[\Lambda_v]$ and call it \emph{the (shape of the) orthogonal lattice of $v$}.  The set $Q_D:=\set{[\Lambda_v]:v\in \bS^2(D)}$ can be considered as a subset of the modular surface $X_2:=\SL_2(\bZ)\backslash\bH$ which parametrizes the space of two-dimensional lattices up to rotation, and carries a natural invariant probability measure $m_{X_2}$. A careful analysis (see e.g., \cite[\S 5.2]{EMVpoints}) shows that the normalized counting measure on $Q_D$ also equidistributes as $D\to \infty$ to $m_{X_2}$, by a variant of Duke's Theorem. This construction yields the following natural problem: does the normalized counting measure on 
\begin{equation}\label{JdFor1in3}
J_D:=\set{\pa{v,[\Lambda_v]}: v\in \bS^2(D)}
\end{equation}
equidistribute to the product measure $m_{\bS^2}\otimes m_{X_2}$ when $D\to\infty$ with $D\neq 0,4,7\mod 8$?

We conjecture that it does (mainly because we don't see a reason why it shouldn't). Here is some “visible” evidence: for the $D$'s below, we divide the modular surface $X_2$ using the height function into two (resp.~three) equal $m_{X_2}$-measure regions and call lattices in each region non-stretched/stretched (resp.~non-stretched/mildly stretched/super-stretched) and color each point on $\frac{1}{\sqrt{D}}\cdot \bS^2(D)$ with a different color according to the type of its orthogonal lattice. In figures \ref{two_colors} and \ref{three_colors} below, one can see the distribution of the corresponding points together with the number of points of each type for $D=101,8011,104851,14500001$.

\begin{figure}[ht]
\centering
\includegraphics[width=0.9\textwidth]{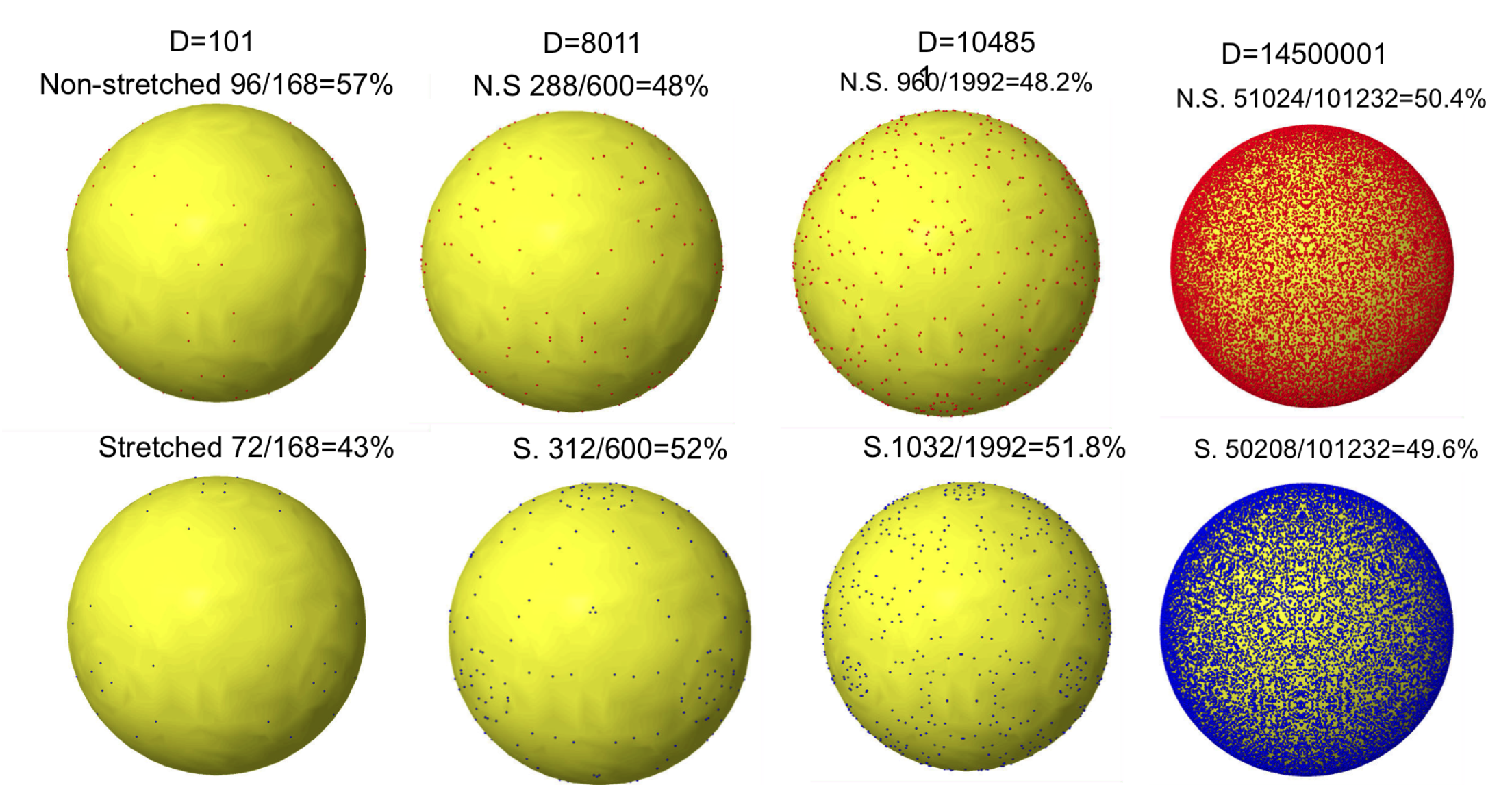}
\caption{non-stretched vs. stretched}
\label{two_colors}
\end{figure}

\begin{figure}[ht]
\centering
\includegraphics[width=0.9\textwidth]{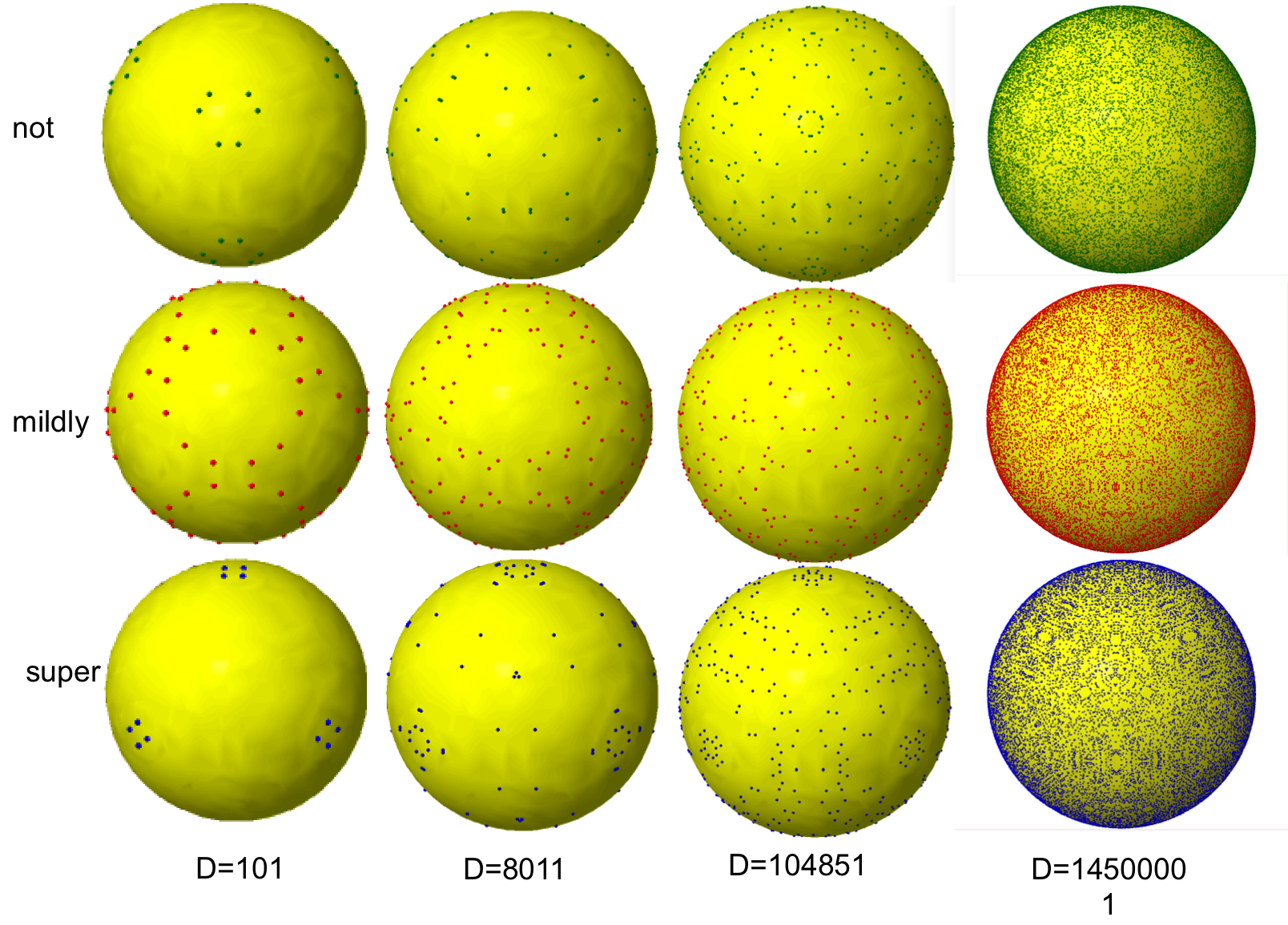}
\caption{non-stretched vs. mildly stretched vs. super-stretched}
\label{three_colors}
\end{figure}

Both equidistribution problems in $\bS^2$ and in $X_2$ may be individually phrased as two individual equidistribution problems on an $S$-arithmetic (or adelic) quotient as defined above (see \S~\ref{sec:TorusOrbitsArith} for more details). Linnik could prove these results under a congruence condition on $D$ modulo a fixed arbitrary prime (see \S~\ref{sec:TorusOrbitsArith} for more details). It turns out that the coupling of $v\in \bS^2(D)$ with its orthogonal lattice $[\Lambda_D]$ gives rise to a joining of the above $S$-arithmetic quotient. Under congruence conditions at \emph{two} fixed arbitrary primes, one could apply Theorem~\ref{thm:mainThm} to deduce the equidistribution of the normalized counting measure on $J_D$ to $m_{\bS^2}\otimes X_2$ when $D\to \infty$ along $D\neq 0,4,7\mod 8$ and the congruence conditions modulo the above two fixed primes. Recently \textcite{blomer2020simultaneous} showed that under the Generalized Riemann Hypothesis, the above equidistribution holds along $D\neq 0,4,7\mod 8$ without any congruence conditions.

\subsection{Bird's-eye view and the organization of this survey}\label{sec:birdsAndIngredients}
Let's give a (very subjective and entropy-centred) bird's-eye view of the main ingredients used in the proofs of the main theorems of the works \textcite{EL_Joinings2007,EL_Joinings2019}.
There are four main ingredients, all related to entropy:

\begin{itemize}
    \item \textbf{Basic ingredient:} Leafwise measures, Lyapunov weights, entropy contribution, and in particular the relation between maximal entropy contribution and invariance.
    \item\textbf{Second ingredient:} Product structure for coarse Lyapunov weights, Abramov--Rokhlin Formula for coarse Lyapunov weights as a corollary. 
    \item\textbf{Third ingredient:}
    The high-entropy method.
    \item\textbf{Fourth ingredient:} The low-entropy method. 
\end{itemize}

As entropy considerations underlie all the above ingredients and every aspect of the work we survey, the title of \S \ref{Sec:EntropyEverything} is entropy. In \S \ref{subsec:settings} we discuss a general setting and notation for the entire survey. In Sections \ref{subsec:Weights}-\ref{subsec:ContAndInva} we review the basic ingredient, giving a bit of intuition and introducing several representative examples that will be discussed throughout the survey. In Sections \ref{subsec:ProductStr}-\ref{Subsec:AbrohmovRokhlinForLya} we discuss the second ingredient. Besides the product structure, which is discussed in \S \ref{subsec:ProductStr}, Lemma~\ref{lem:Inequality} also plays a central role. Its proof is based on a construction of a special partition (a partition which is subordinate to a given subgroup of a given stable horospherical subgroup). Such constructions play a major role also in the proofs of the results, which constitute the basic ingredient. 

Using only the basic and the second ingredients, one can already draw several interesting corollaries, which we discuss in Section \ref{sec:twoCoro}.   For example, Corollary \ref{cor:differntRootDisj}, which we call the “two ingredients joinings theorem”, already gives the ability to reach strong conclusions on joinings in specific situations (which actually arise in applications), and Corollary \ref{cor:SupportProjectsOnto}  gives an important first step toward proving Theorem \ref{thm:mainThm}, and is used again and again throughout the proof.

In section \ref{sec:highRank} we introduce the third ingredient, the high-entropy method, which is summarized in theorem \ref{thm:highEntropy}. The main Theorem of \textcite{EL_Joinings2007} uses exactly these three ingredients (and could be referred to as the “three ingredients joinings theorem”) and the method of its proof is exemplified by classifying all joinings in a specific representative example, see Theorem \ref{Thm:SL3Classification}.  

The main theorem of \textcite{EL_Joinings2019} (Theorem \ref{thm:mainThm}) generalizes the main theorem of \textcite{EL_Joinings2007} in many senses, but the main difference is that  algebraic groups $\bG_i$ which are products (with at least two factors) of forms of $\SL_2$  had been previously excluded. Indeed, the third ingredient, the high-entropy method, implies nothing when the algebraic groups $\bG_i$ are products of rank one  groups. To also treat these cases, Einsiedler and Lindenstrauss used the fourth ingredient, the low-entropy method, yielding the “four ingredients joinings theorem” - Theorem \ref{thm:mainThm}. The general formulation of the low-entropy method may seem intimidating, but the proof of it, although highly intricate, is very concrete. We therefore choose to give a rough sketch of a complete proof of the step involving the low-entropy method for another representative example, see \S \ref{sec:Rank2}. This  sketch is based on a corresponding step in the proof of  arithmetic quantum unique ergodicity by \textcite{LindenstraussQUE}, where the low-entropy method was introduced (compare also to \cite[Ch.10]{PisaNotes}), and to the corresponding step in \textcite{EKL06}. 

The second part of the survey deals with applications of Theorem \ref{thm:mainThm}, which we discuss in Section \ref{sec:Applicaitions}. In \S \ref{sec:TorusOrbitsArith} we give a (rather long) survey of how arithmetic problems relate to adelic torus orbits. We hope that the novice reader could gain some intuition about the use of adeles and $p$-adic numbers in dynamics from this subsection. In \S \ref{subsec:ArithGeoJoinings} we finally present several problems which are related (directly or retrospectively) to a coupling of the problems discussed in \S \ref{sec:TorusOrbitsArith}. We then explain how and under which conditions, Theorem \ref{thm:mainThm} can give key input towards the solutions of these coupled problems. We finally present a new application for groups with high rank in \S \ref{subsec:HighRankAppli}.



\subsection*{Acknowledgements}
It is a pleasure to thank Manfred Einsiedler for being so generous with his knowledge throughout the years, and in particular, for great walks with Saskia while explaining to me  the ideas behind the low-entropy method, leading to the content of \S 5, for suggesting Theorem \ref{Thm:HighRankAppli} and for sketching its proof. I would also like to thank Andreas Wieser and Manuel Luethi for our discussions and for reading parts of a preliminary version of this survey. Thank you to Andreas Wieser for Figures \ref{Fig:Wieser1}-\ref{Fig:Wieser3}, and to Alex  Kontorovich for Figure \ref{Fig:Kontrovich}. Finally, I would like to thank Clemens Bannwart for generating several figures in TikZ, and René Rühr for giving good talks about the low-entropy method.

\section{Entropy}\label{Sec:EntropyEverything}
\subsection{Setting}\label{subsec:settings} We start with a few general definitions. We recall that the action of $G$ or its subgroups on a homogeneous space of the form $\Gamma\backslash G$ is given by $g.(\Gamma h):= \Gamma hg^{-1}$. Given a closed orbit of the form  $\Ga Hg$ for $H<G$, the group preserving it is $g^{-1}Hg$. If the restriction of a Haar measure to a fundamental domain of  $\Gamma\cap\pa{g^{-1}Hg}$ in $g^{-1}Hg$ has a finite measure, we can normalize it to give a fundamental domain measure one. Pushing it forward via the above action to $\Gamma\backslash G$, we get the \emph{normalized uniform/Haar measure} on the orbit~$\Ga Hg$.  

To ease the notation, we will fix a slightly simplified setting and refer to it below. We use the same notation as in Theorem \ref{thm:mainThm} but fix $r=2$. For clarity purposes, we repeat some notations: we let $\bG_1$ and $\bG_2$ be two semisimple algebraic groups defined over $\bQ$. One may keep in mind the following  “baby-cases”: $\bG_1=\SL_n,\bG_2=\SL_{n'}$ (with $n=n'$ being an interesting case), and the case $\bG_1=\bG_2=\SL_2\times\SL_2$. As in Theorem~\ref{thm:mainThm}, we let $G_i<\bG_i(\bQ_S)$ be subgroups,  $\Gamma_i$ be two \emph{irreducible} lattices in $G_i$ and denote $X_i=\Gamma_i\backslash G_i$ and $X=X_1\times X_2$. We assume that the action of $G_i$ on $X_i$ is saturated by unipotents. In the above “baby-cases” one can think about $S=\set{\infty}$ and $G_i=\SL_n(\bR)$ or $G_i=\SL_2(\bR)\times\SL_2(\bR)$ (or $G_i=\SL_2(\bR)\times\SL_2(\bQ_p)$) where the saturated by unipotents assumption is satisfied.  In the “baby-cases” we consider below,  all the images of the class $\cA'$-homomorphisms we consider, will be embedded in the (product of the) respective diagonal subgroup of $\SL_n$ and it will be easy to verify the class-$\cA'$ assumption for them. We denote further by $a=\pa{a_1,a_2}$ a diagonalizable element. The spaces $X_i$ are equipped with the Haar measure $m_{X_i}$ and we denote by $\mu$ an unknown measure on $X$, which is normally the unknown joining that we are trying to classify.

In very rough terms, given a joining $\mu$ as in the main theorem, Einsiedler and Lindenstrauss utilize and develop methods concerning entropy in order to find that $\mu$ is invariant under an unipotent element. Measure rigidity results \parencite{Ratner91,RatnerSadic95,MargulisTomanov94} are then employed to conclude that $\mu$ is algebraic. It is therefore essential for us to introduce and discuss a few of these entropy methods. These methods were predominantly developed by Einsiedler, Katok, Lindenstrauss and Spatzier. 

We will assume that the reader knows the basic definitions and properties of entropy (see e.g., \cite[\S 3]{PisaNotes}) and of conditional measures (see e.g., \cite[\S 5.3]{EWBOOK}).

\subsection{Weights and Lyapunov weights}\label{subsec:Weights}
For a diagonalizable regular element $a\in G$ the subgroup
$$
G_a^-=\set{g\in G: a^nga^{-n}\stackrel{n\to \infty}{\longrightarrow}e},
$$
where $e$ denotes the identity element in $G$, is called the  \emph{stable horospherical subgroup of~$a$}. Its counterpart $G_a^+:=G_{a^{-1}}^-$ is called the \emph{unstable horospherical subgroup of} $a$. These groups are central to the study of the action of $a$ on homogeneous spaces  of the form $\Gamma\backslash G$. For instance, we will  explain later that the entropy $h_\mu(a)$ of the action of $a$ is equal to a quantity that may be calculated solely through $G_a^{-}$ --- the entropy contribution $h_\mu(a,G_a^-)$. For now, let's just say that $G_a^-$ is “defined using the dynamics of $a$” and that it contains “dynamical information”. It turns out that $G_a^-$ is too crude for us in order to extract the information we need for the joinings classification (like invariance under one of its elements). It is therefore interesting to know which subgroups of $G_a^-$ can also be defined “dynamically”. To this end, we will need to consider the action of the whole diagonalizable subgroup $A=\phi(\bZ^d)$ (as in the notation of Theorem \ref{thm:mainThm}) and define weights for the action of $A$.

Recall that the adjoint action of an element $g\in G$ on the Lie algebra $\gog=\Lie G$ describes locally the conjugation action on $G$ as  $\exp$ and $\log$ (at least in characteristic zero) are local isomorphisms. This in turn describes the local dynamics on a homogeneous space of the form $\Gamma\backslash G$, as each point in $\Gamma\backslash G$ (for a discrete group $\Gamma$) is locally isomorphic to $G$.

Consider now $A=\phi(\bZ^d)$ for a class-$\cA'$ homomorphism $\phi$. A character $\lambda\colon A\to k^\times $ (here $k=\bQ_s$ for some $s\in S$) is called a \emph{weight} or a \emph{Lyapunov weight} if there exists a non-zero $\rmx\in\gog$ which is a common eigenvector for the adjoint action of $A$, that is, for every $a\in A$ we have $\Ad_a(\rmx)=\lambda(a)\rmx$, where $\Ad$ denotes the adjoint representation $\Ad_a(\rmx)=a\rmx a^{-1}$ . Once we consider a character via precomposition with $\phi$ as a map from $\bZ^d$ to $k^\times$, it follows from the class-$\cA'$ assumption, that characters are of the form $\lambda(\mathrm{n})=e^{\mathrm{n}\cdot w_\lambda}$ for some $w_\lambda\in \bR^d$ when $k=\bR$, and of the form $\lambda(\mathrm{n})=\pa{\theta_p}^{\mathrm{n}\cdot w_\lambda}$ for some $w_\lambda\in \bZ^d$  when $k=\bQ_p$ (where $\theta_p$ was defined just before Theorem $\ref{thm:mainThm}$). 
For a fixed character $\lambda$ the set of such common eigenvectors forms the weight spaces $\gog^\lambda$. As $A$ (or rather $\Ad(A)$) is simultaneously diagonalizable, we can decompose $\gog$ as 
$$
\gog=\sum_{\lambda\in \Phi}\gog^\lambda
$$
where $\Phi$ is the set of all weights for the action of $A$.
As $G_a^-$ is invariant under conjugation, it follows that 
\begin{equation}\label{eq:StableHoroDecompostion}
    \Lie(G_a^-)=\sum_{\lambda\in \Phi,\av{\lambda(a)}<1} \gog^\lambda
\end{equation}
 and that $\exp$ is a global homomorphism from this nilpotent Lie algebra to $G_a^-$. We may then ask ourselves if $\exp(\gog^\lambda)$ is a subgroup and if it is “dynamically defined”. First note that $\gog^\lambda$ is not necessarily a subalgebra when $\lambda^2$ is also a weight (in general we have  $[\gog^\gamma,\gog^\eta]\subset \gog^{\gamma+\eta}$). But more importantly, in order to find “dynamically-defined” subgroups, we may carry out one of the following two equivalent constructions. We can check which subalgebras of the form $\gog^\lambda$ are indistinguishable from one another using the whole dynamics of $A$. More precisely, we may say (temporarily) that two characters are equivalent and write $\lambda\sim\lambda'$ if for any $a\in A$ we have 
 $$
 \gog^\lambda\subset \Lie(G_a^-)\iff \gog^{\lambda'}\subset \Lie(G_a^-).
 $$ The union of an equivalence class of weights under this relation forms a subalgebra, and its image under $\exp$ is a subgroup of $G_a^-$ for some $a\in A$, which is “defined dynamically”.  Alternatively, we could have asked: which subgroups are the smallest non-trivial intersection of subgroups of the form $G_a^-$ for various $a\in A$? Both questions lead to the same answer; the resulting subgroups are called the coarse Lyapunov weights. To define these only in terms of the weights, we say that two weights $\lambda$ and $\eta$ are equivalent and write $\lambda\sim \eta$, if there exist positive integers $m,\ell$ with $\lambda^m=\eta^\ell$. The equivalence classes are called \emph{coarse Lyapunov weights}. Given a coarse Lyapunov weight $[\lambda]$, the sum $\sum_{\eta\sim\lambda}\gog^\eta$ is called the \emph{coarse Lyapunov weight space} and is denoted by $\gog^{[\lambda]}$. This is a nilpotent subalgebra and $\exp$ defines a global homomorphism from $\gog^{[\lambda]}$ to a nilpotent subgroup denoted by $G^{[\lambda]}$, which is called a \emph{coarse Lyapunov subgroup}.
 
 In order to visualize the weight structure in given examples, it is convenient to consider the logarithm of the (real or $p$-adic) absolute value of a given character: $\log\av{\lambda(\mathrm{n})}$. Viewing this as a map from $\bZ^d$ to $k$ we get a linear map that can be considered as an element of the dual of $\bZ^d\otimes k\cong k^d$. In the real case, $\log\circ \lambda$ is just the inner product with the vector $w_\lambda$ defined above, so one may identify the weight $\lambda$ with $w_\lambda$. Under this identification, the coarse Lyapunov weight $[\lambda]$  is the union of all weights $\eta$ such that $w_\eta\in\bR^+\cdot w_\lambda$.

Recall now the setting in \S \ref{subsec:settings}. Theorem \ref{thm:mainThm} tells us that the set of possible joinings is strongly connected to the relation between $G_1$ and $G_2$, equipped with the corresponding torus action $\phi_i:\bZ^2\ra G_i,\, i=1,2$. As we will explain later, the first instance where this reflects in the proof is the comparison between the weights structure of $G_1$ (with respect to $\phi_1$) and the weights structure of $G_2$ (with respect to $\phi_2$). We will consider four representing examples (Examples \ref{exa:SL2Squared}\eqref{exa:SL2IdenticalWeights}-\ref{exa:SL2Squared}\eqref{exa:SL2DifferentSpeeds}, and Example \ref{exa:SL3} below). The first three are all with $G_1=G_2=\SL_2(\bR)\times \SL_2(\bR)$ but equipped with different torus action. In Example \ref{exa:SL2Squared}\eqref{exa:SL2IdenticalWeights}, both weight structures will be identical, which will give rise to the presence of  possible “diagonal” joinings, e.g.~the one supported on the graph of the identity isomorphism $\mathrm{Id}\colon G_1\to G_2$. In Example \ref{exa:SL2Squared}\eqref{exa:SL245Rotation} the weights structures will be “unrelated” to each other, which will allow us to show that there is no non-trivial joining. For this case, we will just need the basic and the second ingredients mentioned in \S \ref{sec:birdsAndIngredients}, and in particular the high or the low entropy method are not needed,  see \S \ref{subsec:DifferentWeights}. In Example \ref{exa:SL2Squared}\eqref{exa:SL2DifferentSpeeds}, the weights will be related but will have “different speeds”. Also in this example, only the trivial joining may occur, but in order to show this one need to use the low entropy method, see \S \ref{sec:Rank2}).

\begin{exem}[Three examples on $\pa{\SL_2\times \SL_2}^2$]\label{exa:SL2Squared}
    Let $\bG_1=\bG_2=\SL_2\times\SL_2$, $S=\set{\infty}$, $G_i=\bG_i(\bR),\,i=1,2, G=G_1\times G_2$. Let $D$ be the diagonal group in $\SL_2$, 
$$
    U=\set{\begin{pmatrix}
  1 & *  \\
 0& 1\end{pmatrix}},\,V=\set{\begin{pmatrix}
  1 & 0  \\
 *& 1\end{pmatrix}},\,a_t=\begin{pmatrix}
  e^{-\frac{t}{2}} &  \\
 & e^{{\frac{t}{2}}}\end{pmatrix}
$$
 and  $\gou,\gov$ the corresponding Lie algebras. We consider 
three different homomorphisms from $\bZ^2$ to $D\times D$ (the latter is a subgroup of both $G_1$ and $G_2$):
 \begin{equation}\label{eq:threeMaps}
 \phi(t,s)=(a_t,a_s),\quad \psi(t,s)=(a_{t+s},a_{t-s}), \quad \tau(t,s)=(a_{2t},a_{2s}).
 \end{equation}
Each one of them is of class $\cA'$-homomorphism and it has a related weight structure: let $\goh=\Lie(\SL_2(\bR)\times\SL_2(\bR))$ and note that with respect to each of the maps in \eqref{eq:threeMaps},  the weight spaces are $$\goh^{\lambda_1}=\gou\times\set{0},\,\goh^{-\lambda_1}=\gov\times\set{0},\, \goh^{\lambda_2}=\set{0}\times \gou,\, \goh^{-\lambda_2}=\set{0}\times \gov$$
 but for different characters $\lambda_1$ and $\lambda_2$: 
\begin{itemize}
    \item For $\phi(t,s)$:   $w_{\lambda_1}=(-1,0)$ (since $\lambda_1(\phi(t,s))=e^{(t,s)\cdot (-1,0)}=e^{-t}$), and    $w_{\lambda_2}=(0,-1)$.
    \item For $\psi(t,s)$: $w_{\lambda_1}=(-1,-1)$ (since $\lambda_1(\psi(t,s))=e^{(t,s)\cdot (-1,-1)}=e^{-(t+s)}$), and  $w_{\lambda_2}=(-1,1)$.     
    \item For $\tau(t,s)$: $w_{\lambda_1}=(-2,0)$ (since $\lambda_1(\tau(t,s))=e^{(t,s)\cdot (-2,0)}=e^{-2t}$), $w_{\lambda_2}=(0,-2)$.
\end{itemize}
In general, one has  $w_{-\eta}=-w_\eta$, so we omit the weight $-\lambda_i$ above.
These roots systems (only the non-trivial weights) are drawn in Figure \ref{rootSL2Squared}.


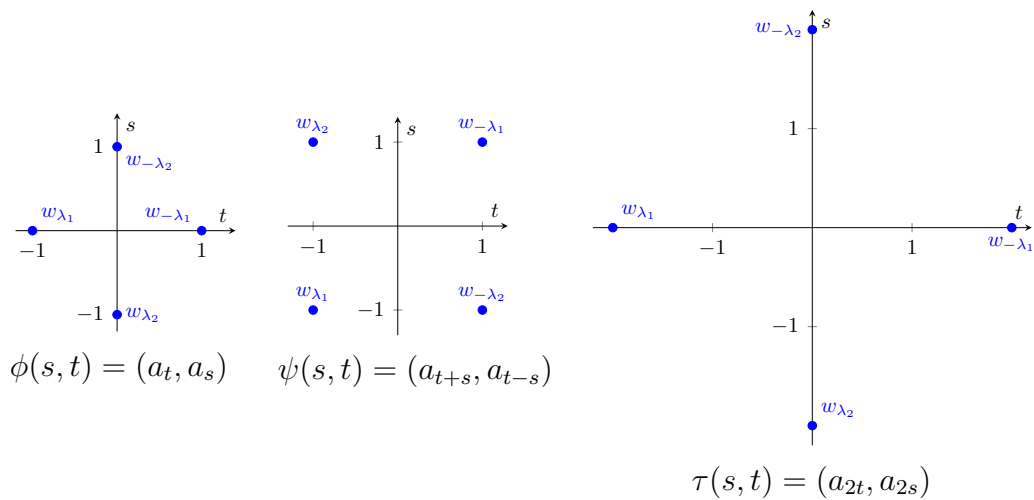
\begin{figure}[h]
\centering
\begin{tikzpicture}[scale=0.8,baseline=-25]
    \begin{axis}[%
    height=5.2cm,width=5.2cm,
    ymin=-1.2,ymax=1.4,xmin=-1.2,xmax=1.4,axis lines=middle,
    every axis/.append style={font=\footnotesize},
    xtick={-1,1},
    ytick={-1,1},
    xlabel={$t$},
    ylabel={$s$}
    ]
    \node[above left] at (1,0) {\color{blue}$w_{-\lambda_1}$};
    \filldraw[blue] (1,0) circle (2pt); 
    \node[below right] at (0,1) {\color{blue}$w_{-\lambda_2}$};
    \filldraw[blue] (0,1) circle (2pt); 
    \node[right] at (0,-1) {\color{blue}$w_{\lambda_2}$};
    \filldraw[blue] (0,-1) circle (2pt); 
    \node[above right] at (-1,0) {\color{blue}$w_{\lambda_1}$};
    \filldraw[blue] (-1,0) circle (2pt); 
    \end{axis}
    \node at (1.7,-0.6) {$\phi(s,t)=(a_t,a_s)$};
\end{tikzpicture}
\hspace{3pt}
\begin{tikzpicture}[scale=0.8]
    \begin{axis}[%
    height=5.2cm,width=5.2cm,
    ymin=-1.3,ymax=1.3,xmin=-1.3,xmax=1.3,axis lines=middle,
    every axis/.append style={font=\footnotesize},
    xtick={-1,1},
    ytick={-1,1},
    xlabel={$t$},
    ylabel={$s$}
    ]
    \node[above] at (1,1) {\color{blue}$w_{-\lambda_1}$};
    \filldraw[blue] (1,1) circle (2pt); 
    \node[above] at (1,-1) {\color{blue}$w_{-\lambda_2}$};
    \filldraw[blue] (1,-1) circle (2pt); 
    \node[above] at (-1,-1) {\color{blue}$w_{\lambda_1}$};
    \filldraw[blue] (-1,-1) circle (2pt); 
    \node[above] at (-1,1) {\color{blue}$w_{\lambda_2}$};
    \filldraw[blue] (-1,1) circle (2pt); 
    \end{axis}
    \node at (2.1,-0.6) {$\psi(s,t)=(a_{t+s},a_{t-s})$};
\end{tikzpicture}
\hspace{3pt}
\begin{tikzpicture}[scale=0.8,baseline=18]
    \begin{axis}[%
    height=8.8cm,width=8.8cm,
    ymin=-2.2,ymax=2.2,xmin=-2.2,xmax=2.2,axis lines=middle,
    every axis/.append style={font=\footnotesize},
    xtick={-1,1},
    ytick={-1,1},
    xlabel={$t$},
    ylabel={$s$}
    ]
    \node[below] at (2,0) {\color{blue}$w_{-\lambda_1}$};
    \filldraw[blue] (2,0) circle (2pt); 
    \node[left] at (0,2) {\color{blue}$w_{-\lambda_2}$};
    \filldraw[blue] (0,2) circle (2pt); 
    \node[above right] at (0,-2) {\color{blue}$w_{\lambda_2}$};
    \filldraw[blue] (0,-2) circle (2pt); 
    \node[above right] at (-2,0) {\color{blue}$w_{\lambda_1}$};
    \filldraw[blue] (-2,0) circle (2pt); 
    \end{axis}
    \node at (3.6,-0.6) {$\tau(s,t)=(a_{2t},a_{2s})$};
\end{tikzpicture}
\caption{Root systems for $\SL_2\times \SL_2$}
\label{rootSL2Squared}
\end{figure}
We  consider now three different torus actions on $\Gamma_1\backslash\bG_1(\bR)\times \Gamma_2\backslash\bG_2(\bR)$ where~$\Gamma_1$ and~$\Gamma_2$ are irreducible arithmetic lattices, and classify joining with respect for each of these three actions.  These actions are given by the following class-$\cA'$ homomorphisms: 
\begin{enumerate}
    \item \label{exa:SL2IdenticalWeights}$(\phi,\phi)\colon\bZ^2\to \pa{D\times D}^2<G_1\times G_2.$ (“Identical weights”)
    \item \label{exa:SL245Rotation}$(\phi,\psi)\colon\bZ^2\to \pa{D\times D}^2<G_1\times G_2.$ (“$45^\circ$ rotation”)
    \item \label{exa:SL2DifferentSpeeds}$(\phi,\tau)\colon\bZ^2\to \pa{D\times D}^2<G_1\times G_2.$ (“Different speeds”)
\end{enumerate}
The  action of $(\phi,\phi) $ has  four (non-trivial) weight spaces and four  coarse Lyapunov weights, the  action of $(\phi,\psi) $ has  eight (non-trivial) weight spaces and eight  coarse Lyapunov weights, and the  action of $(\phi,\tau) $ has  eight (non-trivial) weight spaces but only four  coarse Lyapunov weights. The weight spaces may be visualized as in figure \ref{figure:SL2weightsDiagrams}





\begin{figure}[ht]
\centering
\tikzmath{\d = 0.15;}
\begin{tikzpicture}[%
x={(1 cm,0 cm)},%
y={(-.1 cm, .3 cm)},%
z={(.2, .8 cm)}]
    \foreach \x/\y in {1/0,-1/0,0/1,0/-1}
    {
        \filldraw[red] (\x,-\d,\y) circle (1pt);
    }
    
    \draw[->] (-1.5,0,0)--(1.5,0,0);
    \draw[->] (0,0,-1.5)--(0,0,1.5);
    
    \foreach \x/\y in {1/0,-1/0,0/1,0/-1}
    {
        \filldraw[blue] (\x,\d,\y) circle (1pt);
    }
    \node at (0,0,-2.5) {\footnotesize identical weights};
\end{tikzpicture}
\hspace{15pt}
\begin{tikzpicture}[%
x={(1 cm,0 cm)},%
y={(-.1 cm, .3 cm)},%
z={(.2, .8 cm)}]
    \draw[->] (-1.5,0,0)--(1.5,0,0);
    \draw[->] (0,0,-1.5)--(0,0,1.5);
    
    \foreach \x/\y in {1/0,-1/0,0/1,0/-1}
    {
        \filldraw[red] (\x,0,\y) circle (1pt);
    }
    \foreach \x/\y in {1/1,-1/1,1/-1,-1/-1}
    {
        \filldraw[blue] (\x,0,\y) circle (1pt);
    }
    \node at (0,0,-2.5) {\footnotesize $45^\circ$ rotation};
\end{tikzpicture}
\hspace{15pt}
\begin{tikzpicture}[%
x={(1 cm,0 cm)},%
y={(-.1 cm, .3 cm)},%
z={(.2, .8 cm)}]
    \draw[->] (-2.5,0,0)--(2.5,0,0);
    \draw[->] (0,0,-2.5)--(0,0,2.5);
    
    \foreach \x/\y in {1/0,-1/0,0/1,0/-1}
    {
        \filldraw[red] (\x,0,\y) circle (1pt);
    }
    \foreach \x/\y in {2/0,-2/0,0/2,0/-2}
    {
        \filldraw[blue] (\x,0,\y) circle (1pt);
    }
    \node at (0,0,-3.5) {\footnotesize different speeds};
\end{tikzpicture}
\caption{Decomposition of $\mathfrak{sl}_2^4=\pa{\mathfrak{sl}_2\times \mathfrak{sl}_2}^2$ into weight spaces}
\label{figure:SL2weightsDiagrams}
\end{figure}
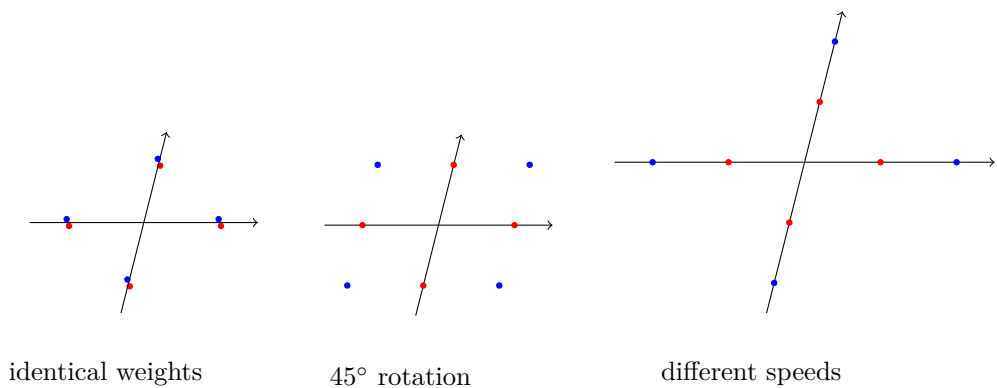

\end{exem}
In the examples above, the (non-opposite) weight spaces always commute with each other. As we will see later, this means that the high entropy method (see \S \ref{sec:highRank}) does not help at all in these cases. This is why these cases were excluded in \textcite{EL_Joinings2007} and dealt with later in \textcite{EL_Joinings2019} using the low entropy method (see \S \ref{sec:Rank2}). To give a representing example where there are non-commuting Lyapunov weights, we consider the following example.
\begin{exem}\label{exa:SL3}
    Let $\bG_1=\bG_2=\SL_3$, $S=\set{\infty}$, $G_i=\bG_i(\bR),\,i=1,2, G=G_1\times G_2$. Let $D<G_i$ be the diagonal group and for $i=1,2$, let  $$\phi_i:\bZ^2\to D,\,\phi_i(t,s)=\begin{pmatrix}
  e^{-\frac{t}{2}} & & \\
 & e^{-\frac{s}{2}}&\\
& & e^{\frac{t+s}{2}} 
 \end{pmatrix}.$$
The map $\phi=\pa{\phi_1,\phi_2}$ is of class $\cA'$ and we set $A=\phi(\bZ^2)<G\times G$. 

Let's concentrate for a moment on one of the factors, say on $G_1=\SL_3(\bR)$ and  $A_1:=\phi_1(\bZ^2)$: for $1\leq i\neq j\leq 3$ let $U_{ij}<\SL_3(\bR)$ be the unipotent one-parameter group with $1$'s on the diagonal and zeroes everywhere except from the $ij$-entry, and $\gou_{ij}$ the corresponding Lie algebra. Then, the eight-dimensional Lie algebra $\mathfrak{sl}_3$ decomposes into weight spaces as $\Lie(D)\oplus \pa{\oplus_{1\leq i\neq j\leq 3} \gou_{ij}}$. 
Furthermore, each of the spaces $\gou_{ij}$ forms a coarse Lyapunov weight with respect to $\phi_1$, and the reader may calculate the corresponding weights to get the weight diagram in Figure \ref{figure:SL3weights}, where we denote the weight vector corresponding to $\gou_{ij}$ with $w_{ij}$.


\begin{figure}[ht]
\begin{center}
\begin{tikzpicture}
    \begin{axis}[%
    height=10cm,width=10cm,
    ymin=-1.6,ymax=1.6,xmin=-1.6,xmax=1.6,
    axis lines=middle,
    every axis/.append style={font=\footnotesize},
    minor tick style={draw=none},
    major tick style=thick,
    grid style={line width=0.1, draw=gray!20},
    major grid style={line width=0.2,draw=gray!60},
    grid=both,
    xtick distance = 1,
    ytick distance = 1,
    minor tick num=3
    ]
    
    \coordinate (w12) at (-.5,.5) {};
    \coordinate (w21) at (.5,-.5) {};
    \coordinate (w13) at (-1,-.5) {};
    \coordinate (w31) at (1,.5) {};
    \coordinate (w23) at (-.5,-1) {};
    \coordinate (w32) at (.5,1) {};
    
    \draw[very thick,blue] (w12)--(w13)--(w23)--(w21)--(w31)--(w32)--(w12);
    
    \filldraw[blue] (w12) circle (3pt);
    \filldraw[blue] (w21) circle (3pt);
    \filldraw[blue] (w13) circle (3pt);
    \filldraw[blue] (w31) circle (3pt);
    \filldraw[blue] (w23) circle (3pt);
    \filldraw[blue] (w32) circle (3pt);
    
    \node[above left] at (w12) {$w_{12}$};
    \node[below right] at (w21) {$w_{21}=-w_{12}$};
    \node[left] at (w13) {$w_{13}$};
    \node[right] at (w31) {$\begin{array}{l}w_{31}= \\ -w_{13}\end{array}$};
    \node[below] at (w23) {$w_{23}$};
    \node[above] at (w32) {$w_{32}=-w_{23}$};
    \end{axis}
\end{tikzpicture}
\end{center}
\caption{Decomposition of $\mathfrak{sl}_3$ into weight spaces}
\label{figure:SL3weights}
\end{figure}
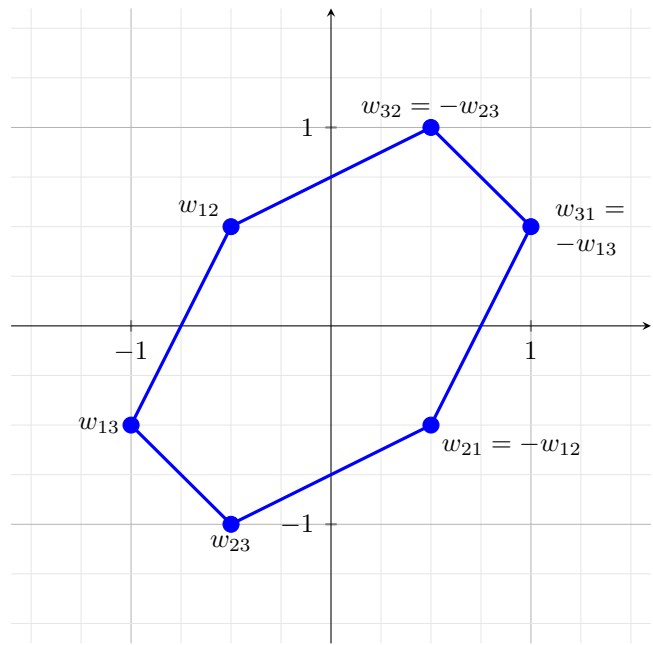

There are commutation relations between the weight spaces, which are also visible via addition of the vectors $w_{ij}$: for example $[\gou_{12},\gou_{23}]=\gou_{13}$. Moreover, one can read from the weight diagram in Figure \ref{figure:SL3weights} the decomposition of the Lie algebra of any stable horospherical: for $a=\phi_1(t,s)$, it follows from the definition of the weight spaces and their corresponding weight vectors that $\pa{\gosl_3}^{-}_{a}$ is the direct sum of all the weight spaces whose weight vector has a negative inner product with  $(t,s)$. That is, the corresponding weight vectors must belong  to a half-plane defined by the line $(t,s)^\perp$. Further, note that for each weight space, the corresponding weight vector can be written as a sum of two other weight vectors, such that all three weight vectors belong to a half-plane related to some $a\in A_1$. It follows that each weight space can be  written as the commutator of two other weight spaces, all belonging to the same stable horospherical of some $a\in A_1$. This fact will be crucial for us in  \S \ref{sec:highRank}.

Fix now  two  irreducible lattices  $\Gamma_1$ and $\Gamma_2$ in $\SL_3(\bR)$ and let's return to consider joinings for the action of $\phi$ on $\Gamma_1\backslash G_1\times \Gamma_2 \backslash G_2$. We have at least two types of joinings in this setting: the trivial joining, which is the product measure $m_{\Gamma_1\backslash G_1}\otimes m_{\Gamma_2\backslash G_2}$, and a “diagonal” joining, for example, the Haar measure supported on an orbit of $\set{(g,g):g\in \SL_3(\bR)}< \SL_3^2(\bR)$. In Theorem \ref{Thm:SL3Classification} we give a more precise statement and see that these are the only possible types of \emph{ergodic} joinings in this case (In general, an arbitrary joining is a convex combination of ergodic joinings). 
\end{exem}

\subsection{Leafwise measures}

The joining classification is a part of the classification of measures with torus actions on homogeneous spaces, under the assumptions that some elements act with positive entropy. In the joinings classification, the positive entropy assumption is “hidden” in the fact that each factor is equipped with the Haar measure and therefore the action on each factor has positive entropy, which imply that any non-trivial $a\in A$ has positive entropy (see Remark \ref{rem:positiveEnt}). Positive entropy is related to growth/decay properties of the measure. For example, for a hyperbolic toral automorphism positive entropy is equivalent to positive Hausdorff dimension of the measure which, in turn, is equivalent to an almost surely decay property of the measure of $r$-balls when $r\to 0$. In our context, the action of $a\in A$ on $\Gamma\backslash G$ is not hyperbolic as $a$~commutes with all elements of $A$. Nevertheless, one could characterize again positive entropy as an almost sure decay property by considering the so-called Bowen-balls. Moreover, as we will state precisely in Proposition~\ref{prop:EntAsVolume}, one can relate the entropy $h_\mu(a)$ of a diagonalizable element~$a$ on a homogenous space $\Gamma\backslash G$ to decay of the measure along the stable horospherical subgroup~$G_a^-$. Given the decomposition above of $G_a^-$ to coarse Lyapunov subgroups, one may ask which subgroups of~$G_a^-$ contribute entropy. To be able to answer such questions and to be able to state the above statements in a choices-free way (e.g.~independently of a choice of a specific generating partition for the entropy calculation) one is led to the notion of leafwise measures.

The aim of leafwise measures in our context is to describe a given measure $\mu$ on $X:=\Gamma\backslash G$, as in Theorem \ref{thm:mainThm}, along orbits of an $a$-normalized subgroup $U<G_a^-$ with $h_\mu(a)>0$. For example, along $U$ being a coarse Lyapunov subgroup as above. Unfortunately, normally there exists no countably-generated $\sigma$-algebra whose atoms are $U$-orbits (at least for ergodic systems when individual orbits have measure zero). Therefore, the natural candidates that will describe $\mu$ along orbits, conditional measures,  cannot be constructed. Leafwise measures are constructed to remedy this problem. Roughly said, if one restricts to a bounded portion of a $U$-orbit, the orbits of this bounded portion can be the atoms of countably-generated $\sigma$-algebra. One considers the conditional measures which respect to such a $\sigma$-algebra. In order to get a measure that describes the measure $\mu$ along the full orbit, the idea is to consider larger and larger parts of the $U$-orbits and to patch  the resulting conditional measures together using the group $U$ itself as a reference object. This construction results in a family of locally-finite measures $\set{\mu_x^U}_{x\in X}$ on $U$, called the \emph{leafwise measures} of $U$. They are defined almost everywhere and up to proportionality (as the “patching” can only be done up to an arbitrary constant). Defining these measures on $U$ rather on the space itself has the advantage that we can directly exploit the group structure on $U$ when we extract information from $\set{\mu_x^U}_{x\in X}$.

In this next subsection, we give a precise definition of leafwise measures and their characterizing property. The best way to understand their essence is via \textcite[\S 6]{PisaNotes} where the above patchwork is carried out carefully. 
\subsubsection{Definition of leafwise measures}
The setting that we have in mind is a homogeneous space $X=\Gamma\backslash G$, an acting diagonalizable element $a\in G$ which has positive entropy with respect to an $a$-invariant measure $\mu$ and a closed subgroup $U<G_a^-$. Note that a priori $\mu$ has no known invariance under $U$. To help the reader to see what is essential for the construction, we use a slightly generalized setting in the following proposition, which defines the leafwise measures. The subgroup $H$ that appears in the proposition will be later $G_a^-$ or an $a$-normalized subgroup of it.  
\begin{prop}\label{prop:LFMeasures}

  Let $X=\Gamma\backslash G$ be an $S$-adic homogeneous space and $\mu$ a locally-finite measure on $X$. Let $H<G$ be a closed subgroup which acts freely on $X$ almost everywhere, that is, for a.e.~$x\in X$ the map 
    \begin{equation}
    \label{eq:injectivity of H}
    h\in H\mapsto h.x:=xh^{-1}    
    \end{equation}
    is injective.
  Then, there exists a set of full measure $X'\subset X$ and family $\set{\mu_x^H}_{x\in X'}$ of locally-finite measures on $H$, unique up-to proportionality, having the following properties.
  \begin{enumerate}
      \item (Characterizing property - describing $\mu$ along $H$) Let $Y\subset X$ be a measurable set with $\mu(Y)<\infty$ and $\cA$ be a countably generated $\sigma$-algebra on $Y$ which is \emph{$H$-subordinate} (that is, its atoms $[y]_\cA$ are bounded pieces of an $H$-orbit, or more precisely, there exists a bounded open subset $V_y\subset H$ with $[y]_\cA=V_y. y$. The subset~$V_y$ is called the shape of the atom $[y]_\cA$). Then, the leafwise measure $\mu_y^H$ describes the conditional measure $\mu_y^\cA$ in the following sense:
      $$
      \mu_y^\cA =\frac{1}{\mu_y^H(V_y)}\pa{\mu_y^H|_{V_y}. y}
      $$
      where $\mu_y^H|_{V_y}$ stands for the restriction of $\mu_y^H$ to $V_y$ and $\mu_y^H|_{V_y}. y$ for pushing this measure under the map \eqref{eq:injectivity of H}.
      \item \label{item:shifty}(Compatibility/Shifting formula) The measure $\mu_x^H$ describes the orbit $H.x$ with $e\in H$ corresponding to $x\in X$. In particular, the following \emph{shifting formula} holds:
      \begin{equation}\label{eq:shifty}
          \pa{R_h}_{*}\mu_{h.x}^H\propto \mu_x^H 
      \end{equation}
      where $R_h\colon H\to H$ is defined by $ h'\mapsto h'h$ and $\propto$ denote proportionality of measures.
      \end{enumerate}
We normalize\footnote{See \textcite[\S 6.29]{PisaNotes} for a discussion of
  possible normalizations.} the measures $\mu_x^H$ to have $\mu_x^H(B_1^H)=1$
where $B_1^H$ is the ball of radius~$1$ in~$H$ around~$e$. We further have    
      
\begin{enumerate} \setcounter{enumi}{2}

      \item \label{item:Invariance}(Invariance) The measure $\mu$ is $H$-invariant if and only if $\mu_x^H$ is equal to the Haar measure on $H$ for almost every $x\in X$.
     
      \item\label{item: equivary} (Measure preservation implies equivariance) Assume further that there exists $a\in G$ that normalizes $H$ and preserves the measure $\mu$. Then 
      \begin{equation}\label{eq:equivariancy}
      \mu_{a.x}^H\propto \pa{\theta_a}_*\mu_x^H    
      \end{equation}
       where $\theta_a:H\to H$ is defined by $h\mapsto aHa^{-1}$. If particular, if $\theta_a$ acts trivially on $H$, we have $\mu_{a.x}^H= \pa{\theta_a}_*\mu_x^H$ (by the above normalization).
  \end{enumerate}
\end{prop}

Properties \ref{item:shifty} and \ref{item: equivary} can be visualized as follows: Imagine that the colour intensity in Figure \ref{figure:shifty} represents the distribution of the given measure $\mu$. An embedding of an $H$-orbit $H.x=H.\pa{h.x}$ gives rise to compatible distributions $\mu_{x}^H,\mu_{h.x}^H$ on $H$ in the sense that wrapping them on the orbit with $e\in H$ corresponding to the point $x\in X$ (resp.~to the point $h.x\in X$) will be compatible to the given distribution of the measure. This compatibility  give rise to the shifting property stated in equation \eqref{eq:shifty}.

Similarly, imagine that the given measure $\mu$ gives rise to some distribution along a given $H$-orbit $H.x$ as in Figure \ref{figure:equivy}. Imagine further that the $a$ action contracts the neighbourhood around $x$. Then, the distribution along the orbit $H.\pa{a.x}$ is just a contracted version of the distribution along $H.x$. This  gives rise to the equivariance property stated in equation \eqref{eq:equivariancy}.


\begin{figure}[ht]
\centering
\begin{tikzpicture}[baseline=-70]
    \fill[blue!25!white] (-3,.97) rectangle (0.3,1.03);
    \shade[left color=blue!25!white,right color=blue] (0.3,.97) rectangle (2,1.03);
    \fill[blue] (2,.97) rectangle (3,1.03);
    
    \node at (-3.8,1) {$\mu_x^H$};
    
    \draw (-1,1.1)--(-1,.9);
    \node[above] at (-1,1.1) {$e$};
    \draw (1.5,1.1)--(1.5,.9);
    \node[above] at (1.5,1.1) {$h$};
    
    \node[above] at (0,0) {$H$};
    
    \fill[blue!25!white] (-3,-1.03) rectangle (-2.2,-.97);
    \shade[left color=blue!25!white, right color=blue] (-2.2,-1.03) rectangle (-.5,-.97);
    \fill[blue] (-.5,-1.03) rectangle (3,-.97);
    
    \node at (-3.8,-1) {$\mu_{h.x}^H$};
    
    \draw (-1,-1.1)--(-1,-.9);
    \node[above] at (-1,-.9) {$e$};
\end{tikzpicture}
\hspace{20pt}
\begin{tikzpicture}[scale=0.8]
    
    \fill[color=blue!25!white]
    (7,4) to[out=190, in=30] 
    (3,3) to[out=210,in=35] 
    (1,1.5) to[out=215,in=45] 
    (-1.5,.3) to[out=225,in=150] 
    (-1,-1.5) to[out=330,in=180]
    (2,-2.3) to
    (3.5,1.5) to[out=110, in=210]
    (4.5,2.8) to[out=30, in=200]
    (7,3.7);
    
    \shade[left color=white, right color=blue,shading angle=68.46] plot [smooth, tension=0.5]
    (2,-2.3) to[out=0,in=200]
    (4.2,-2) to
    (5.4,1) to[out=180,in=290]
    (3.5,1.5) to
    (2,-2.3);
    
    \draw[thick,blue!80!white] (4.2,-2)--(5.4,1);
    
    \draw[very thick,blue!25!white] (2,-2.3)--(3.5,1.5);
    
    \fill[blue!80!white] 
    (4.2,-2) to[out=20,in=290] 
    (6,.5) to[out=110, in=0]
    (5.4,1);
    
    \draw[thick] 
    (7,4) to[out=190, in=30] 
    (3,3) to[out=210,in=35] 
    (1,1.5) to[out=215,in=45] 
    (-1.5,.3) to[out=225,in=150] 
    (-1,-1.5) to[out=330,in=180]
    (2,-2.3) to[out=0,in=200]
    (4.2,-2) to[out=20,in=290] 
    (6,.5) to[out=110, in=0]
    (5.4,1) to[out=180,in=290]
    (3.5,1.5) to[out=110, in=210]
    (4.5,2.8) to[out=30, in=200]
    (7,3.7);
    
    

    \draw plot [smooth, tension=0.5] coordinates {(-1,0) (1.5,-.6) (4,-.5) (5.5,0)};
    
    
    
    \filldraw (1.5,-.6) circle (1pt);
    \node[below] at (1.5,-.6) {$x$};
    \filldraw (4,-.5) circle (1pt);
    \node[below] at (4,-.5) {$h.x$};
\end{tikzpicture}
\caption{Shifting property}
\label{figure:shifty}
\end{figure}
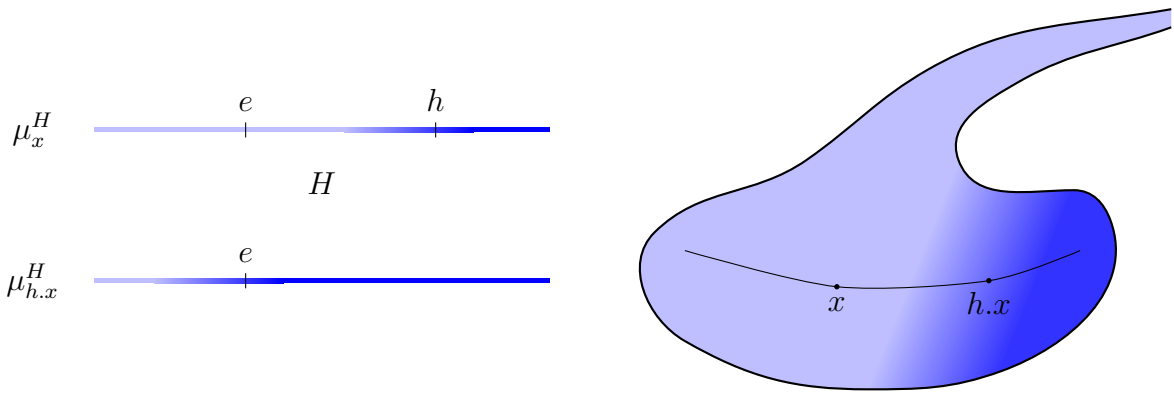


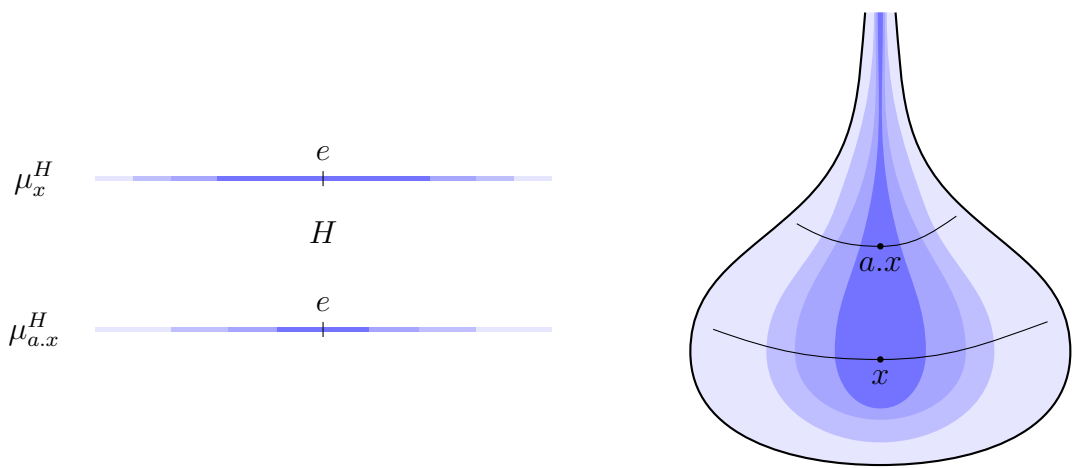
\begin{figure}[ht]
\centering
\begin{tikzpicture}[baseline=-80]
    \foreach \h/\d/\c in {1/3/10, 1/2.5/25, 1/2/35, 1/1.4/55, -1/3/10, -1/2/25, -1/1.25/35, -1/.6/55}
    {
    \fill[blue!\c!white] 
    (-\d,\h-.03) rectangle (\d,\h+.03);
    }
    
    \node at (-3.8,1) {$\mu_x^H$};
    \draw (0,1.1)--(0,.9);
    \node[above] at (0,1.1) {$e$};
    
    \node[above] at (0,0) {$H$};
    
    \node at (-3.8,-1) {$\mu_{a.x}^H$};
    \draw (0,-1.1)--(0,-.9);
    \node[above] at (0,-.9) {$e$};
\end{tikzpicture}
\hspace{40pt}
\tikzmath{
    \x1 = 0.2;
    \y1 = 6;
    \x2 = 1;
    \y2 = 3.5;
    \x3 = 2.5;
    \y3 = 1.5;
    \a = 45; 
    }
\begin{tikzpicture}
    
    \draw[thick,fill=blue!10!white]
    (-\x1,\y1) to [out=265, in=\a]
    (-\x2,\y2) to [out=\a+180, in=90]
    (-\x3,\y3) to [out=270, in=180]
    (0,0) to [out=0, in=270]
    (\x3,\y3) to [out=90,in=-\a]
    (\x2,\y2) to [out=180-\a, in=275]
    (\x1,\y1);
    
    
    
    \foreach \t/\aa/\l/\h/\c in {
    .6/70/.1/.3/25,
    .45/65/.3/.6/35,
    .24/75/.5/.75/55
    } 
    {
        \fill[blue!\c!white]
        (-\x1*\t*.7,\y1) to [out=270, in=\aa]
        (-\x2*\t,\y2-\l) to [out=\aa+180, in=90]
        (-\x3*\t,\y3) to [out=270, in=180]
        (0*\t,\h) to [out=0, in=270]
        (\x3*\t,\y3) to [out=90,in=-\aa]
        (\x2*\t,\y2-\l) to [out=180-\aa, in=270]
        (\x1*\t*.7,\y1);
    }
    
    \draw 
    (-2.2,1.8) to [out=-20, in=180] 
    (0,1.4) to [out=0, in=200]
    (2.2,1.9);
    \filldraw (0,1.4) circle (1pt);
    \node[below] at (0,1.4) {$x$};
    
    \draw
    (-1.1,3.2) to [out=-30, in=180]
    (0,2.9) to [out=0, in=215]
    (1,3.3);
    \filldraw (0,2.9) circle (1pt);
    \node[below] at (0,2.9) {$a.x$};
\end{tikzpicture}
\caption{Equivariance property}
\label{figure:equivy}
\end{figure}

    
 
 We note that the assumption \eqref{eq:injectivity of H} in Proposition  \ref{prop:LFMeasures} holds  when $H=U$ for an $a$-normalized subgroup $U<G_a^-$: if $u.x=x$ for some $u\in U$ then $aua^{-1}.\pa{a.x}={a.x}$. Since $U<G_a^-$ this shows that the injectivity radius of $a^n.x$ goes to zero. So either $u.x=x$ cannot hold (e.g.,~when $X$ is compact) or $x$ fails Poincaré recurrence (with respect to the measure that $a$ preserves) and therefore $x$ belongs to a null set.

\subsection{Entropy contribution and invariance}\label{subsec:ContAndInva}
We stated above that the entropy $h_\mu(a)$ is related to a volume decay property for the stable horospherical subgroup. We want now to use the notion of leafwise measures to define this more precisely, and use the same definition to quantify the entropy contribution  of ($a$-normalized) subgroups of the stable horospherical subgroup of $a$.

Consider the setting of Proposition \ref{prop:LFMeasures} with the additional assumptions in the \ref{item: equivary}-th item. We denote by $\theta_a\colon G\to G$ the conjugation map by $a$ and recall that  $B_1^H$ is the ball of radius one in $H$ around $e$. The \emph{volume decay entropy at $x$} is defined by 
\begin{equation}\label{eq:VolDef}
    \vol_\mu(a,H,x)=-\lim_{n\ra \infty} \frac{\log \mu_x^H\pa{\theta^n_a\pa{B_1^H}}}{n}. 
\end{equation}
This limit exists almost everywhere and should be thought of as a pointwise entropy contribution: we define its integral
\begin{equation}\label{eq:EntContDef}
    h_\mu(a,H)=\int_X\vol_\mu(a,H,x)\mathrm{d} \mu 
\end{equation}
to be the \emph{entropy contribution of $H$}.

\begin{prop}[Entropy as volume decay] \label{prop:EntAsVolume} Let $X=\Gamma\backslash G$ be an $S$-adic homogeneous space with a probability measure $\mu$ which is invariant under a diagonalizable element $a\in G$. Then $h_\mu(a)=h_\mu(a,G_a^-)$.
\end{prop}

The volume decay with respect to the Haar measure $m_X$ is a local quantity that can be directly calculated using the adjoint map on the Lie algebra. For example, 
$$ \vol_{m_X}(a,G_a^-,x)=-\log\av{\det \Ad_a|_{\Lie(G_a^-)}},\quad \text{ for every } x\in X.
$$

There are many measure-preserving systems where the uniform (e.g.~Haar) measure is the unique measure of maximal entropy. For many $S$-arithmetic quotients this fact can be deduced from the following theorem:
\begin{theo} \label{thm:EntropyDetEW}
With the assumptions as in Proposition \ref{prop:EntAsVolume}, let $U<G_a^-$ be an $a$-normalized subgroup (e.g.~a coarse Lyapunov subgroup). We have
\begin{equation}\label{eq:boundEntCont}
h_\mu(a,U)\leq -\log\av{\det \Ad_a|_{\Lie(U)}}=h_{m_X}(a,U)    
\end{equation}
with equality if and only if $\mu$ is $U$-invariant.
\end{theo}
In other words, $\mu$ is $U$-invariant if and only if the entropy contribution of $U$ with respect to $\mu$ agrees with the entropy contribution with respect to the Haar measure. One should keep this in mind together with the previous characterization we had for invariance in Proposition \ref{prop:LFMeasures}\eqref{item:Invariance}.

This theorem together with the decomposition of $\Lie(G_a^-)$ in \eqref{eq:StableHoroDecompostion} and the ability to define the entropy contribution for any coarse Lyapunov subgroup raises the question of understanding the relations between $\mu_x^{G_a^-}$ and $\mu_x^{G^{[\lambda]}}$ for the various coarse Lyapunov subgroups $G^{[\lambda]}< G_a^-$. \textcite{EinsiedlerKatok} show that this relation is quite simple, as we explain in the next subsection. We remark that Proposition \ref{prop:EntAsVolume} and the results of the next subsection are variations of preceding results of \textcite{LedYoungI,LedYoungII} for $C^2$-diffeomorphism of compact manifolds.

\subsection{The product structure}\label{subsec:ProductStr}
The following was proven by \textcite[Theorem 8.4]{EinsiedlerKatok}: let $U=G_a^-$ for some diagonalizable element $a$ that preserves a probability measure $\mu$ on an $S$-arithmetic homogeneous space $X$. Let $G^{[\lambda_1]},\dots,G^{[\lambda_k]}$ denote all coarse Lyapunov subgroups and for a weight $\lambda$, and $\mu_x^{[\lambda]}$ denote the leafwise measure $\mu_x^{G^{[\lambda]}}$. Then, the map 
$$\iota\colon \prod_{i=1}^k G^{[\lambda_i]}\to U,\,\iota(u_1,\dots,u_k)=u_1\dots u_k
$$
is a homeomorphism and for almost all $x\in X$ we have the so-called “product structure”:
\begin{equation}\label{eq:ProdStructure}
\mu_x^U\propto \iota_*\pa{\mu_x^{[\lambda_1]}\otimes\dots\otimes \mu_x^{[\lambda_k]}}.    
\end{equation}
This readily implies the following addition property for entropy contribution:
\begin{equation}\label{eq:entContAddtion}
h_\mu(a,G_a^-)=\sum_{i=1}^k h_\mu(a,G^{[\lambda_i]}).
\end{equation}

\subsection{Abramov--Rokhlin entropy addition formula}\label{Subsec:AbrahmovRokhlin}
The joinings we consider (as in Theorem \ref{thm:mainThm})  have natural factors, both of which are equipped with the corresponding Haar measure. This should imply that we have $h_a(\mu)>0$ for any non-trivial $a\in A$. The Abramov--Rokhlin enables to show this and to calculate the entropy in a step-wise manner through the factors. To state the formula, we shortly recall the definition of conditional entropy.

Given a Borel probability space $(X,\cB,\mu)$ and two $\sigma$-algebras $\cA,\cC\subset \cB$ with $\cC$ countably generated, one defines the \emph{information function on $\cC$ given $\cA$} as 
$$
I_\mu(\cC|\cA)(x)=-\log\mu_x^\cA\pa{[x]_{\cC}},
$$
and \emph{the conditional entropy of $\cC$ given $\cA$}, $H_\mu(\cC|\cA)$, as the integral of the information function 
$$
H_\mu(\cC|\cA)=\int_X I_\mu(\cC|\cA)(x)\mathrm{d}\mu.
$$
One should think about this conditional entropy as measuring how much new information $\cC$ has, given the information that we already have from $\cA$. To define the conditional entropy of a measure-preserving transformation $T\colon X\ra X$, one defines
$$
h_\mu(T|\cA)=\sup_{\eta:H_\mu(\eta)<\infty} h_\mu(T,\eta|\cA)
$$
with 
$$
h_\mu(T,\eta|\cA)=\lim_{n\to\infty} \frac{1}{n}H_\mu(\eta_0^{n-1}|\cA).
$$

With these definitions, using the setting and the notation from  subsection \ref{subsec:settings}, the Abramov--Rokhlin entropy addition formula states for a diagonalizable element $a=(a_1,a_2)$ that 
\begin{equation}\label{eq:usualAbrahmovRohk}
    h_\mu(a)=h_{m_{X_1}}(a_1)+h_\mu \pa{a|\pi_1^{-1}\cB_{X_1}},
\end{equation}
where $\pi_1:X\to X_1$ is the natural projection map and $\cB_{X_1}$ is the Borel $\sigma$-algebra on $X_1$. By symmetry of the assumptions, the same holds with $1$ and $2$ interchanged.

\begin{rema}\label{rem:positiveEnt}
It follows directly from \eqref{eq:usualAbrahmovRohk} that any non-trivial $a\in A$ has positive entropy.
\end{rema}

\subsection{An entropy inequality}
One of the main ideas of \textcite{EL_Joinings2007} is that Abramov--Rokhlin formula \eqref{eq:usualAbrahmovRohk} may be coupled with 
\eqref{eq:entContAddtion} to get an “Abramov--Rokhlin type formula/inequality for the entropy contribution”.  In other words, a main step of the proof of the joining classification is the following inequality, which shows that the entropy contribution with respect to a joining can be bounded in a stepwise manner using the natural factors:
\begin{lemm}\label{lem:Inequality}
With the setting of \S \ref{subsec:settings} let $a=(a_1,a_2)$ be a diagonalizable element with $a_i,\,i=1,2$ acting ergodically on $X_i=\Gamma_i\backslash G_i$ equipped with the Haar measure $m_{X_i}$,  let $U_1<(G_1)_{a_1}^-$ and $U_2<(G_2)_{a_2}^-$ be two subgroups, and set $U=U_1\times U_2<G$. Assume that $U$ is $a$-normalized and let $\mu$ be  a joining on $X$ (that is, an $a$-invariant measure with $\pi_i(\mu)=m_{X_i},\,i=1,2$). Then, we have 
\begin{equation}\label{eq:EntInEq}
h_\mu(a,U)\leq h_{m_{X_1}}(a_1,U_1)+h_\mu(a,\set{e}\times U_2).    
\end{equation}
Moreover, for the stable horospherical subgroup $U=G_a^-$ (that is, for $U_i=(G_i)_{a_i}^-,\,i=1,2$) we have equality in \eqref{eq:EntInEq}.
\end{lemm}
The same results of course hold with $1$ and $2$ interchanged, by the symmetry of the assumptions of this lemma.
The proof of this lemma uses tailor-made partitions which are subordinate to the stable horospherical subgroups $(G_i)_{a_i}^-$ and to the subgroups $U_i$ (see \cite[Prop. 3.1]{EL_Joinings2007}).

\subsection{Equality for coarse Lyapunov weights}\label{Subsec:AbrohmovRokhlinForLya}
Lemma~\ref{lem:Inequality} together with the product structure \eqref{eq:ProdStructure} implies that  we also have equality in \eqref{eq:EntInEq} for  a coarse Lyapunov subgroup $U$:

\begin{prop}[Abramov--Rohklin for Coarse Lyapunov subgroups]\label{Prop:ABraRokhForCoarse} With the setting of Lemma~\ref{lem:Inequality}  the following equality holds for any coarse Lyapunov subgroup $G^{[\lambda]}=G_1^{[\lambda]}\times G_2^{[\lambda]}$ 
\begin{equation}\label{eq:AbRohForEntCont}
    h_\mu(a,G^{[\lambda]})=h_{m_{X_1}}(a_1,G_1^{[\lambda]})+h_\mu \pa{a,\set{e}\times G_2^{[\lambda]}}.
\end{equation}  By symmetry, the same statement holds with $1$ and $2$ interchanged.
\end{prop}
\begin{proof}
  This is just  simple arithmetic, as the quantities in \eqref{eq:AbRohForEntCont} are additive with respect to a product of measures, and since we have equality for the stable horospherical and an inequality for each coarse Lyapunov weight. More precisely, let $G^{[\lambda]}$ be a Lyapunov weight and choose $a\in G$ such that $G^{[\lambda]}<G_a^-$. Let $G_a^-=G^{[\lambda_1]}\dots G^{[\lambda_k]}$ be the corresponding decomposition of the stable horospherical to coarse Lyapunov subgroups (with, say, $\lambda=\lambda_1$). For each $1\leq \ell\leq k$ we have by Lemma~\ref{lem:Inequality}  
  $$
  h_\mu(a,G^{[\lambda_\ell]})\leq h_{m_{X_1}}(a_1,G_1^{[\lambda_\ell]})+h_\mu \pa{a,\set{e}\times G_2^{[\lambda_\ell]}}.
  $$
  By the product structure \eqref{eq:ProdStructure} (applied to $\mu$ and also to $m_{X_1}$) and the definition of entropy contribution the sum of each term over $\ell=1,\dots k$, sums up to the corresponding term for $G_a^-$. It follows that
  \begin{align}\label{summingUpBegin}
  h_\mu(a,G_a^-) &= \sum_{\ell=1}^k h_\mu(a,G^{[\lambda_\ell]})\\
  &\leq \sum_{\ell=1}^k h_{m_{X_1}}(a_1,G_1^{[\lambda_\ell]})+ \sum_{\ell=1}^k h_\mu \pa{a,\set{e}\times G_2^{[\lambda_\ell]}} \\ &= h_{m_{X_1}}(a_1,(G_1)_{a_1}^-)+h_\mu \pa{a,\set{e}\times (G_2)_{a_2}^{-}}.\label{summingUpEND}
\end{align}
But we know from Lemma~\ref{lem:Inequality} that the left-hand side of \eqref{summingUpBegin} is equal to the right-hand side of \eqref{summingUpEND}, so all the above inequalities must be equalities.
\end{proof}

\section{Two corollaries}\label{sec:twoCoro}
We can prove now two corollaries of the equality \eqref{eq:AbRohForEntCont}. This shows that this equality already contains substantial content -- that is the reason we called \eqref{eq:AbRohForEntCont} and the following corollaries the “second ingredient” in the overview in \S \ref{sec:birdsAndIngredients}.

The first corollary is in particular interesting and can be referred to as “the two ingredients joinings theorem”: it allows for a complete classification of joinings in some situations, e.g., for Example \ref{exa:SL2Squared}\eqref{exa:SL245Rotation}). Also, it (or more precisely, Remark \ref{rema:NTtimesTrivial}) will be used again and again in various stages of the proof of Theorem \ref{Thm:SL3Classification} and in \S \ref{sec:Rank2}. Moreover, this corollary and its proof explain what we are looking for (utilizing entropy methods to gain an unipotent invariance) and also what we cannot hope for (finding unipotent invariance which is trivial in one of the factors). We give more details below.
\subsection{Disjointness due to different Lyapunov weights}\label{subsec:DifferentWeights}
It may happen that the root structure of $G_1$ and $G_2$ are different enough, so that for some Lyapunov weight $\lambda$, one of the coarse Lyapunov subgroups  $G_1^{[\lambda]}$ or  $G_2^{[\lambda]}$ is trivial. This immediately  “decouples” the joining:
\begin{coro}[Two ingredients joinings theorem]\label{cor:differntRootDisj}
With the setting of \S \ref{subsec:settings} let $\mu$ be a joining and assume that there is a Lyapunov weight $\lambda$ such that (say) $G^{[\lambda]}=G_1^{[\lambda]}\times \set{e}$. Then $\mu$ is the trivial joining.
\end{coro}

\begin{proof}
   For the coarse Lyapunov weight $[\lambda]$ we have $G_2^{[\lambda]}=\set{e}$ and since the entropy contribution of the trivial group is zero, \eqref{eq:AbRohForEntCont} reads as follows
  \begin{equation}
  h_\mu(a,G^{[\lambda]})=h_{m_{X_1}}(a_1,G_1^{[\lambda]})+h_\mu \pa{a,\set{e}\times G_2^{[\lambda]}}=h_{m_{X_1}}(a_1,G_1^{[\lambda]}).    
  \end{equation}
 But 
 \begin{align*}
 h_{m_{X_1}}\pa{a_1,G_1^{[\lambda]}}&= -\log\av{\det \Ad_{a_1}|_{\Lie\pa{G_1^{[\lambda]}}}}\\
 &=-\log\av{\det \Ad_a|_{\Lie\pa{G_1^{[\lambda])}\times \set{e}}}}&=h_{m_{X}}\pa{a,G_1^{[\lambda]}\times \set{e})},    
 \end{align*}
 so Theorem \ref{thm:EntropyDetEW} then implies that $\mu$ is invariant under
 the unipotent group $G_1^{[\lambda]}\times \set{e}$. Recall that an unbounded
 closed subgroup such as $G_1$  acts ergodically on $X_1$ (see e.g.~\cite[Theorem
 2.2.6]{ZimmerBook}). Therefore, we can find an element of the form $(u,e)$
 with $u$ acting ergodically on $X_1$ that preserves $\mu$.  This immediately
 implies disjointness, that is, $\mu$ must be the trivial joining (see e.g.\ \cite[Lemma~7.2]{AkaMussoWieser}).
\end{proof}
\begin{rema}\label{rema:NTtimesTrivial}
We record the following useful fact from the above proof: having invariance under an element of the form $(u,e)$ with $\langle u\rangle$ unbounded, immediately implies disjointness.  
\end{rema}

\begin{exem}\label{ex:Exp45RotationDisj} Consider Example \ref{exa:SL2Squared}\eqref{exa:SL245Rotation}, that is, the “$45^\circ$-rotation” torus action given by $(\phi,\psi)$. As we explained there, with this torus action, all the eight coarse Lyapunov weights have the form for which Corollary \ref{cor:differntRootDisj} is applicable. Therefore, in this setting, the only possible joining is the trivial joining. We remark that a $p$-adic variant of exactly this situation arises in the arithmetic application considered in \textcite{AEW21}. See Example \ref{exa:2in4} for more details. It is interesting to remark that the statement of Corollary \ref{cor:differntRootDisj} is not explicitly visible in the statement of Theorem \ref{thm:mainThm}; it is nevertheless there: any non-trivial joining will be related to a local isomorphism between $\bG_1$ and $\bG_2$. Existence of such an isomorphism implies that the non-trivial weights for $G_1$ and $G_2$ must be identical, and therefore rules out the possibility of finding a coarse Lyapunov subgroup for which Corollary \ref{cor:differntRootDisj} is applicable.
\end{exem}



The proof of Corollary \ref{cor:differntRootDisj} hints on how we aim to proceed in general: we aim to utilize entropy considerations to establish that a certain coarse Lyapunov subgroup (which is automatically an unipotent subgroup) has maximal entropy contribution and therefore leaves the joining invariant. Once we establish an unipotent invariance, we can utilize measure rigidity results for unipotent flows as shortly described in \S \ref{subsec:settings}. But in general, this unipotent need not be trivial in one of the factors.  here are two  examples that the reader should keep in mind: 
\begin{exem}\label{ex:DiagoanlJoinings}
When we consider possible joinings for with $G_1=G_2=G$ and $\Gamma_1=\Gamma_2=\Gamma$, and with identical action, we always have the so-called “diagonal-joining”: the push-forward of $m_{\Gamma\backslash G}$ under $\iota\colon G\to  G\times G,\,\iota(g)=(g,g)$. This fits with the fact that Corollary \ref{cor:differntRootDisj} is not applicable, as the weights for $G_1$ and $G_2$ are identical so all the (coarse) Lyapunov subgroups a product of \emph{non-trivial} weight spaces.  More generally, if $\Psi\colon\bG_1\to\bG_2$ is a $\bQ$-algebraic isomorphism, then the push-forward $m_{X_1}$ under $\iota\colon G_1\to G_1\times G_2,\, g\mapsto (g,\Psi(g))$ is also a non-trivial joining which is supported on the graph of the isomorphism $\Psi$. In the case where $G_1$ and $G_2$ are simple $\bQ$-groups, Theorem \ref{thm:mainThm} implies that such joinings are the only possible non-trivial joinings. For a concrete example, see Example \ref{exa:SL3} with $\Gamma_1=\Gamma_2$. We will prove a classification of all (ergodic) joinings for  Example \ref{exa:SL3} in  \S \ref{sec:highRank}.  
\end{exem}

 On the other hand, it may well happen that all the coarse Lyapunov subgroups are product of non-trivial weight spaces, but nevertheless the only possible joining is the trivial one:  

\begin{exem}
  Consider Example \ref{exa:SL2Squared}\eqref{exa:SL2DifferentSpeeds}, i.e.~the “different speeds” action $(\phi,\tau)$. Here, there are also only four weights and all the \emph{coarse} Lyapunov subgroups are a product of non-trivial weight spaces. Nevertheless, it follows from Theorem \ref{thm:mainThm} that ergodic $(\phi,\tau)$-invariant joinings must be trivial: any non-trivial joining must arise from an algebraic automorphism $\Psi$ of $\bG_1=\SL_2\times \SL_2$ such that the image $\bZ^2$ under $(\phi,\tau)$ belongs to $\set{\pa{g,\Psi(g):g\in G}}$. But such an automorphism does not exist. It is very interesting to see how this arises in the proof, especially when one compares it to Example \ref{exa:SL2Squared}\eqref{exa:SL2IdenticalWeights}: we  discuss this in \S \ref{subsec:ClassInEx1} and \S \ref{subsec:Ex2Classi}.  
\end{exem}

\subsection{Support of the leafwise measures}
Our ultimate goal is to show that a joining $\mu$ admits some unipotent invariance, for example, under one of the coarse Lyapunov subgroups $G^{[\Lambda]}$, or a non-trivial subgroup of it. By Theorem \ref{thm:EntropyDetEW} this is equivalent to showing that $G^{[\Lambda]}$ (or a non-trivial subgroup of it) has maximal entropy contribution, and by Proposition \ref{prop:LFMeasures}\eqref{item:Invariance} this is also equivalent to showing that $\mu_x^{[\Lambda]}$ equals the Haar measure $m_{G^{[\Lambda]}}$ almost everywhere (or the Haar measure on a non-trivial subgroup of it). In this subsection, we show that a much weaker statement follows from \eqref{eq:AbRohForEntCont}, namely that the support of $\mu_x^{[\Lambda]}$ is large (at least once projected, see below) almost everywhere. In the next sections, we will try to use further methods/ingredients (the low and high entropy methods) to upgrade this partial information about the support to $\mu_x^{[\Lambda]}$ to a statement that will imply invariance under $U^{[\Lambda]}$ (or one of its non-trivial subgroups).

In general, we cannot expect $\mu_x^{[\Lambda]}$ to have full support. For instance, in  the “identical weights” Example  \ref{exa:SL2Squared}\eqref{exa:SL2IdenticalWeights}, if $\mu$ is the diagonal joining induced by $\iota(g)=(g,g)$, then  $\mu_x^{[\Lambda]}$ is supported on
$$
\set{\iota(u):u\in G_1^{[\Lambda]}}\subset G_1^{[\Lambda]}\times G_2^{[\Lambda]}.
$$
But in general, we might hope (since having a joining should imply this fact) that the support will be large once projected to each of the factors. This is indeed the case: recall that by Proposition \ref{Prop:ABraRokhForCoarse} we have equality in \eqref{eq:EntInEq} for any coarse Lyapunov subgroup. This has the following consequence. 

\begin{coro}\label{cor:equalityImpOnto}
In the setting of Lemma~\ref{lem:Inequality} assume that equality holds in \eqref{eq:EntInEq} for $U=U_1\times U_2<G$. Then, $U_1$ is the smallest connected $a_1$-normalized subgroup of $G_1$ containing $\pi_1\pa{\supp{\mu_x^U}}$ for almost every $x\in X$.
\end{coro}
\begin{proof}
Assume for contradiction that the smallest connected $a_1$-normalized subgroup of $G_1$ containing $\pi_1\pa{\supp{\mu_x^U}}$ for almost every $x\in X$ is $V_1\lneq U_1$. We claim that the following strict inequality holds:
\begin{align*}
h_\mu\pa{a,{U_1\times U_2}}& 
= h_\mu\pa{a,{V_1\times U_2}} \\
&\leq h_{m_{X_1}}\pa{a_1,V_1}+h_\mu(a,\set{e}\times U_2)\\
&<  h_{m_{X_1}}\pa{a_1,U_1}+h_\mu(a,\set{e}\times U_2).
\end{align*}
Indeed, the first equality follows from the definition of the entropy contribution \eqref{eq:VolDef} and  \eqref{eq:EntContDef} (strictly speaking, one needs to verify that for $U'<U$ we have $\pa{\mu_x^U}|_{U'}=\mu_x^{U'}$, i.e.~a leafwise measure variant of  the double conditioning formula for conditional measures (see e.g.\ \cite[Prop.~5.20]{EWBOOK}). Now, the first inequality follows from \eqref{eq:EntInEq} applied to $\pa{V_1\times\set{e}} \pa{\set{e}\times U_2}$. For the last strict inequality, note first that $V_1\lneq U_1<\pa{G_1}_{a_1}^-$, so all the directions  in $G_{a_1}^{-}$ are being contracted at least by some factor $<1$. Therefore,  we have the following strict inequality for the contribution with respect to the Haar measures (see~\eqref{eq:boundEntCont}):
$$
h_{m_{X_1}}(a_1,V_1)=-\log\av{\det \Ad_{a_1}|_{\Lie(V_1)}}< -\log\av{\det \Ad_{a_1}|_{\Lie(U_1)}}= h_{m_{X_1}}(a_1,U_1).
$$
This implies the strict inequality above. But having a strict inequality is a contradiction to our assumption that \eqref{eq:EntInEq} holds for $U$.
\end{proof}

\begin{coro}\label{cor:SupportProjectsOnto}
  In the setting of \S \ref{subsec:settings} let $\mu$ be an ergodic $A=\phi(\bZ^d)$-invariant joining, and let $U=U_1\times U_2$ be a coarse Lyapunov weight. Let $P_x$ be the smallest connected $A$-normalized subgroup of $U$ containing $\supp{\mu_x^U}$. Then, there is set $X'\subset X$ with $\mu(X)=\mu(X')$ and  a subgroup $P<U$ with $\pi_i(P)=U_i$ for $i=1,2$ such that $P=P_x$ for every $x\in X'$.
\end{coro}
\begin{proof}
  According to Corollary \ref{cor:equalityImpOnto} we have $\pi_i(P_x)=U_i$ for $i=1,2$ for almost all $x\in X$, so we just have to show that $P_x$ is almost everywhere constant. To this end, first we note that $x\mapsto P_x$ is a measurable map (which is not at all obvious: one checks this using the characterizing property in Prop \ref{prop:LFMeasures}, where the set of (connected) subgroups of~$U$ is equipped with the topology on the Grassmannian of $\Lie(U)$). Moreover, since $P_x$ is $A$-normalized, it follows from \eqref{eq:equivariancy} that for every $a\in A$ and almost every $x$ we have $P_{a\cdot x}=P_{x}$. By the ergodicity assumption, the $A$-invariant measurable map $x\mapsto P_x$ must be then almost everywhere constant, as we wanted to show. 
\end{proof}

The last corollary could be loosely interpreted, by saying that the leafwise measures of coarse Lyapunov weights of a joining behave like the joining: they project onto each of  their natural factors. In particular, we see that leafwise measures of coarse Lyapunov weights have “interesting” support. In the next two sections (depending on the context, and on if different coarse Lyapunov weights commutes or not) we will see how one could use the non-trivial support of these measures to generate an unipotent invariance.

\section{the high rank case}\label{sec:highRank}

When we have two non-commuting coarse Lyapunov weights, the following theorem \parencite{EinsiedlerKatok}, coupled with the input we get from Corollary \ref{cor:SupportProjectsOnto}, finds an unipotent invariance for us.

\begin{theo}[The high entropy method]\label{thm:highEntropy} In the setting of Section \ref{subsec:settings}, let $\mu$ be an  $A$-invariant ergodic probability measure
and fix $a\in A$. Then, there exist two connected and $A$-normalized subgroups $H\unlhd P\leq G_a^-$ such that the following holds
\begin{enumerate}
    \item For almost every $x\in X$, $P$ is the minimal $A$-normalized subgroup supporting $\mu_x^{G_a^-}$.
    \item For almost every $x\in X$, $\mu_x^{G_a^-}$ is left- and right-invariant under $H$.
    \item If $[\lambda]\neq [\eta]$ are two different coarse Lyapunov weights with $G^{[\lambda]},G^{[\eta]}\leq G_a^-$ and $g_1\in P\cap G^{[\lambda]},\, g_2\in P\cap G^{[\eta]}$, then $[g_1,g_2]\in H$, that is, $g_1H$ and $g_2H$ commute with each other in $P/H$. 
    \item \label{item:HighEnt_Invariance}For every coarse Lyapunov weight $[\lambda]$ with $G^{[\lambda]}<G_a^-$ and for almost every $x\in X$, $\mu_x^{G^{[\lambda]}}$ is left- and right-invariant under $H\cap G^{[\lambda]}$.
\end{enumerate}
\end{theo}



We referred to this theorem as the “third ingredient” in \S \ref{sec:birdsAndIngredients}. We will explain now how the three ingredients we named there allow us to classify joinings when $\bG_1$ and $\bG_2$ are semisimple without rank one factors. As the main steps of the proof are already visible when one considers the case $\bG_1=\bG_2=\SL_3$, we will restrict ourselves to this case for simplicity and prove:

\begin{theo}\label{Thm:SL3Classification} In the setting of Example \ref{exa:SL3} all ergodic joinings are algebraic. More precisely, the only possible ergodic joinings are the trivial one and joinings which are the Haar measure on an orbit of a diagonal embedding of $\SL_3$ into $\SL_3^2$, that is, an orbit of a group of the form $\set{(g,\Psi(g)):g\in \SL_3(\bR)}< \SL_3^2(\bR)$, for an automorphism $\Psi\colon\SL_3\to \SL_3$.
\end{theo}

\subsection{Proof of Theorem \ref{Thm:SL3Classification}}
Recall that $G_i=\bG_i(\bR),\, i=1,2$ which is $\SL_3(\bR)$ in our case and  that $G=G_1\times G_2$. Recall that $\mu$ denotes the joining we are trying to identify, and we begin by considering the subgroup $I<G$ which is generated by all one-parameter unipotent subgroups that leave the measure $\mu$ invariant. If we knew that $\mu$ was algebraic (in the sense of Theorem \ref{Thm:SL3Classification}), $I$ would have projected onto each factor. Reversing the order, we begin by establishing the latter as a first step towards the algebraicity of $\mu$:
\begin{lemm}\label{lem:goi projects}
  Let $\goi$ be the Lie algebra of $I$. Then, for $i=1,2$ we have $\pi_i(\goi)=\gog_i$ (as above $\gog_i=\mathfrak{sl}_3$ denotes the Lie algebra of $G_i$).
\end{lemm}
\begin{proof}
As we explain in Example \ref{exa:SL3}, we have six non-trivial coarse Lyapunov groups $\gog^{ij}=\gou_{ij}\times \gou_{ij}$. Consider, for example, $\gog^{13}$. We would like to write $\gog^{13}$ as a commutator of two other coarse Lyapunov groups, which appear together with $\gog^{13}$ in some stable horospherical subgroup. More precisely, recall that $A=\phi(\bZ^2)$ is a subgroup of the diagonal group and pick $a\in A$ with $\gog_a^-=\Lie(G_a^-)=\gog^{12}\oplus\gog^{13}\oplus\gog^{23}$ (the Lie algebra of the Heisenberg group). For example, one can pick $a=\phi(1,2)$ with the notation of Example~\ref{exa:SL3}. Note that $[\gog^{12},\gog^{23}]=\gog^{13}$. We apply the high entropy Theorem (Theorem~\ref{thm:highEntropy}) with these choices and let~$P$ and~$H$ the subgroups appearing in that theorem. We know that $H\cap G^{13}$ contains all commutators $[g_1,g_2]$ with  $g_1\in P\cap G^{12}$ and $g_2\in P\cap G^{23}$.  Moreover, we know from Corollary~\ref{cor:SupportProjectsOnto} that $P\cap G^{12}$ must project onto $G_i^{12},\,i=1,2$, using the notation $G^{ij}=G_1^{ij}\times G_2^{ij}$. The same holds for the projection of $P\cap G^{23}$ on the corresponding two factors. As $[G_i^{12},G_i^{23}]=G_i^{13}$ we see that $H$~must project onto $G_i^{13}$ for $i=1,2$. Going back to the Lie algebra, this means, that for $i=1,2$ and $w\in \gog_i^{13}=\Lie(G_i^{13})=\gou_{13}$, $\goi$ must contain an element whose $i$-component is~$w$. 

As we explain in Example \ref{exa:SL3},  each coarse Lyapunov space can be written as commutators of other Lyapunov spaces appearing in the decomposition of some stable horospherical subgroup. Therefore, the above argument can be applied to each of the coarse Lyapunov weight spaces. As $\gog_i=\gosl_3,\,i=1,2$ is generated by $\gou_{ij},\,i\neq j$, the lemma follows.
\end{proof}

Loosely speaking, we already see the possible joinings only by looking at the possibilities for $H\cap G^{ij}$: it is either two-dimensional and equals $G^{ij}$ itself, or it is one-dimensional and equal to a graph of an isomorphism from $U^{ij}$ to itself. We note just for fun, that if $H\cap G^{ij}$ is two-dimensional for some $i\neq j$, $\mu$ will be invariant under an element of the form $(u,e)$ or $(e,u)$ with $u\in G_i$ acting ergodically on the corresponding quotient $G_i$; in other words, we will be again in the situation of the proof of Corollary \ref{cor:differntRootDisj} which shows that in this case, $\mu$ must be the trivial joining. For the  general case, we will need  the following Lemma.
\begin{lemm}\label{lem:LieIdeal}
  The Lie Algebra $\goi\cap \pa{\gog_1\times\set{0}}$ is a Lie ideal in $\gog=\gog_1\times\gog_2$ (and therefore also in  $\gog_1\times\set{0}\cong \gog_1$). The same holds for $\goi\cap \pa{\set{0}\times\gog_2}$.
\end{lemm}
\begin{proof}
  Consider $(w,0)\in \goi\cap \pa{\gog_1\times\set{0}}$ and $(v_1,v_2)\in \gog$. Using Lemma~\ref{lem:goi projects} pick an element in $\goi$ having the same first component as $(v_1,v_2)$, say $(v_1,v_2')\in \goi$. Then, 
  $$
  \gog_1\times\set{0}\ni ([w,v_1],0)=[(w,0),(v_1,v_2)]=[(w,0),(v_1,v_2')]\in \goi,
  $$ 
  as we wanted to show.
\end{proof}

\begin{proof}[Proof of Theorem \ref{Thm:SL3Classification}]
In our case $\gog_1=\gog_2=\gosl_3$ is a simple Lie algebra, so the Lie ideals in Lemma~\ref{lem:LieIdeal} are either trivial or equal to $\gosl_3$. If  one of them is equal to $\gosl_3$,  we can find a non-trivial element of the form $(u,e)$ or $(e,u)$ with $u$ unipotent, as in the proof of Corollary \ref{cor:differntRootDisj}, which shows that $\mu$ must be the trivial joining. So we may assume now that both ideals are trivial. By the folklore Goursat's Lemma,
$\goi$ must be the graph of an isomorphism, that is, 
$$
\goi=\set{(w,\Phi(w):w\in \gog_1}
$$
for an \emph{isomorphism} $\Phi\colon\gog_1\to\gog_2$. It follows that
$$
I=\set{(g,\Psi(g):g\in G_1}
$$
for an \emph{isomorphism} $\Psi\colon G_1\to G_2$.
This isomorphism must intertwine the diagonal action $\phi=(\phi_1,\phi_2)$, that is,   $\phi(s,t):=(\phi_1(s,t),\phi_2(s,t))$ must agree with $\pa{\phi_1(s,t),\Psi(\phi_1(s,t))}$. Indeed, if not, $\mu$ would be invariant under two elements having the same first coordinate but a different second coordinate, so we could again find an element of the form $(e,a)$ with $a$ acting ergodically on $X_1$, preserving $\mu$. This puts us again in the situation of the proof of Corollary \ref{cor:differntRootDisj}. It follows that $A\subset I$, and therefore $\mu$  must be ergodic with respect to the action of $I$, since $\mu$ was ergodic with respect to the action of $A$. Ratner's measure classification can be then applied with the group $I$ showing that $\mu$ is the Haar measure on an orbit of $I$, which concludes the proof.
\end{proof}
\begin{rema}
We actually get more information about the possible joinings from the above proof. First, the automorphism $\Psi$ must intertwine the action $\phi$. Second, if  there exists a non-trivial joining, Ratner's Theorem will imply that the lattices $\Gamma_1$ and $\Gamma_2$ must be commensurable.
\end{rema}
The proof of the general theorem, say when $\bG_1$ and $\bG_2$ are both semisimple of rank $\geq 2$, is very similar. One essentially replaces the ad-hoc argument for $\gosl_3$ in the beginning of Lemma~\ref{lem:goi projects} with a general argument that works for any semisimple Lie algebra without rank one factors; see  \textcite[Lemma~4.2]{EL_Joinings2007}.

\section{Joinings with forms of $\SL_2$ }\label{sec:Rank2}
In this section we wish to classify the possible ergodic joinings in Example \ref{exa:SL2Squared}\eqref{exa:SL2IdenticalWeights},  the “identical weights” example, and Example \ref{exa:SL2Squared}\eqref{exa:SL2DifferentSpeeds}, the “different speeds” example. Note that we already used Corollary \ref{cor:differntRootDisj} to show that all the possible joinings in Example~\ref{exa:SL2Squared}\eqref{exa:SL245Rotation} are trivial.
These sub-examples of Example \ref{exa:SL2Squared} may seem very specific, but they are actually very representative examples. Conceptually, classifying joinings when $\bG_1=\bG_2=\SL_2\times\SL_2$ is the main new case, which is included in Theorem \ref{thm:mainThm}, but not included in the results of \textcite{EL_Joinings2007}. Moreover, the above three structures of the torus action cover all possible types of joined actions. Furthermore, an analogous full classification of joinings in a general $S$-adic version of Example \ref{exa:SL2Squared}, for products of $\SL_2$-forms, will suffice for most of the applications discussed in section \ref{sec:Applicaitions}.

As we explained in \S \ref{sec:birdsAndIngredients}, the new ingredient used
in \textcite{EL_Joinings2019} in comparison to \textcite{EL_Joinings2007} is
the low-entropy method. This method was originally developed by
\textcite{LindenstraussQUE}. We wish to follow the main steps of \S 7 of this
paper (which are also nicely explained in \cite[\S 10]{PisaNotes}) in order to sketch a proof of the crucial step, Proposition~\ref{prop:mainStep}, in classifying the possible joinings in the two examples listed above. In some sense, we “merely” run the same argument, that Lindentrauss ran in one factor of the form $\Gamma\backslash\pa{\SL_2(\bR)\times\SL_2(\bR)}$, in two such factors simultaneously. Our torus action is embedded “diagonally”, that is, it acts on both factors simultaneously. This gives rise to several complications (and sometimes to cumbersome notation) that must be taken into account. 

Let us recall and set up some notation: we let $\bG_1=\bG_2=\SL_2\times\SL_2$, $G_i=\bG_i(\bR),\,i=1,2$ and for simplicity let $\Gamma_1=\Gamma_2$ be an irreducible lattice in $\SL_2(\bR)\times\SL_2(\bR)$ (e.g., $\SL_2(\bZ[\sqrt{2}])$ diagonally embedded). The letter $d$ denotes a left-invariant metric on $G$, or on $G_i$, or on $\SL_2(\bR)$, depending on the context. Let $X_i=\Gamma_i\backslash G_i$ and set $\bG=\bG_1\times \bG_2$, $G=G_1\times G_2$, and $X=X_1\times X_2$. Note that $G$ (or $\bG$) have four factors, first two in $G_1$ and last two in $G_2$. We refer to them as the first, second, third, and fourth factor accordingly. 
Recall the notation of $\phi,\tau\colon\bZ^2\to G_i$ given by $\phi(t,s)=(a_t,a_s)$ and $\tau(t,s)=(a_{2t},a_{2s})$. We will concentrate on the map $\Phi_1=(\phi,\phi)\colon\bZ^2\to G$, i.e, on Example \ref{exa:SL2Squared}\eqref{exa:SL2IdenticalWeights}, remarking, where needed, what would have changed if we had considered the map $\Phi_2=(\phi,\tau)\colon\bZ^2\to G$ corresponding to Example \ref{exa:SL2Squared}\eqref{exa:SL2DifferentSpeeds}. By abuse of notation, we denote $A=\Phi_1(\bZ^2)$ or $A=\Phi_2(\bZ^2)$ in both cases. 

We recall from Example \ref{exa:SL2Squared} that for $\Phi_1$ there are four weight spaces, each giving rise to a different coarse Lyapunov subgroup, and for $\Phi_2$ there are 8 weight spaces, coupled through coarse equivalence into pairs, which give rise to the same four Lyapunov subgroups. We denote by $[\alpha]$ the coarse Lyapunov weight with  $w_\alpha=(-1,0)$ and by $[\beta]$ be the coarse Lyapunov weight with $w_\beta=(0,-1)$. With $U^+$ (resp.~ $U^-$) denoting the upper (resp. lower) triangular unipotent subgroups in $\SL_2(\bR)$, the corresponding coarse Lyapunov subgroups in both cases are 
\begin{align*}
    U^{[\alpha]}=&\pa{U^+\times\set{e}}\times\pa{U^+\times\set{e}}\\ 
    U^{[\beta]}=&\pa{\set{e}\times U^+}\times\pa{\set{e}\times U^+}\\
    U^{[-\alpha]}=&\pa{U^-\times\set{e}}\times\pa{U^-\times\set{e}}\\ 
    U^{[-\beta]}=&\pa{\set{e}\times U^-}\times\pa{\set{e}\times U^-}.
\end{align*}
We denote by $\mu$ the $A$-invariant joining on $X$ that we wish to identify. Recall that $\mu^{[\alpha]}_x,\mu^{[\beta]}_x,\mu^{[-\alpha]}_x,\mu^{[-\beta]}_x$ denote the  leafwise measures corresponding to the coarse Lyapunov subgroup as above.
For each of the coarse Lyapunov subgroups (i.e., for $\lambda\in \set{\pm\alpha,\pm\beta}$) let  
\begin{equation}\label{eq:LocalLWInvGroup}
I_{x}^{[\lambda]}=\set{u\in U^{[\lambda]}: u \text{ preserves the measure }\mu_x^{[\lambda]}}.
\end{equation}
The main goal of this section is to outline a proof of the following proposition:
\begin{prop}[Main step]\label{prop:mainStep}
Let $\mu$ be an ergodic joining on $X$, that is, let $\mu$ be an $A$-invariant ergodic measure on $X$ which projects to the Haar measure on $X_i$ in each factor. For any coarse Lyapunov weight $[\lambda]$, and for $\mu$-almost every $x$, $I_x^{[\lambda]}$ is not the trivial subgroup. 
\end{prop}
We will sketch a proof of this proposition in \S \ref{subsec:proofMainStep}.
\subsection{Classification under the assumption of Proposition \ref{prop:mainStep}}

\subsubsection{Preparation: coordinates on $\UAlpha$}\label{subsec:CoordUAlpha}
Throughout this section we will concentrate on $\lambda=\alpha$. 
First note, that via 
\begin{equation}\label{LieAlgWeight}
\Lie(U^{[\alpha]})\cong \bR^2,\,(n_x,0,n_y,0)\mapsto (x,y),\, n_x:=\begin{pmatrix}
  0 & x  \\
 0& 0\end{pmatrix}
\end{equation}
and the fact that  $\exp\colon\Lie(\UAlpha)\to\UAlpha$ is bijective, we can identify  $\Lie(\UAlpha)$ or $\UAlpha$ with $\bR^2$.
 For $u\in \UAlpha$ we define $u(s_1)$ and $u(s_2)$ as the real numbers satisfying   
\begin{equation}\label{eq:CoordForUAl}
    u=\pa{\begin{pmatrix} 1 & u(s_1)  \\ 0 & 1\end{pmatrix},e,\begin{pmatrix} 1 & u(s_2)  \\ 0 & 1\end{pmatrix},e}.
\end{equation}
So we can define an element $u\in \UAlpha$ by specifying $u(s_1)$ and $u(s_2)$ and refer to these as the coordinates of $u$. Any non-trivial $u\in \UAlpha$  is contained in a unique one-parameter unipotent subgroup $L_{u}\subset\UAlpha$, defined by the line in $\Lie(\UAlpha)$ through $\log(u)$. The $x$-axis corresponds to  $U^+\times\set{e}\times\set{e}\times\set{e}$, and the $y$-axis  to $\set{e}\times\set{e}\times U^+\times\set{e}$. By Remark \ref{rema:NTtimesTrivial}, showing invariance of a joining $\mu$ under a non-trivial element lying on an axis, immediately implies that $\mu$ is the trivial joining.

Identifying $\UAlpha$ with $\bR^2$ as above, let's study the  action of $A=\Phi_i(\bZ^2)$ on $\UAlpha$ for $i=1,2$ by conjugation. For $\Phi_1$, the element $\Phi_1(1,0)$  acts as scalar multiplication by $e^{-1}=(2.718\dots)^{-1}$ and $\Phi_1(0,1)$ acts trivially.  For $\Phi_2$, the element $\Phi_2(1,0)$ acts on $U^{[\alpha]}$ by multiplication with 
$\begin{pmatrix}
  e^{-1} & 0  \\
 0& e^{-2}\end{pmatrix}$, and $\Phi_2(0,1)$ acts trivially. 
\subsubsection{Classification in \ExampleOne}\label{subsec:ClassInEx1}
\begin{theo} \label{thm:Example1Classi} Any ergodic $\Phi_1(\bZ^2)$-invariant joining $\mu$ is either the trivial joining or a diagonal joining, that is, a joining supported on a graph of an isomorphism of $G_1$ with $G_2$ (as is Example \ref{ex:DiagoanlJoinings}).
\end{theo}
Before proving this theorem,  let's first analyze $I_{x}^{[\lambda]}$
for $\lambda\in \WeightSet$.  One can show that the map $x\mapsto I_{x}^{[\lambda]}$ is measurable, and by \eqref{eq:equivariancy} we also have that  
\begin{equation}\label{eq:AActionOnInv}
\forall a\in A,\quad I_{a.x}^{[\lambda]}= a I_{x}^{[\lambda]}a^{-1}.
\end{equation}
From these two facts we have:
\begin{lemm}\label{lem:everywhereConstant}
For any $\lambda\in\WeightSet$, the group $I_{x}^{[\lambda]}$ must be almost everywhere constant. 
\end{lemm}
\begin{proof}
For concreteness, consider $U^{[\alpha]}$. Recall that the element $\Phi_1(1,0)$ acts as scalar multiplication by $e^{-1}=(2.718\dots)^{-1}$ and $\Phi_1(0,1)$ act trivially. Using Poincaré recurrence as in \textcite[Lemma~7.3]{LindenstraussQUE}, or with an argument similar to the proof of Lemma~\ref{lem:mainDingEx2} below, one can show that for almost any $x\in X$, if $e\neq u\in I_{x}^{[\lambda]}$ then $L_u\subset I_{x}^{[\lambda]}$. Therefore, any subgroup of $U^{[\alpha]}$ is $A$-normalized. In particular, $I_x^{[\alpha]}$ is $A$-normalized $\mu$-almost everywhere. Since $A=\Phi_1(\bZ^2)$ is assumed to act ergodically on $X$ with respect to $\mu$, this means that $I_{x}^{[\alpha]}$ is almost everywhere constant. Such an argument holds for any other $\lambda\in\WeightSet$. 
\end{proof}
We denote this constant group by $I^{[\lambda]}$ and call it the invariance group for $[\lambda]$.
From Item \ref{item:Invariance} of Proposition \ref{prop:LFMeasures}, we have:
\begin{coro}
For any $\lambda\in\WeightSet$, $\mu$ is invariant under $I^{[\lambda]}$.
\end{coro}
This gives us enough information to analyse the ergodic joinings in  \ExampleOne.

\begin{proof}[Proof of Theorem \ref{thm:Example1Classi}]
 First observe, that each of the coarse Lyapunov subgroups is two-dimensional. By Proposition \ref{prop:mainStep} each of  the invariance groups is non-trivial. If one of them is two-dimensional, we have found an element for which Remark \ref{rema:NTtimesTrivial} applies, so $\mu$ must be the trivial joining. 
 
As we explained above, we know that if $e\neq u\in \LocalInvGp{}{\alpha}$, $L_u\subset \LocalInvGp{}{\alpha}$. Since any one-parameter subgroup of $U^{[\alpha]}$ is $\Phi(1,0)$-normalized (equivalently $A$-normalized), it might as well be that $I^{[\alpha]}$, and similarly all the other invariance subgroups $I^{[\lambda]}$, are one-dimensional. If one of them is supported on the axes, we will know that $\mu$ is trivial, again by applying Remark \ref{rema:NTtimesTrivial}. When none of them are, then each of them projects onto both of its factors, so the group generated by them must project onto $\SL_2(\bR)$ in each of the four factors (as $\SL_2(\bR)$ is generated $U^+$ and $U^-$). We implicitly apply Goursat's Lemma:  these invariance subgroups might be so compatible with each other, so the group generated by them will be the graph of an automorphism of $G_1$ into $G_2$, giving rise to a diagonal joining. Otherwise, there will be an element for which Remark \ref{rema:NTtimesTrivial} applies, resulting in the trivial joining. So this proves (modulo Proposition \ref{prop:mainStep}) everything we wanted to know about \ExampleOne.
\end{proof}

\subsubsection{Classification in \ExampleTwo}\label{subsec:Ex2Classi}
\begin{theo}\label{thm:Example2Classi} Any $\Phi_2(\bZ^2)$-invariant joining $\mu$ must be the trivial joining.
\end{theo}
To prove this theorem, it is enough to show the following:

\begin{lemm}\label{lem:mainDingEx2}
The group $\LocalInvGp{}{\alpha}$ must contain the group corresponding to the $x$-axis in the identification \eqref{LieAlgWeight}, that is the group $U^+\times\set{e}\times\set{e}\times\set{e}$.
\end{lemm}
A similar statement (and proof) will hold for any $\lambda\in \WeightSet$; it is nonetheless enough to concentrate on $\lambda=\alpha$.
\begin{proof}[Proof of Lemma~\ref{lem:mainDingEx2}]
The main difference to \ExampleOne, is that not every subgroup of $\LocalInvGp{x}{\alpha}$ is $\Phi_2(\bZ^2)$-normalized.  Recall that  the element $\Phi_2(1,0)$ acts on $U^{[\alpha]}$ by multiplication with 
$\begin{pmatrix}
  e^{-1} & 0  \\
 0& e^{-2}\end{pmatrix}$, and $\Phi_2(0,1)$ acts trivially. 
Dynamically, we see that $\Phi_2(1,0)^n$ pushes elements not belonging to the axes, towards the $x$-axis. We claim that this, coupled with Poincaré recurrence, proves the Lemma, by an argument inspired by the work of  \textcite[p.~206-207]{EinsiedlerKatok}: let $\epsilon>0$ and use Lusin's Theorem \parencite[p.~76]{LusinREF} to find a set $K_\epsilon$ of measure $1-\epsilon$ on which $x\mapsto \LocalInvGp{x}{\alpha}$ is continuous. Denote $a=\Phi_2(1,0)$ and let $x\in K_{\epsilon}$, and $e\neq u\in \LocalInvGp{x}{\alpha}$ denote the element we get from Proposition \ref{prop:mainStep}.  Poincaré recurrence finds for us, for almost any $x\in K_{\epsilon}$, a subsequence   
\begin{equation}\label{eq:PoinSavesTheDay}
a^{n_k}x\to x \text{ with }a^{n_k}.x\in K_{\epsilon}.    
\end{equation}
It follows from \eqref{eq:AActionOnInv} that the group $\LocalInvGp{a^{n_k}.x}{\alpha}$ contains the image of  $\langle u\rangle$, under the action of $a^{n_k}$, that is $\langle a^{n_k}ua^{-n_k}\rangle$,   which is a “squashing” of $\langle u\rangle$ towards the $x$-axis (here, $\langle u\rangle$ is the group generated by $u$). By \eqref{eq:PoinSavesTheDay}, $\LocalInvGp{x}{\alpha}$ contains all elements in $\UAlpha$ which are limits of $\langle a^{n_k}ua^{-n_k}\rangle$; that is, the entire group corresponding to the $x$-axis. Taking $\epsilon\to 0$ we establish the above for $x$ in a conull set of $X$, as we wanted to show. 
\end{proof}

\begin{proof}[Proof of Theorem \ref{thm:Example2Classi} ] As in \S \ref{subsec:ClassInEx1}, we will have that $\mu$ will be invariant under $I^{[\alpha]}$. Lemma~ \ref{lem:mainDingEx2} says that $I^{[\alpha]}$ contains elements for which Remark \ref{rema:NTtimesTrivial} applies, showing that $\mu$ must be trivial.
\end{proof}

\subsection{Proof of Proposition \ref{prop:mainStep}}\label{subsec:proofMainStep}
\subsubsection{A small reformulation} 

\begin{lemm}\label{lemm:TransInvToPropo}
Proposition \ref{prop:mainStep} is equivalent to establishing that for any $\lambda\in\WeightSet$ and  almost any $x\in X$,   
\begin{equation}\label{eq:proporInsteadofInv}
\exists e\neq u\in U^{[\lambda]},\quad (R_u)_*\mu_x^{[\lambda]}\propto\mu_x^{[\lambda]}.    
\end{equation}
\end{lemm}
\begin{proof}
  This follows again from Poincaré recurrence and can be proven very similarly to \textcite[Lemma~7.3]{LindenstraussQUE} or using a similar argument to the one in the proof of Lemma~\ref{lem:mainDingEx2}.
\end{proof}
Note that if we find two points $x,y\in X$ with $\LWM{x}{\lambda}=\LWM{y}{\lambda}$ and $x=u.y$ with $e\neq u\in \UAlpha$, that is, two points with identical leafwise measure on the same $U^{[\lambda]}$-leaf,  we get from  \eqref{eq:shifty} that 
\begin{equation}\label{eq:OnTheSameLeaf}
\pa{R_{u}}_{*}\LWM{x}{\lambda}=\pa{R_{u}}_{*}\LWM{u.y}{\lambda}\propto \LWM{y}{\lambda}= \LWM{x}{\lambda}
\end{equation}
So $x$ will satisfy \eqref{eq:proporInsteadofInv} with  $e\neq u\in \UAlpha$. This led Lindenstrauss to the following idea.
\subsubsection{An optimistic idea: using Ratner's H-principle}\label{subsec:OptimisticIdea}
For concreteness, assume that $\lambda=\alpha$ and consider \ExampleOne (that is, with  $A=\Phi_1(\bZ^2)$). Using Poincaré recurrence for the action of $\Phi_1(0,1)$, which commutes with $\UAlpha$, we can find many pairs of nearby points $x,y\in X$, with $\LWM{x}{\alpha}=\LWM{y}{\alpha}$, and arbitrarily small displacement $g$,  but not necessarily belonging to the same $\UAlpha$-leaf. That is, the displacement $g$ does not necessarily belong to $\UAlpha$ (see Lemma~\ref{lem:InputHMAchine} below for a more precise statement). Then, in order to achieve a pair of points with the same leafwise measure and on the same $\UAlpha$-leaf, Lindenstrauss had the following, a priori extremely optimistic, idea to use the $H$-principle of Ratner as follows: First note, that by \eqref{eq:shifty}, shearing two points $x$ and $y$ with  $\LWM{x}{\alpha}=\LWM{y}{\alpha}$ with the \emph{same} $u\in\UAlpha$ preserves the equality of the leafwise measures. That is, for any $u\in \UAlpha$
\begin{equation}
    \LWM{x}{\alpha}=\LWM{y}{\alpha}\implies \LWM{u.x}{\alpha}=\LWM{u.y}{\alpha}.
\end{equation}
To see where this is going, think for a moment about $\SL_2(\bR)$ and $U^+$, the upper unipotent subgroup. Ratner's H-principle tells us that shearing along $U^{+}$, a pair of nearby points in general position with respect to each other, for the right amount of time, will give us two sheared points that differ, non-trivially, only in the $U^+$ direction, up to a small error, which tends to zero with the initial displacement (We will formulate  this more precisely for the case at hand in Lemma~\ref{lemm:HPrinciple}).
Doing this with pairs having displacement tending to $e$, and taking a limit of the sheared pairs, we find  a pair of points, differing non-trivially, exactly in the $U^+$-direction, i.e., belonging to the same $U^+$-leaf. This seminal observation of Ratner  essentially amounts to a careful analysis of the matrix multiplication in \eqref{eq:RatnerHPrinHeart}.

So if we pretend that the map $x\mapsto\LWM{x}{\alpha}$ is continuous, and find many nearby pairs of points $(x_n,y_n)$ in general position with $\LWM{x_n}{\alpha}=\LWM{y_n}{\alpha}$, the H-principle will give us two points with identical leafwise measure on the same $\UAlpha$-leaf.   

But there is one caveat: the measurable map $x\mapsto\LWM{x}{\alpha}$ is a priori not  continuous. We can try to use Lusin's Theorem to find a compact set $K$ of arbitrarily large measure on which $x\mapsto\LWM{x}{\alpha}$ is continuous. But then, trying to restrict the argument to $K$ stirs up many serious complications. Seeing them, the author of this survey would have turned back and given up. As we will outline below, \textcite{LindenstraussQUE} chose to face them and dealt with each of them with an astonishing mastery.
\subsubsection{Preparations for the $H$-principle}
Before we follow Lindenstrauss' footsteps to face all the difficulties that arise, let us write, in our “joined” setting, the information we have for running the H-principle of Ratner.
\begin{lemm}[Input to the H-machine]\label{lem:InputHMAchine}
Let $X'$ be a set of positive measure and consider a sequence  $\delta_n\to 0$.  We can find two sequences $\pa{x_n}_{n=1}^\infty,\,\pa{y_n}_{n=1}^\infty\subset X'$ with
\begin{enumerate}
    \item \label{item:nearbyPairs} $g_n.x_n=y_n$ and $d(g_n,e)<\delta_n$,
    \item \label{item:goodDisplacement} $g_n=(g_n^{(1)},g_n^{(2)},g_n^{(3)},g_n^{(4)})$ with both $g_n^{(1)},g_n^{(3)}\notin U^+$,
    \item \label{item:sameLFW}$\mu_{x_n}^{[\alpha]}=\mu_{y_n}^{[\alpha]}$ for all $n\in\bN$.
\end{enumerate}
\end{lemm}
\begin{proof}
This essentially follows from  Poincaré recurrence for the map $b:=\Phi_1(0,1)=(e,b_1,e,b_2)$ and irreducibility of $\Gamma_1$ and $\Gamma_2$. Indeed, we can assume first  without loss of generality that $\LWM{x}{\alpha}$ is defined for any $x\in X'$. The element $b$ preserves $\mu$ so by Poincaré recurrence we can find a sequence $n_k\to \infty$ with $b^{n_k}.x\to x$ and $b^{n_k}.x\in X'$. As $b$ commutes with $\UAlpha$, we also have by \eqref{eq:equivariancy}, that for every $x\in X'$ and $k\in \bN$,  $\LWM{x}{\alpha}=\LWM{b^{n_k}.x}{\alpha}$.    This generates for us pairs satisfying Items \ref{item:nearbyPairs} and \ref{item:sameLFW}. To show Item \ref{item:goodDisplacement}, assume for contradiction that we found a pair of nearby points $x,y$ with $g.x=y=b^r.x$ for some large $r\in \bN$ and with  $g^{(1)}\in U^+$. Note that the first component of $\UAlpha$, $U^+$, is contracted by $a:=\Phi_1(1,0)$. By Poincaré recurrence, we can find a sequence $m_k\to\infty$ with $a^{m_k}.x\to x$. Acting on  $g.x=b^r.x$ with a subsequence of $a^{m_k}$, we contract $g^{(1)}$ to the identity and find a pair of nearby points $x,y'$ with 
$$
\pa{e,g^{(2)},*,*}.x=\pa{e,b_1^r,e,b_2^r}.x
$$ with $g^{(2)}$ small and $b_1^r$ large, so $\Gamma_1$ must contain a non-trivial element of the form $(e,h)$, which is a contradiction to its irreducibility. The same argument works for $g^{(3)}$, and also for $\Phi_2$ instead of $\Phi_1$.
\end{proof}

\subsubsection{Polynomial divergence - how much  should we shear?}
Still ignoring the fact that $x\mapsto\LWM{x}{\alpha}$ is not continuous, we fix some notation and explain how much we  need to shear the pairs of nearby points we get from Lemma~\ref{lem:InputHMAchine}: for  $g_n=(g_n^{(1)},g_n^{(2)},g_n^{(3)},g_n^{(4)})$, the displacements between $x_n$ and $y_n$ from  Lemma~\ref{lem:InputHMAchine},  we write
\begin{equation}\label{eq:bothDisplacements}
    g_n^{(1)}=\begin{pmatrix} a_n & b_n  \\ c_n & d_n\end{pmatrix},\, g_n^{(3)}=\begin{pmatrix} \tilde a_n & \tilde b_n  \\ \tilde c_n & \tilde d_n\end{pmatrix}.
\end{equation}
For $\delta_n$ small enough, once we shear $x_n$ and $y_n$ with $u_n\in \UAlpha$ we get that the displacement between $u_n.x_n$ and $u_n.y_n$ is $u_ng_nu_n^{-1}$. We would like to choose $u_n\in \UAlpha$ such that $\set{u_ng_nu_n^{-1}}_{n\in\bN}$ will contain a non-trivial element of $\UAlpha$ as an accumulation point. 

\begin{lemm}\label{lemm:HPrinciple}
 Recall the notation \eqref{eq:CoordForUAl} and the setting and notation of Lemma~\ref{lem:InputHMAchine}.  Assuming $\delta_n<\delta_0$ for $\delta_0$ small enough, we have the following: for any $\rho\in (0,1)$ there are  $C>0$ and  $\pa{S_n}_{n=1}^\infty$ with $S_n\stackrel{\delta_n\to 0}{\longrightarrow}\infty$ such that for any  
 $$
 s_1^{(n)},s_2^{(n)}\in T(\rho,S_n):=[\rho S_n,S_n]\cup [ -S_n,-\rho S_n],
 $$
 there exist $\sigma_1,\sigma_2$ with at least one of them belonging to  $[\tfrac{1}{C},C]$ satisfying 
 \begin{equation}
     u_n.y_n= g_n'u_\sigma u_n.x_n,\,\text{ with } d(g_n',e)<\delta_n^{\frac12}.
 \end{equation}
where $u_n,u_\sigma\in U^{[\alpha]}$ are the elements with $u_n(s_i)=s_i^{(n)}$  and $u_\sigma(s_i)=\sigma_i$ for $i=1,2$. 
\end{lemm}
In words: for a given $\rho\in (0,1)$ we can find a “timescale” $S_n$ such that shearing $x_n$ and $y_n$ with any element of $\UAlpha$ with coordinates $(s_1,s_2)$ in the “time window” $T(\rho,S_n)$ results in two points which differ, up to a small correction $g_n'$,  by an element $u_\sigma\in \UAlpha$ belonging, at least in one of its components, to a fixed (that is, independent of $n$) compact set of non-trivial unipotent elements. 
\begin{proof}
 We choose $C=\frac{1}{\rho}$ and choose $S_n$ as a \emph{minimum} as follows:
 \begin{equation}\label{eq:choiceOfS}
     S_n=\min\pa{S_n^{(1)},S_n^{(2)}} \text{ with } S_n^{(i)}=\min\pa{\frac{C}{\av{d_i-a_i}}, \frac{\sqrt{C}}{\sqrt{\av{c_i}}}},\,i=1,2.
 \end{equation}
 The fact the lemma holds with these choices, goes back to Ratner and is based on the following calculation in $\SL_2(\bR)$: let $g=\bigl(\begin{smallmatrix}  a & b  \\ c & d\end{smallmatrix}\bigr)$ and assume that $g$ is close to the identity, that is, that $\av{a-1},\av{c},\av{b},\av{d-1}<\delta$. We think of $g$ as the displacement between two $\delta$-close points $x$ and $y$ in general position, for a small $\delta$. We calculate now the displacement $g_s$ between $u.x$ and $u.y$: 
\begin{equation} \label{eq:RatnerHPrinHeart}
  g_s:= \begin{pmatrix} 1 & s  \\ 0 & 1\end{pmatrix}
   \begin{pmatrix} a & b  \\ c & d\end{pmatrix}
 \begin{pmatrix} 1 & -s  \\ 0 & 1\end{pmatrix}= \begin{pmatrix} a+cs & b+(d-a)s-cs^2  \\ c &    d-cs\end{pmatrix}.
\end{equation}
For the given $\rho\in (0,1)$ we define $C=\tfrac{1}{\rho}$ as above and 
$S=\min\pa{\frac{C}{\av{d-a}}, \frac{\sqrt{C}}{\sqrt{\av{c}}}}$. We then have for any  $s\in T(\rho,S)$, that  $g_s\approx \bigl(\begin{smallmatrix} 1 & \sigma  \\ 0 & 1\end{smallmatrix}\bigr)$, for $\sigma\in [\tfrac{1}{C},C]$, where $\approx$ stands here for an error smaller then $\delta^{1/2}$. See also \textcite[Lemmata 7.4\&7.5]{LindenstraussQUE}. 

The above calculation  should be done now for both $g_n^{(1)}$ and $g_n^{(3)}$. As we choose $S_n$ in \eqref{eq:choiceOfS} as a minimum, we can only guarantee that one of the resulting $\sigma$'s will belong to the interval $[\tfrac{1}{C},C]$.  See also Remark \ref{rem:SimulInBothFactors} below.
\end{proof}
We refer below to $S_n$, as the “\emph{shearing timescale}” (for $x_n$ and $y_n$), and to $T(\rho,S_n)$ as the “\emph{shearing time-window}”.
\begin{rema}\label{rem:SimulInBothFactors}
We are following the footsteps of the work of \textcite{LindenstraussQUE}. But this is where we start to see a slight difference: due to the fact that in our situation we shear according to \emph{two} displacements $g_n^{(1)},g_n^{(3)}$ and the shearing group $\UAlpha$ is two-dimensional, we are forced to choose the shearing timescale $S_n$ as a minimum between $S_n^{(1)}$ and $S_n^{(2)}$. Note that $S_n^{(1)}$ (resp.~$S_n^{(2)}$) is the time needed for the displacement $g_n^{(1)}$ (resp.~$g_n^{(3)}$) to grow in the $U^+$-direction. It might as well be (and we cannot control this at all since the displacements are given to us from Poincaré recurrence) that the timescale $S_n$ was not large enough for \emph{both} $g_n$ and $g_n'$ to grow. Indeed, for all we know, it may be that $\av{c_n}$ and $\av{d_n-a_n}$ might be tiny in comparison to $\av{\tilde c_n}$ and $\av{\tilde d_n- \tilde a_n}$. This must be taken into account once we admit that $x\mapsto \LWM{x}{\alpha}$ is not continuous. 
\end{rema}

\subsubsection{Maximal ergodic theorems}\label{subsec:MaxErgTheo}
Let $\epsilon >0$. As said above, by Lusin's Theorem we can find a compact set $K=K_\epsilon$ of measure $>1-\epsilon$, with $x\mapsto \LWM{x}{\alpha}$ being continuous when restricted to $K$. Using Lemmata \ref{lem:InputHMAchine} and \ref{lemm:HPrinciple} and their notation, we see that we “just” need to make sure that for every $n\in \bN$ we find $u_n\in \UAlpha$ with
\begin{enumerate}
    \item\label{item:InTimeWindow} $u_n(s_1)$ and $u_n(s_2)$ belonging the time window $T(\rho,S_n)$, and
    \item\label{item:InK} the sheared points $u_n.x_n$, $u_n.y_n$ belonging to $K$.
\end{enumerate}
If we could ensure this, we would choose accumulation points $x$ and $y$ of $\set{u_n.x_n}, \set{u_n.y_n}$ to find a $u\neq e$ and $x$ and $y$  as in \eqref{eq:OnTheSameLeaf} leading to $u\neq e$ and $x$ satisfying  \eqref{eq:proporInsteadofInv}. But how do we actually know if such $u_n$ exist?

To discuss this further, let's denote by  
$$
B_S(\UAlpha)=\set{u\in \UAlpha:u_n(s_1),u_n(s_2)\in [-S,S]}.
$$ 
A maximal ergodic theorem will help us   “measure” the subset of  $B_S(\UAlpha)$ whose elements satisfy Item \ref{item:InK}. More precisely, first notice that if $\mu$ was $\UAlpha$-invariant (which is a ridiculous assumption, since if this invariance was known to us, we wouldn't be having this discussion), then we could have used a maximal ergodic theorem with respect to the action of $\UAlpha\cong\bR^2$. Such a theorem will tell us that we can find a set $X'=X'(K_\epsilon)$ of measure $1-C\epsilon^{1/2}$ (with $C$ being a universal constant), such that for every $x\in X'$ and for all timescales $S$, we have
\begin{equation}
m_{\UAlpha}\pa{\set{u\in B_S(\UAlpha): u.x\in K}}\geq \pa{1-\epsilon^{\frac{1}{2}}  }  m_{\UAlpha}\pa{B_S(\UAlpha)}.
\end{equation}
where $m_{\UAlpha}$ denotes, under the identification \eqref{LieAlgWeight}, the Borel measure on $\bR^2$. Having this for $x_n$ and $y_n$ as above would definitely help to choose $u_n$ as above. 

But $\mu$ cannot be assumed to be $\UAlpha$-invariant. The only thing we know is that $\mu$ is $A$-invariant. As $\UAlpha$ is $A$-invariant, the only information connecting $\mu$ and $\UAlpha$ is that there is a $\UAlpha$-invariant foliation of the space. Lindenstrauss (together with Rudolph) proves in the Appendix of \textcite{LindenstraussQUE},  that this information is enough in order to deduce the existence of $X'=X'(K_\epsilon)$ exactly as above,  satisfying the exact statement as above, but with $m_{\UAlpha}$ interchanged with $\LWM{x}{\alpha}$: there exists a set $X'=X'(K_\epsilon)$ of measure $1-C\epsilon^{1/2}$ (with $C$ a universal constant), such that for every $x\in X'$  
\begin{equation}\label{eq:MaxErgForLWM}
\forall S>0\,\,\LWM{x}{\alpha}\pa{\set{u\in B_S(\UAlpha): u.x\in K}}\geq \pa{1-\epsilon^{\frac{1}{2}}  }  \LWM{x}{\alpha}\pa{B_S(\UAlpha)}.
\end{equation}
This gives us great information with respect to satisfying Item \ref{item:InK}, but the minute we try to couple this with satisfying Item \ref{item:InTimeWindow} we face another problem: 
let 
$$
F(\rho,S)=\set{u\in \UAlpha: u_n(s_1),u_n(s_2)\in T(\rho,S)}\subset B_S(\UAlpha)
$$
be the points that satisfy Item \ref{item:InTimeWindow}. Note that $F(\rho,S_n)$ is the set of all elements belonging to the timescale $S_n$ and keeping distance $\rho S_n$ from the axes, i.e.~the complement in $B_{S_n}(\UAlpha)$ of a $\rho S_n$-thickening of the axes. So
in order to satisfy Items \ref{item:InTimeWindow} and \ref{item:InK} simultaneously, we need to know, for the $x_n$'s, $y_n$'s and all timescales $S_n$ given to us by Lemma~\ref{lemm:HPrinciple}, that $\LWM{x}{\alpha}(F(\rho,S_n))$ is large in comparison to $\LWM{x}{\alpha}(B_S(U))$. In other words, that $\LWM{x}{\alpha}$ is not concentrated near the axes for a typical $x\in X$. How can this be guaranteed? The short answer is, that it can't. We, or more precisely Lindenstrauss, needed to find a workaround. Before describing it, let's see what we can guarantee.

\subsubsection{Where the joining assumption is used}
 As a first step, let's show that $\LWM{x}{\alpha}$ is not supported \emph{on} the axes. This is based on the following fact about leafwise measures:
\begin{lemm}
  Let $U$ be one of the coarse Lyapunov subgroups above and $U'<U$ an $A$-normalized connected subgroup. We have $$\mu_x^U(U')>0 \iff
      \mu_x^U=\mu_x^{U'}.$$
\end{lemm}
\begin{proof}
See \textcite[Lemma~5.2]{EL_Joinings2019}. The proof goes via a third equivalence, namely, that the entropy contributions are the same, $h_\mu(a,U)=h_\mu(a,U')$ for every $a\in A$.
\end{proof}
This lemma with $U=\UAlpha$ and $U'$ denoting one of the subgroups corresponding to the axes, implies, that for the typical points $x\in X'$, where $X'$ is the conull set from Corollary \ref{cor:SupportProjectsOnto}, we have that $\LWM{x}{\alpha}$ gives measure zero to both axes. Otherwise, the support of $\LWM{x}{\alpha}$ will be contained in one of the axes, and won't project surjectively on both factors. 

Now, fixing $S=1$ and taking $\rho$ small enough, it follows that $\rho$-thickening of the axes inside $B_1(\UAlpha)$ gets arbitrarily small $\LWM{x}{\alpha}$-mass. In other words, given $\epsilon>0$ there exists a set $\tilde X\subset X'$ with $\mu(\tilde X)>1-\epsilon$, and a $\rho>0$ such that 
$$
\LWM{x}{\alpha}(F(\rho,1))>\frac{1}{2}\LWM{x}{\alpha}(B_1(\UAlpha))
$$ 
for all $x\in \tilde X$. 

Now, in the setting of \ExampleOne, we can use the usual maximal ergodic theorem with the element $a:=\Phi_1(1,0)$ that preserves $\mu$, to conclude the following: there exists a subset $Y\subset X$ with $\mu(Y)>1-\epsilon^{1/2}$ such that for every $x\in Y$ and every $R>0$, most ($1-\epsilon^{\frac{1}{2}}$ of the times) $r\in [0,R]$ have $a^r.x\in \tilde X$. Together with \eqref{eq:equivariancy}, this implies that for $x\in Y$ and most $r\in [0,R]$ we have
\begin{equation}\label{eq:manyScales}
\LWM{x}{\alpha}(F(\rho,e^r))>\frac{1}{2}\LWM{ x}{\alpha}(B_{e^r}(\UAlpha)).    
\end{equation}
In other words, we can find a large set of $x\in X$ and  many timescales for which we know that $\LWM{x}{\alpha}$ is not concentrated near the axes. This is a self-similarity property for $\LWM{x}{\alpha}$; one can expect such properties for measures invariant under a diagonalizable element with positive entropy. 

\begin{rema}
A similar statement holds for \ExampleTwo. The only difference is that we need to replace the $B_{e^r}(\UAlpha)$, which is defined as “the square in $\UAlpha$  with lengths $e^r$”  with a “rectangle in $\UAlpha$ with lengths $e^r$ and $e^{2r}$” (and change similarly all the other sets appearing above). 
\end{rema}

Having \eqref{eq:manyScales} is unfortunately, not enough for running the $H$-principle argument we are after. We still cannot guarantee that the timescales given to us by Lemma~\ref{lemm:HPrinciple} are not exactly the problematic scales we don't cover in \eqref{eq:manyScales}. Lindenstrauss had another rabbit in the hat: he noticed that tweaking the initial pair $x_n$ and $y_n$ from Lemma~\ref{lem:InputHMAchine} by $a^t:=\pa{\Phi_1(1,0)}^t$ for $t\in [0,T_n]$ still preserves the equality of the corresponding leafwise measures. Note that $T_n=T_n(\delta_n)$ are needed to be chosen carefully so the displacement between $a^t.x_n$ and $a^t.y_n$ won't grow too much. Having this extra freedom in choosing the initial pairs might give more possible timescales, for which we might be able to guarantee regularity properties as in \eqref{eq:manyScales}. An intricate calculation shows that this is indeed the case, enabling Lindenstruass to run the $H$-principle against all odds ---see \textcite[\S 7.2]{LindenstraussQUE}. Essentially the same calculation (but done in the first and the third component simultaneously),  enables us to run it also in our setting. We avoid discussing it here and finish our outline of the proof of Proposition \ref{prop:mainStep} under the simplifying assumption that a regularity property as in $\eqref{eq:manyScales}$ holds for any scale.

\subsubsection{Proof of Proposition \ref{prop:mainStep} under a simplifying self-similarity assumption}
For concreteness consider $\lambda=\alpha$ and to simplify notation denote $B_S=B_S\pa{\UAlpha}$.
We assume that there is a $\rho\in (0,1)$  such that for $\mu$-almost every $x$ we have
\begin{equation}\label{eq:EveryScale}\forall S>1\quad
\LWM{x}{\alpha}(F(\rho,S))>\frac{1}{2}\LWM{ x}{\alpha}(B_S),    
\end{equation}
 and outline a proof of Proposition \ref{prop:mainStep}. The reader may compare this outline to \textcite[\S 7.1]{LindenstraussQUE}.  Assume, for contradiction, that Proposition \ref{prop:mainStep} does not hold. Then, \eqref{eq:proporInsteadofInv} does not hold with $\lambda=\alpha$ on a subset of positive measure which is $A$-invariant by \eqref{eq:AActionOnInv}. Ergodicity then implies that for a typical point $x\in X$, \eqref{eq:proporInsteadofInv} does not hold with $\lambda=\alpha$. Now, let $\epsilon>0$ and choose a \emph{compact} subset $X_1\subset X$ with $\mu(X_1)>1-\epsilon$ and with:
\begin{enumerate}
    \item $\LWM{x}{\alpha}$ is defined for $x\in X_1$, and the map $x\mapsto \LWM{x}{\alpha}$ is continuous on $X_1$ ,
    \item \eqref{eq:proporInsteadofInv} does not hold for $x\in X_1$ (with $\lambda=\alpha$),
\end{enumerate}
By \S \ref{subsec:MaxErgTheo} there exists a compact set $X_2$ with $\mu(X_2)>1-\epsilon^{\frac{1}{2}}$ such that any $x\in X_2$ satisfies   \eqref{eq:MaxErgForLWM} with $K=X_1$.

Now, use Lemma~\ref{lem:InputHMAchine} with $X'=X_2$ to find pairs of nearby points as stated there; let $x$ and $y$ be such a pair with distance $\delta=\delta(\epsilon)>0$ 
 small enough so that the corresponding timescale $S$ from Lemma~\ref{lemm:HPrinciple} will be large enough, so that \eqref{eq:EveryScale} holds.  Let
$G_x=\set{u\in \UAlpha:u.x\in X_1}$ and $G_y=\set{u\in \UAlpha:u.y\in X_1}$. From \eqref{eq:MaxErgForLWM} we know that $$\LWM{x}{\alpha}\pa{G_x\cap B_S}\geq \pa{1-\epsilon^{\frac{1}{2}}  }  \LWM{x}{\alpha}\pa{B_S}
$$
and similarly for $y$. So, 
$$\LWM{x}{\alpha}\pa{G_x\cap G_y\cap  B_S}\geq \LWM{x}{\alpha}\pa{B_S}- \LWM{x}{\alpha}\pa{  B_S\setminus G_x}- \LWM{x}{\alpha}\pa{  B_S\setminus G_x}\geq   (1-2\epsilon^{\frac{1}{2}})\LWM{x}{\alpha}\pa{B_S}.
$$
In particular for $\epsilon$ small enough, we can find $u\in B_S$ enabling us to use the $H$-principle to find $x'=u.x,y'=u.y\in X_1$ satisfying 
\begin{itemize}
    \item $\LWM{x'}{\alpha}=\LWM{y'}{\alpha}$
    \item $y'=(g'u').x'$ with $u'\in \UAlpha$ having at least one of its coordinates, $u'(s_1)$ or $u'(s_2)$, in $[\frac{1}{C},C]$,  and with $g'\in G$ of size at most $\delta^{\frac{1}{2}}$.
\end{itemize}
Applying this to all the pairs of nearby points coming from Lemma~\ref{lem:InputHMAchine}, we get pairs $x_n',y_n'\in X_1$ as above with $\delta_n\to 0$. We choose a subsequence where the resulting $u_n'$ as above all belongs to a compact set in at least one of their coordinates, say their first one. That is, we assume that $u_n'$ satisfy $u_n'(s_1)\in [\frac{1}{C},C]$ along this subsequence. As $X_1$ is compact, and $x\mapsto \LWM{x}{\alpha}$ is continuous on $X_1$,  we explained in \S \ref{subsec:OptimisticIdea} that limit points of $\set{x_n}$ and of $\set{y_n}$ will give us points $x,y=u.x\in X_1$ with $u\neq e$ with equal leafwise measure. That is, we will find  $x\in X_1$ satisfying  \eqref{eq:proporInsteadofInv}, in contradiction to the definition of $X_1$.

\section{Applications of the joinings Theorem }\label{sec:Applicaitions}
In \S \ref{sec:TorusOrbitsArith} we survey many arithmetic problems and their relation to adelic/$p$-adic torus orbits; In \S \ref{subsec:ArithGeoJoinings} we couple these problems to get problems related to the  classification of joinings. We then explain what  Theorem \ref{thm:mainThm} can say about these coupled problems. Finally, in 
\S \ref{subsec:HighRankAppli} we give a new arithmetic application of Theorem \ref{thm:mainThm} for simple groups of high rank.

We recall that when $(X,m_{X})$ is a probability measure space we say that a sequence of finite set $A_i\subset X$ \emph{equidistribute}, when the normalized counting measure on $A_i$ converges in the weak-* topology to $m_{X}$.
\subsection{Torus orbits and arithmetic}\label{sec:TorusOrbitsArith}

\subsubsection{Three classical arithmetic distribution problems}\label{subsubsec:3ClasicalProblems} Let us first revisit the arithmetic distribution problem we considered in the introduction (see \eqref{eq:EquiOnSpheres} for the notation): the distribution of $\frac{1}{\sqrt D}\bS^2(D)$ inside $\bS^2$ for large $D>0$ with $D\neq 0,4,7\mod 8$. Note that the set $\bS^2(D)$ can be also considered as the set of integer points on the level set $Q_1=D$ for the quadratic form $Q_1(x,y,z)=x^2+y^2+z^2$. 

The second distribution problem we consider is very similar. For $D<0$, one can consider the set of primitive binary quadratic forms with discriminant $D$
\begin{equation}
    \Bin_D=\set{q_{\pa{a,b,c}}(x,y)=ax^2+bxy+cy^2:b^2-4ac=D, \,\gcd(a,b,c)=1}.
\end{equation}
By identifying $\Bin_D$ with the set of the corresponding coefficient vectors $\pa{a,b,c}$, one can realize $\Bin_D$ as the set of integer points on the level set for the quadratic form $Q_2(X,Y,Z)=Y^2-4XZ$: 
\begin{equation}
    V_D:=\set{(X,Y,Z)\in \bR^3:Y^2-4XZ=D}
\end{equation}
and ask about the distribution of $\frac{1}{\sqrt{\av{D}}} \Bin_D\subset V_{-1}$ when $D\to -\infty$ (along negative  \emph{discriminants}, which are by definition the set on  $D$'s with $\Bin_D\neq \emptyset$). Note that the natural “cone” measure on  $V_{-1}\subset \bR^3$ is infinite. This is not really an issue; one just measures the distribution with respect to all possible  quotients of two test functions. But for other reasons, one can reformulate this distribution problem in an equivalent “modular” way, which yields a distribution problem in a finite-measure space. To this end, we recall that \textcite{Gauss} defined   an action of $\GL_2(\bZ)$ on $\Bin_D$ by 
$$
g.q(x,y):=\frac{1}{\det(g)} q\pa{(x,y)g},\quad q(x,y)\in \Bin_D,\,g\in \GL_2(\bZ)
$$
and showed that this action has a finite number of orbits (forming an abelian group, see below). For $q(x,y)\in \Bin_D$ let $z_q$ denote the unique root of $q(z,1)$ belonging to the upper half plane $\bH=\set{z\in \bC:\Im(x)\geq 0}$. The set of such roots are called \emph{complex multiplication points}\footnote{Equivalently, a point $z\in \bH$ is called a \emph{complex multiplication point of discriminant $D$} if the automorphisms group of the corresponding lattice $\mathrm{span}_{\bZ}(1,z)\subset \bC$ is the quadratic order of discriminant $D$.} or \emph{Heegner points} of discriminant $D$. Furthermore, let $[z_q]$ denote the orbit of $z_q$ under the action of $\SL_2(\bZ)$ by Möbius transformations on the $\bH$. One can then verify that the $\GL_2(\bZ)$-orbit of $q\in \Bin_D$ corresponds to the $\SL_2(\bZ)$-orbit $[z_q]$. So the set $\cH_D:=\set{[z_q]:q\in \Bin_D}\subset \SL_2(\bZ)\backslash\bH$ is finite. A dual reformulation of the above distribution problem on $V_{-1}$ is the distribution of the normalized counting measures of the finite sets $\cH_D$ on  $\SL_2(\bZ)\backslash\bH$ equipped with the finite, normalized uniform measure $\frac3 \pi{\displaystyle (ds)^{2}={\frac 3 \pi\frac {(dx)^{2}+(dy)^{2}}{y^{2}}}}$ when $D\to-\infty$. Figures \ref{Fig:Wieser1}-\ref{Fig:Wieser3} were made by Andreas Wieser to show the distribution of $\cH_D$ with growing $\av{D}$'s.
 
 \begin{figure}
   \centering
       \begin{minipage}{0.3\textwidth}
       \centering
\includegraphics[scale=0.5]{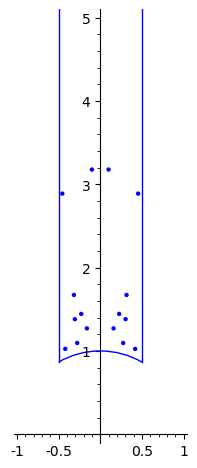} 
       \caption{$D = -1009$}
       \label{Fig:Wieser1}
   \end{minipage}\hfill
   \begin{minipage}{0.3\textwidth}
       \centering
\includegraphics[scale=0.5]{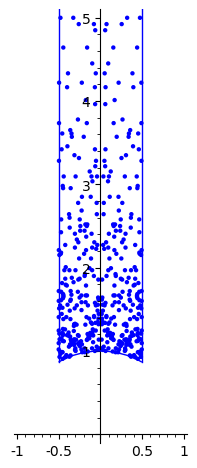}
       \caption{$D = -105509$}
   \end{minipage}\hfill
   \begin{minipage}{0.35\textwidth}
       \centering
\includegraphics[scale=0.5]{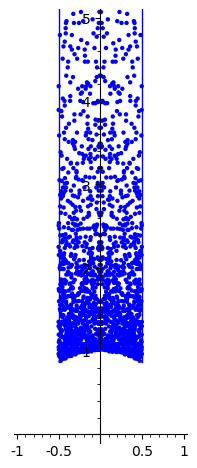}
       \caption{$D = -1299821$}
       \label{Fig:Wieser3}
   \end{minipage}
\end{figure}


We shortly mention a third, related distribution problem. In the last distribution problem, one can also consider positive discriminants $D>0$ and the corresponding distribution of $\frac{1}{\sqrt{D}}\Bin_D\subset V_1$ as $D\to \infty$ along positive discriminants. This problem also admits a modular formulation as above: one sees that $\GL_2(\bZ)$-orbits in $\Bin_D$ correspond to a finite set of closed geodesics on the modular surface $\SL_2(\bZ)\backslash\bH$. See the work of \textcite{ELMVDuke} for more details and nice images of this distribution for some large $D$'s.
As we partially explain in \ref{subsec:ClassGroups}, roughly speaking, the sets $\bS^2(D)$ for $D>0$, and the sets   $\cH_D$ or $\Bin_D$ for $D<0$ with  $D=3\mod 4$, and  the sets $\cH_{4D}$ or $\Bin_{4D}$ for $D<0$ with  $D=1,2\mod 4$,  are all torsors\footnote{That is, they are acted upon freely and transitively by $\ClassGp{D}$} for the class group $\ClassGp{D}$ of the quadratic order of discriminant $D$, and therefore have size $h_D=\av{\ClassGp{D}}$. For $D>0$, the finite set of closed geodesics corresponding to the $\GL_2(\bZ)$-orbits in $\Bin_D$ has size $h_D$.  Gauss conjectured by Gauss that for $D>0$,  $h_D$ is infinitely often equal to 1.  Figure \ref{Fig:Kontrovich} shows images made by Alex Kontorovich showing the distribution of these geodesics for two comparable discriminants, one with a trivial class group and one with a class group of order $4$ (The notation $[a,b,c]$ stands for the binary form $ax^2+bxy+cy^2$).

\begin{figure}[h]
\centering
\includegraphics[width=0.9\textwidth]{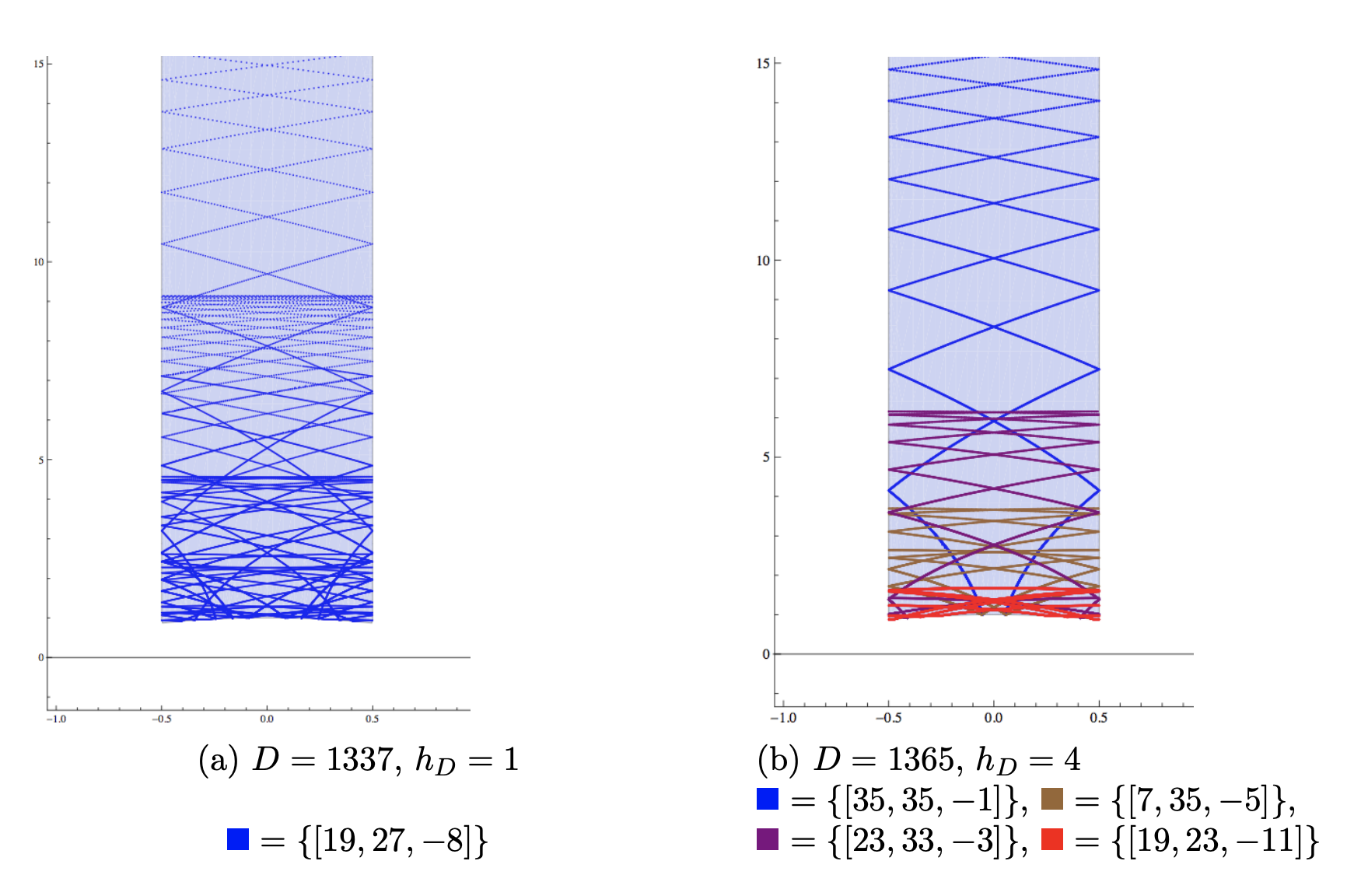}
\caption{Distribution of closed geodesics (by Alex Kontorovich)}
\label{Fig:Kontrovich}
\end{figure}

We avoid giving historical details for these problems. We just mention that, as the images might suggest, in  these three problems the sets $\bS^2(D)$ (resp.~$\Bin_D$ or equivalently~$\cH_D$) equidistribute with the respect to the natural measures on $\bS^2$ (resp.~$V_{-1}$,$V_{1}$ or equivalently $\SL_2(\bZ)\backslash\bH$). Under an \emph{auxiliary} congruence condition on~$D$  these results were proven by Linnik and Skubenko in the late 50s (see \cite{LinnikBook}). \textcite{Duke88} proved it in full generality, building on a work of \textcite{IwaniecBreak}. We refer below to these equidistribution results (and some other variants of them mentioned below) collectively as \emph{Duke's Theorem}. 

There are several surveys of these results which might interest the reader: for a modern approach to the methods of Linnik and his school, we refer to the work of \textcite{ELMVDuke}, \textcite{EMVpoints} and to the work of  \textcite{WieserLinnik}. Some methods used in these works are partially related to the methods used in the proof of Theorem \ref{thm:mainThm}. For a broader overview, also on the works of Duke and Iwaniec mentioned above, the reader can consult the survey by \textcite{DukeSurvey2007} and by \textcite{MichelVenkateshICM}.  Finally, these results (and most of the applications mentioned in this survey) are also considered in the recent survey by \textcite{lindenstrauss2021Survey}.

\subsubsection{Relation to class groups and adelic torus orbits}\label{subsec:ClassGroups} It is not a priori clear how  the above distribution questions relate to questions about the distribution of torus orbits in ($S$-arithmetic/ adelic) homogeneous spaces. Let us concentrate on the first problem above, the distribution of $\frac{1}{\sqrt D}\bS^2(D)\subset \bS^2$, $D>0$. We claim that the distribution properties of $\bS^2(D)$ can be studied through the distribution properties of an orbit of the group $\bH_v(\bA)$ in the quotient $\SO_{3}(\bQ)\backslash\SO_{3}(\bA)$. Here, $v$ denotes a vector in $\bS^2(D)$, $\bH_v$ denotes the stabilizer group of $v$ with respect to the standard action of $\SO_3$ on the three-dimensional space, and $\bA$ denotes the ring of adèles of $\bQ$. Note that, in this case, the group $\bH_v(\bA)$ is the adelic points of a rank one, $\bQ$-anisotropic algebraic torus. We remark that $\SO_3=\SO_{Q_1}$ where $Q_1$ was defined in \S \ref{subsubsec:3ClasicalProblems}. Everything we will say below can be done analogously with $Q_2$ instead of $Q_1$, when one considers $\cH_D$ instead of $\bS^2(D)$.

To convince the reader that this relation is plausible, and relate it also to the class group $\ClassGp{-D}$, let us define a map that seems contradictory:  we fix a vector $v\in \bS^2(D)$ and we wish to define an (essentially bijective) map from a set of double-cosets related to $\bH_v$ (see \eqref{eq:classGroupAdele}) to the set $\bS^2(D)$, by using elements of $\bH_v(\bA_f)$ to move $v$ \emph{non-trivially} to other vectors in $\bS^2(D)$. Here, $\bA_f$ denotes the finite adèles. To define this map, we first note that $\SO_3$ has class number one, that is,
\begin{equation}\label{eq:classNumOne}
    \SO_3(\bA_f)=\SO_3(\bQ)\SO_3( \widehat{\bZ})
\end{equation}
where $\SO_3(\bQ)$ is embedded diagonally in $\SO_3(\bA_f)$ and $\widehat{\bZ}$ denotes the (compact) ring $\prod_{p\text{ prime}}\bZ_p$ (For a proof of \eqref{eq:classNumOne}, one may consult \cite[\S 5.1]{EMVpoints}). Note that this fact is true for this specific ternary form $Q_1$, although general forms can be treated with the same techniques too. One can therefore write any element of $\SO_3(\bA_f)$, and in particular any $h=\pa{h_p}_p\in \bH_v(\bA_f)$ as
\begin{equation}\label{eq:ClassNumber1}
    h=\gamma^{-1}k,\,\text{ with } \gamma\in \SO_3(\bQ),\,k=(k_p)_p\in \SO_3(\widehat{\bZ}). 
\end{equation}
Using this decomposition we can take  $h\in \bH_v(\bA_f)$ and use it to generate a vector in $\bS^2(D)$: 
for $h=\pa{h_p}_p\in \bH_v(\bA_f)$ let $\gamma$ and $k=(k_p)_{p}$ be as in equation \eqref{eq:ClassNumber1} and define 
\begin{equation}
    h*v:=\gamma.v\in \bS^2(D).
\end{equation}
To show that the vector $h*v$ is indeed in $\bS^2(D)$ and that this process is well-defined, first note that for any prime $p$ we have (as $\gamma$ is diagonally-embedded) 
\begin{equation}
    \bQ^3\ni\gamma.v=\gamma.h_p.v\stackrel{\eqref{eq:ClassNumber1}}{=}(\gamma\gamma^{-1}k_p).v=k_p.v\in \bZ^3_p.
\end{equation}
As $\bQ\cap\widehat{\bZ}=\bZ$ we have that $\gamma.v\in \bZ^3$ and since we act by the orthogonal group, it has the same length and is therefore indeed an element of $\bS^2(D)$. We also have  $\gamma.v=k_p.v$ for every prime $p$. However, $h*v$ does depend slightly on the choices made in \eqref{eq:ClassNumber1}: choosing a different decomposition in \eqref{eq:ClassNumber1} will yield the same vector up-to $\SO_3(\bZ)$, so this process gives a well-defined map to $\SO_3(\bZ)\backslash\bS^2(D)$. The finite group $\SO_3(\bZ)$ leaves all the measures we are interested in invariant, so this small detail does not matter too much for us (and we will keep saying, imprecisely, that we got a map to $\bS^2(D)$). A priori, there is no reason to believe that we get new vectors from this process, that is, that the image of this map is interesting. Purely algebraically (see  \cite[Theorem 8.2]{PR94}) one can show that the set of new vectors we generate from $v$ out of $\bH(\bA_f)$ via the above process, is the so-called \emph{genus} of $v$: 

\begin{equation}\label{eq:genusOfv}
    \gen(v)=\SO_3(\bZ)\backslash\set{w\in \bQ^3:\forall p\,\,w\sim_{\bZ_p}v,\,w\sim_{\bQ}v}
\end{equation}
where with $w\sim_{R}v$ for a ring $R$ we mean that there exists $g\in \SO_3(R)$ with $g.w=v$. Moreover, it follows from purely algebraic considerations, that this process defines a bijection between $\mathrm{gen}(v)$ and the following double quotient:
\begin{equation}\label{eq:classGroupAdele}
    \bH_v(\bQ)\backslash\bH_v(\bA_f)/\bH_v(\widehat{\bZ}).
\end{equation}
Note that this double quotient  is a (abelian) group since $\bH_v$ is abelian.
Assume for simplicity that $D$ is fundamental. As an algebraic torus, $\bH_v$ is isomorphic to the rank one $\bQ$-anisotropic torus $\mathrm{Res}_{K/\bQ}\bG_m/\bG_m$ for $K=\bQ(\sqrt{-D}$). This identifies (with small bounded index depending on this isomorphism of the algebraic tori) the double quotient in \eqref{eq:classGroupAdele} with the similar double quotient for $\mathrm{Res}_{K/\bQ}\bG_m/\bG_m$ (see \S~\ref{subsec:GenrQA} for more details). The latter is by definition the idèle class group $K^*\backslash \bA_{K,f}^\times/\widehat{\cO_{-D}}^\times$, which is isomorphic to the class group $\ClassGp{-D}$. In this specific case, one can show that $\mathrm{gen}(v)$ is equal to $\SO_3(\bZ)\backslash\bS^2(D)$ (see for example \cite[\S 6]{EMVpoints}). In summary, the above describes a (essentially bijective) map from the class group $\ClassGp{-D}$ onto $\bS^2(D)$ as promised (ignoring again the division by $\SO_3(\bZ)$ and the above-mentioned small bounded index issues). This map depends of course on the choice of the “base point” $v$. In \textcite[Proposition 3.5]{EMVpoints} they explain how to construct from these base-dependent maps a base-free way to give $\bS^2(D)$ a torsor structure for $\ClassGp{-D}$, that is, how to construct a free and transitive action of $\ClassGp{-D}$ on $\bS^2(D)$.

We remark that the above algebraic construction does not tell us \emph{anything} about how $\bS^2(D)$ distributes, and in particular does not give us any reason to believe that $\frac{1}{\sqrt D} \bS^2(D)$ equidistributes. In section \ref{subsec:WhyItWorks} we will give a dynamical interpretation of this construction, from which the above construction arises, and through which one can see why $\frac{1}{\sqrt D} \bS^2(D)$ should equidistribute.

\subsubsection{A hidden distribution problem}\label{subsec:CMSS}
We shortly describe another arithmetic problem because it relates in a slightly different way to the distribution of  adelic torus orbits and because it gives rise naturally to a problem about joinings. We avoid giving many details, skip some definitions, and instead refer the interested reader to the introduction of \textcite{aka2020simultaneous}.

Let $\cO_D$ be the quadratic imaginary order with discriminant $D<0$. We say that an elliptic curve $E$ over $\bC$ has complex multiplication (CM) by $\cO_D$ if its endomorphism ring is isomorphic to $\cO_D$. We let $\CM_D$ denote the set of isomorphism classes of elliptic curves having complex multiplication by $\cO_D$ (Roughly speaking, one may realize this set as the isomorphism classes of  $\set{\bC/\Lambda_{z}:[z]\in \cH_D}$ where $\Lambda_{z}$ denotes the lattice $\mathrm{span}_\bZ(1,z)\subset \bC$). This set is finite, of order $h_D\asymp \av{D}^{\frac12+\epsilon}$ and consists  of curves which are defined over a finite algebraic extension of $\bQ$.

Fix an odd prime $p$ and consider now only negative $D$'s for which $p$ is inert in~$\cO_D$. Fixing an embedding of $\overline{\bQ}$ in $\overline{\bQ_p}$, one can  naturally define the reduction map $\red_p\colon \CM_D\to \cS_p$ where $\cS_p$ is the set of supersingular elliptic  curves over $\overline{\bF_p}$. The set $\cS_p$ is finite, of order $\asymp\frac{p-1}{12}$. \textcite{DeuringSurj} showed  that any element of $\cS_p$ is the reduction of an element of $\CM_D$ for some negative $D$. It raised the question if there is a discriminant~$D$ (with $p$ being inert in~$\cO_D$) for which this reduction map is surjective. Proving a variant of the works of Duke and Iwaniec mentioned above, and based on a work of  \textcite{Gross87}, \textcite{Michel04} proved  that surjectivity holds for large enough $\av{D}$. In fact, he explained that this question is a slight variant of the problems considered above: the distribution of the adelic points of a rank one, $\bQ$-anisotropic  tori in a homogeneous spaces of the form $\SO_Q(\bQ)\backslash\SO_Q(\bA)$ for some specific quadratic form~$Q$ (related only to~$p$). In \S \ref{subsec:GenrQA} we shortly review a generalized setting,  in which one can consider all the above problems in a uniform way. Before we do this, let us just hint how this surjectivity problem could lead to problems involving joinings: fix  two odd primes $p,q$ and consider now only negative $D$'s with both $p$ and $q$ being inert in $\cO_D$. Is the map $(\red_p,\red_q)$ surjective for $\av{D}$  large enough? If so, can  this be generalized to any finite number of primes?



\subsubsection{A common generalization}\label{subsec:GenrQA}
All the above problems fit in the following setting. Let $B$ be a quaternion algebra over $\bQ$. Recall that quaternion algebras over $\bQ$ are classified by a finite, even-sized set of places $S$, where they ramify, that is, where $B\otimes\bQ_p$ is  a division-algebra. Let $\Nr$ be the associated reduced norm and $Q:=\Nr|_{B^0}$ the quadratic form induced by the restriction of $\Nr$ to the traceless quaternions $B^0$. Recall that 
\begin{align}
\begin{split}
  B^\times\stackrel{\pi}{\to}PB^\times:=B^\times/Z(B^\times)\cong \,&\,\,\SO_Q \\
 g\mapsto \,&\,\, \pa{x\in B^0\mapsto gxg^{-1}\in B^0 }
\end{split}
\end{align}
where $Z(B^\times)$ denotes the center of $B^\times$. For $v\in B^0$, the preimage of the stabilizer group $\bH_v<\SO_Q$ under this map is the group of the invertible elements of the ring $$\bQ[v]\cong\bQ[X]/\pa{X^2+\Nr(v)}=:K,$$ 
since for $v\in B^0$ we have $v^2=-v\bar v=-\Nr(v)\in \bQ$.
Let $\bG$ denote the algebraic group $PB^\times$ and $\bT_v=\Res_{K/\bQ}\bG_m/\bG_m<\bG$ denote the image of $\bQ[v]$ under $\pi$ in $\bG$. Note that $\bT_v$ is a rank one $\bQ$-anisotropic algebraic torus.

All the above arithmetic problems, and many others, can be recast naturally, as follows. We consider the adelic homogeneous space $\bG(\bQ)\backslash\bG(\bA)$, and for a \emph{primitive} vector $v\in B^0(\bZ)$ and call an orbit of the form 
\begin{equation}\label{eq:toralPacket}
\bG(\bQ)\bT_v(\bA)\pa{g_\infty,g_f},\quad \pa{g_\infty,g_f}\in \bG(\bR)\times \bG(\bA_f)  
\end{equation}
a \emph{toral packet} related to $v$. We equip  a toral packet with the pushforward of the natural Haar probability measure on the quotient $\bT_v(\bQ)\backslash\bT_v(\bA)$ under the map $t\mapsto t\pa{g_\infty,g_f}$. One can attach to such an orbit a \emph{discriminant} which is, roughly said, related to $-\Nr(v)$. More precisely, it can be defined as the product of local discriminants; for a non-Archimedean place, say at $p$, the local discriminant is the discriminant of the of order $\bQ_p[v]\cap g_f(B^0(\bZ_p))g_f^{-1}$, and for the Archimedean place, it is a  measure of the distortion of the torus $g_\infty^{-1}\bT(\bR)g_\infty$. 

For example, for $B=B_{\infty,2}$, which is  Hamilton's quaternion algebra, $\Nr|_{B^0}$ is the sum of three squares. So for any $v\in  B^0(\bZ)$, we have $\bS^2(D)\subset B^0(\bZ)$ with $D=\norm{v}^2_2$. Recall that in this case $\bG=PB^\times$ is $\SO_3$ and that $\bH_v$ denotes the stabilizer of $v$. We will explain in \S \ref{subsec:WhyItWorks}, how the toral packet $\bG(\bQ)\bH_v(\bA)\pa{k_v,e_f}$ with $k_v\in \SO_3(\bR)$ moving $(0,0,1)^t$ to $\frac{1}{\sqrt D}v$ and $e_f\in \bG(\bA_f)$ denoting the identity element, relates naturally to the set $\gen(v)$ defined in \eqref{eq:genusOfv}. Note that in this case the local discriminant in the real place is constant, as  $k_v^{-1}\bH_v(\bR)k_v$ is the fixed torus $\SO_2(\bR)$ (see \S \ref{subsec:WhyItWorks} for more details).
This explains a bit the above terminology: the orbit \eqref{eq:toralPacket} “codes” a packet of vectors which are in the genus of $v$. 

In this generalized context, we refer to the following statement as Duke's Theorem: toral packets equidistribute, when their discriminants tend to $+\infty $ or $-\infty$, to a $\bG(\bA)^{+}$-invariant probability measure\footnote{The group $\bG(\bA)^{+}$ is the image of the adelic points of the algebraic simply-connected cover of $\bG$ under the canonical projection map (which is the Spin group of $Q$ in our case).} on $\bG(\bQ)\backslash\bG(\bA)$. Note that in many cases, but not in general,  this probability measure could be shown to be  $\bG(\bA)$-invariant, i.e.~the Haar measure on $\bG(\bQ)\backslash\bG(\bA)$.  This formulation of Duke's Theorem is the conglomeration of several results by many authors, achieved before and after the seminal result of  \textcite{Duke88}. We refer the reader to \textcite[Theorem 4.6]{ELMVCubic} for a more precise formulation and for a list of references. Until now, there were three main approaches for proving such a statement: via ergodic theory (essentially Linnik's original method, under a congruence condition assumption), via modular form (as \cite{Duke88}) and via subconvexity estimates. The last two approaches yield an error estimate. The subconvexity approach works naturally in the generality stated above, and can be further generalized to obtain stronger variants, such as the work of \textcite{MichelHarcos06} (see also the discussion in Example \ref{ex:PowerInClassGroup} below).

The problems mentioned above fit in the above setup as follows:
\begin{itemize}
    \item The distribution of $\bS^2(D)$, as already mentioned above,  relates to the choice $B=B_{\infty,2}$ for which the norm of $v\in B^0$ calculates as $Q(v)=\Nr(v)=\Nr(xi+yj+zk)=x^2+y^2+z^2$.
    \item The distribution of $\cH_D$ (resp.~closed geodesics) in $\SL_2(\bZ)\backslash\bH$ relates to the choice $B=B_{\emptyset}=M_{2\times 2}$ for which the norm of $v\in B^0$ calculates as $Q(v)=\Nr(v)=\Nr\bigl( \begin{smallmatrix}x & y\\ z & -x\end{smallmatrix}\bigr)=-\pa{x^2+yz}$, where the scalar matrices are identified with the underlying field. This corresponds, when $Q(v)\ra \infty$  to the distribution of Heegner points, and when $Q(v)\to -\infty$ to the distribution of closed geodesics.
    \item The problem considered in \ref{subsec:CMSS}, that is, the reduction of elliptic curves with CM by $\cO_D$ to $\cS_p$, the set of supersingular elliptic curves over $\overline{\bF_p}$, corresponds to the case $B=B_{\infty, p}$, but in an interesting way: one can show that with this choice the double cosets in 
    \begin{equation}\label{eq:DoubleQuoSS}
        \bG(\bQ)\backslash\bG(\bA)/\bG(\bR\times\widehat{\bZ}),
    \end{equation} 
    or said differently, the clopen (closed and open) orbits of $\bG(\bR\times \widehat{\bZ})$ on $\bG(\bQ)\backslash\bG(\bA)$, are in bijection with $\cS_p$. 
    As $p$ is inert in $\cO_D$ there is an embedding of $\Res_{\bQ(\sqrt{D})/\bQ}\bG_m/\bG_m$ into $\bG$ and the CM curves are related to a toral packet as above. Surjectivity  of $\red_p$ translates then to the question if the toral packet visits all the above clopen set, or equivalently, if the projection of toral packet to \eqref{eq:DoubleQuoSS} is surjective. In this respect, this application is closer in nature to the work of \textcite{EV08LocalGlobal}, which is the first instance known to the author where the algebraic ideas we explain here, were coupled with homogeneous dynamics considerations. Duke's theorem implies in this context that $\red_p\colon\CM_D\to \cS_p$ is surjective and actually it gives a strengthening of surjectivity: the asymptotic size of the fibres for $E\in\cS_p$,  $\lim_{D\ra -\infty}\frac{\av{\red_p^{-1}(E)}}{\av{\CM_D}}$ is equal to the $m_{\bG(\bQ)\backslash\bG(\bA)}$-measure of the   $\bG(\bR\times\widehat{\bZ})$-orbit corresponding to $E$. Here $m_{\bG(\bQ)\backslash\bG(\bA)}$ denotes the natural $\bG(\bA)$-invariant probability measure on ${\bG(\bQ)\backslash\bG(\bA)}$. In this sense, one may say that the projection map $\red_p$ equidistributes.
\end{itemize}


\subsubsection{Why and how this works?}\label{subsec:WhyItWorks}
In section \ref{subsec:ClassGroups} we defined a map from 
$\bH_v(\bQ)\backslash\bH_v(\bA_f)/\bH_v(\widehat{\bZ})$ to $\frac{1}{\sqrt D}\bS^2(D)\subset \bS^2$. In this subsection, we would like to reinterpret this map via a projection of a toral packet between two homogeneous spaces. 
Consider the natural projection
\begin{equation}
    \pi\colon\SO_3(\bQ)\backslash \SO_3(\bA)\to \SO_3(\bQ)\backslash \SO_3(\bA)/\SO_3(\widehat{\bZ})
\end{equation}
where as above, $\SO_3(\bQ)$ is diagonally embedded in $\SO_3(\bA)$. We will show below the identification
\begin{equation}\label{eq:identification}
     \SO_3(\bQ)\backslash \SO_3(\bA)/\SO_3(\widehat{\bZ})\cong\SO_3(\bZ)\backslash \SO_3(\bR).
\end{equation}
With this identification assumed,  we may interpret $\pi$ as a covering map 
\begin{equation}\label{eq:CoveringMap}
    \pi\colon\SO_3(\bQ)\backslash \SO_3(\bA)\to \SO_3(\bZ)\backslash \SO_3(\bR)
\end{equation}
 from an adelic homogeneous space  to a real homogeneous space (with compact fibres isomorphic to $\SO_3(\widehat{\bZ})$).
The identification \eqref{eq:identification} follows “tautologically” from \eqref{eq:classNumOne}:
\begin{align}
    \SO_3(\bQ)\backslash \SO_3(\bA)& \stackrel{\eqref{eq:classNumOne}}{=}\SO_3(\bQ)\backslash\pa{\SO_3(\bQ)\SO_3(\bR\times\widehat{\bZ})} \\
    &\cong\pa{ \SO_3(\bQ)\cap \SO_3(\bR\times\widehat{\bZ})} \backslash \SO_3(\bR\times\widehat{\bZ})\\
    &=\SO_3(\bZ) \backslash \SO_3(\bR\times\widehat{\bZ})
\end{align}
where the third equality follows from the fact that $\bZ=\bQ\cap\widehat{\bZ}$ as subsets of $\bA_f$. Dividing now everything from the right by $\SO_3(\widehat{\bZ})$ gives the desired identification. It is nevertheless useful to explicate this identification to start to “see” dynamics: consider the double coset $\SO_3(\bQ)\pa{g_\infty,g_f}\SO_3(\widehat{\bZ})$ with $g_\infty\in \SO_3(\bR),g_f\in \SO_3(\bA_f)$. This coset may be identified with $\SO_3(\bR)g_\infty$ if $g_f$ is “small”, that is, if $g_f$ belongs to the compact group $\SO_3(\widehat{\bZ})$. But if $g_f$ is “large”, that is, if $g_f\notin\SO_3(\widehat{\bZ})$, we first need to “bring it back to $\SO_3(\widehat{\bZ})$” using the discrete group $\SO_3(\bQ)$. This means that, we first need to find $\gamma\in \SO_3(\bQ)$ with $\gamma g_f=:k\in \SO_3(\widehat{\bZ})$ and then, recalling that $\SO_3(\bQ)$ is embedded diagonally, we have 
$$
\SO_3(\bQ)(g_\infty,g_f)\SO_3(\widehat{\bZ})=\SO_3(\bQ)(\gamma g_\infty,\gamma g_f)\SO_3(\widehat{\bZ})=\SO_3(\bQ)(\gamma g_\infty,k)\SO_3(\widehat{\bZ}).
$$
As $k=\gamma g_f$ is “small”, the latter is identified as above with $\SO_3(\bZ)\gamma g_\infty$.

Although everything we said so far is purely algebraic, the adjectives “small”, “large” and the process of “bringing $g_f$ back” stem from a dynamical perspective. Before we can explain this perspective, we just need to understand how the sphere $\bS^2$, or more precisely $\SO_3(\bZ)\backslash\bS^2$, can be described as a homogeneous space. Consider the natural action of $\SO_3(\bR)$ on $\bS^2$ (or equivalently on $\bR^3$) and let $\SO_2(\bR):=\mathrm{Stab}_{\SO_3(\bR)}(e_3)<\SO_3(\bR)$, where $e_3$ is the unit vector $\pa{0,0,1}^t$. This yield the identification $\SO_3(\bR)/\SO_2(\bR)\cong \bS^2$ and therefore also 
\begin{equation}\label{eq:SphereModulo}
\SO_3(\bZ)\backslash\SO_3(\bR)/\SO_2(\bR)\cong \SO_3(\bZ)\backslash\bS^2.
\end{equation} 
Said differently, $\SO_2(\bR)$-orbits in the homogeneous space $\SO_3(\bZ)\backslash\SO_3(\bR)$ correspond to points in $\SO_3(\bZ)\backslash\bS^2$.

After all of these preparations,  we are finally ready to state the upshot of this subsection, which we can then also interpret dynamically. For $v\in \bS^2(D)$ choose $k_v\in \SO_3(\bR)$ with $k_v(e_3)=\frac1{\sqrt{D}}\cdot v$. With this choice, we have $\bH_v(\bR)=k_v\SO_2(\bR)k_v^{-1}$. Tracing through the definitions and identifications above, one is led to the following conclusion:  the projection (via $\pi$ from \eqref{eq:CoveringMap}) of the adelic torus orbit 
\begin{equation}\label{eq:AdelicOrbit}
\SO_3(\bQ)\bH_v(\bA)(k_v,e_f)\subset \SO_3(\bQ)\backslash\SO_3(\bA)    
\end{equation}
 to the real homogeneous space $\SO_3(\bZ)\backslash\SO_3(\bR)$, where $e_f\in SO_3(\bA_f)$ denotes the identity element, is a collection of $\SO_2(\bR)$-orbits which, under \eqref{eq:SphereModulo}, corresponds precisely to $\gen(v)$ viewed as a subset of $\SO_3(\bZ)\backslash \bS^2$. We urge the reader to verify this as a good way to repeat all the algebraic considerations above. In particular, this reproduces the map related to $\bH_v$ discussed in \ref{subsec:ClassGroups}, and explains how it arises naturally.

Up to now we were just chasing definition and identifications, and from time to time, were hinting at a dynamical interpretation. We finally can explain how dynamics can give us  valuable input about the distribution of $\gen(v)$ (which is equal to $\bS^2(D)$ in our specific case). For all we know, $\gen(v)$ can be quite a boring set. What we know from the above, is that elements of $\gen(v)$ are of the form $\gamma.v,\,\gamma\in \SO_3(\bQ)$ for $\gamma$'s that were found in order to “bring large elements in $\SO_3(\bA_f)$ back to $\SO_3(\widehat{\bZ})$”. Let's explain this more precisely, under the assumption that the adelic orbit \eqref{eq:AdelicOrbit} “goes everywhere”.  That is, we assume that it is dense in $\SO_3(\bQ)\backslash\SO_3(\bA)$. Then, in particular, for an arbitrary $g_\infty\in \SO_3(\bR)$ we can find $h\in \bH_v(\bA_f)$ with 
\begin{equation}\label{eq:Approx}
\SO_3(\bQ)(k_v,h)\approx\SO_3(\bQ)(g_\infty,e_f).\end{equation}
Projecting \eqref{eq:Approx} to the real homogeneous space, this approximation found for us a $\gamma\in \SO_3(\bQ)$ with $\gamma h\in \SO_3(\widehat{\bZ})$, such that the vector $\gamma v\in \bZ^3$ satisfy 
\begin{equation}\label{eq:goingEverywhere}
\gamma.v=\gamma k_v(e_3)\approx g_\infty (e_3)=\text{An arbitrary direction in $\bS^2$!}    
\end{equation}
Note that in order to have a chance that the orbit \eqref{eq:AdelicOrbit} “will go everywhere”, $\bH_v(\bA)$ needs to be a large group. In particular, our dynamical considerations are (unfortunately!) usually restricted to working in one particular place, say at a fixed prime $p$. This explains the congruence conditions that we must impose on $D=\norm{v}^2$ in order to use these dynamical considerations: they are there to verify that $\bH_v$ splits at the fixed prime $p$ and consequently to ensure that $\bH_v(\bQ_p)$ contains large elements. 

\begin{rema}
One can consider the above methods in higher dimensions, say, for the quadratic form $\sum_{i=1}^n x_i^2$ with $n\geq 6$. The dynamical methods showing the equidistribution of orbits as in \eqref{eq:AdelicOrbit} with $3$ replaced by $n$ are much simpler (they are easy special cases of \textcite{MozesShah95} and \textcite{GorodnikOh11} since $\bH_v<\SO_n$ is maximal). Based on this, one can work out from our discussion above a proof of the equidistribution of integer points on $\bS^{n-1}$ based on homogeneous dynamics.
\end{rema}

\subsection{Arithmetically motivated joinings problems}\label{subsec:ArithGeoJoinings}

In \S \ref{subsec:ManyJointExamples} we consider several arithmetic problems which are mainly coupling of  individual equidistribution problems of the type that we discussed in \S \ref{sec:TorusOrbitsArith}. 
For all the problems presented below, there is no complete solution.
In \S \ref{subsec:WithoutConjCond} we discuss  partial results have been achieved without using Theorem \ref{thm:mainThm}. Then, in \S \ref{subec:PartSolWithMainThm} we discuss partial results that use Theorem \ref{thm:mainThm}. For some of them, Theorem \ref{thm:mainThm} gives a key input, which allows for a solution under congruence conditions at two distinct prime numbers. We will also discuss in \S \ref{subec:PartSolWithMainThm} why and how these congruence conditions are used in order to make Theorem \ref{thm:mainThm} applicable.

\subsubsection{Several explicit and implicit joint distribution problems}\label{subsec:ManyJointExamples}
The first application of Theorem \ref{thm:mainThm} was already discussed in the introduction:
\begin{exem}\label{ex:1in3}
  With notation as in equation \eqref{JdFor1in3}, one asks about the equidistribution of 
  $$
  J_D=\set{\pa{\tfrac{1}{\sqrt D}v,[\Lambda_v]}: v\in \bS^2(D)}\subset \bS^2\times\pa{\SL_2(\bZ)\backslash\bH}
  $$
  when $D\to \infty$ along $D=0,4,7\mod 8$. Recall from \S \ref{sec:TorusOrbitsArith} that\footnote{to be completely precise, we must consider $\SO_3(\bZ)\backslash \bS^2(D)$ and possibly other finite-index issues, but we ignore them for simplicity}  $\bS^2(D)$ is a torsor of class group  $\Pic(\cO_{-D})$.  It turns out, that the set $\set{[\Lambda_v]: v\in \bS^2(D)}$ is also contained in another torsor of $\Pic(\cO_{-D})$: the set $\set{[\Lambda_v]: v\in \bS^2(D)}$ is a subset of $\cH_{-D}$ when $D=3\mod 4$ and of $\cH_{-4D}$ when $D=1,2\mod 4$, and the containing set in both cases is a torsor of the class group $\Pic(\cO_{-D})$. The novelty of \textcite{AES2016} was to relate (in the spirit of \S \ref{sec:TorusOrbitsArith}) the distribution of $J_D$ to the distribution of a \emph{joined} adelic orbit of a stabilizer group (as in \S \ref{subsec:ClassGroups}) in a product of two homogeneous spaces, which correspond to $\bS^2$ and $\SL_2(\bZ)\backslash\bH$. This stabilizer group is again a $\bQ$-anisotropic algebraic torus of rank one. The distribution of these adelic orbits can be studied via Theorem \ref{thm:mainThm}, once one assumes congruence conditions on $D$. We will give more details in \S \ref{subec:PartSolWithMainThm}.
\end{exem}

Although Example \ref{ex:1in3} was motivated by a geometric construction, examining it closely through the $\Pic(\cO_{-D})$-torsor structure on $\bS^2(D)$ and on $\cH_{-D}$ (or on $\cH_{-4D}$), one can see that $J_D$ corresponds to the set
\begin{equation}\label{eq:ArithStrOrtoLattc}
    \set{\pa{t.v,t^2.[\Lambda_v]}:t\in \Pic(\cO_{-D})}
\end{equation}
where $t^2$ stand for the square of $t$ in the class group $\Pic(\cO_{-D})$. This leads to several interesting and arithmetically motivated, “coupled” problems. Before stating them, we want to stress our point of view in this section: in the first component of \eqref{eq:ArithStrOrtoLattc} we consider an orbit of the class group $\Pic(\cO_{-D})$ and in the second component of \eqref{eq:ArithStrOrtoLattc} we consider an orbit of a subgroup of $\Pic(\cO_{-D})$, namely, the subgroup of squares in $\Pic(\cO_{-D})$. 

Let us first concentrate  on equidistribution in each component, that is, on the interesting problem of \emph{individual} equidistribution of orbits of  subgroups of $\Pic(\cO_{-D})$, or more generally of cosets of such subgroups.  
We recall that by \textcite{Siegel1935}  we know that $\av{\Pic(\cO_{-D})}=D^{\frac{1}{2}+o(1)}$. One can easily find (for example for the torsor $\cH_{-D}$) subgroups of size $D^{o(1)}$ whose orbits do not equidistribute (see also the work of \textcite{McMullenUniformly} for interesting results in this context). \textcite[Conjeture 1.11]{ELMVDukePeriodic}  make a bold conjecture in this context. We phrase it in the above context (e.g., for the $\Pic(\cO_{-D})$-torsors $\cH_{-D}$ or $\bS^2(D)$), but it can also be phrased for  all the problems related to toral packets we considered in \S \ref{subsec:GenrQA}.
\begin{conj}\label{conj:equidisSmallCollections} 
Let $\rho>0$. Individual equidistribution holds for orbits of cosets of subgroups of size $D^{\rho}$.
\end{conj}
Note that the generalized Riemann conjecture implies this conjecture for $\rho>\frac14$. \textcite{MichelHarcos06} prove this conjecture for $\rho\in (\frac{1}{2}-\epsilon_0,\frac{1}{2}]$ for some $\epsilon_0>0$ (e.g.~for $\epsilon_0=\frac{1}{2827}$ in the context of $\cH_{-D}$). See also \textcite{SubCollecDuke} for a much weaker result achieved by homogeneous dynamics techniques. It is a fascinating conjecture from the homogeneous dynamics perspective: until now there are no results that could utilize the algebraic structure of the class group $\Pic(\cO_{-D})$ (as the latter is rather mysterious, this unfortunately makes sense). 

Individual equidistribution plays an important role in the following two examples:

\begin{exem}\label{ex:PowerInClassGroup}
We generalize the above setting. For $1\leq i\leq r$ we let $X_i$ denote a measure space with a uniform measure $m_{X_i}$. The reader should think about $X_i=\bS^2$ or about $X_i=\SL_2(\bZ)\backslash \bH$ as in the examples above. Let $\cG_{D}^{(i)}\subset X_i$ denote (identical or different) torsors  of $\Pic(\cO_{-D})$, for $D$'s in a sequence tending to $\infty$. Although the above is phrased generally, we actually only consider here the problems we discussed in \S \ref{sec:TorusOrbitsArith} (more precisely, problems that fit to the general formulation given in \S \ref{subsec:GenrQA} or small variants of them).
For concreteness, the reader should think about $\cH_{-D}$, or $\bS^2(D)$, or $\CM_{-D}$  from \S \ref{subsec:CMSS}.

Assume that  $\cG_{D}^{(i)}$ equidistribute to $m_{X_i}$. Fix some exponents $k_1,\dots,k_r$ and base points $g_D^{(i)}\in \cG_{D}^{(i)}$, and consider 
  \begin{equation}
      P_D:=\set{\pa{t^{k_1}.g_D^{(1)},\dots,t^{k_r}.g_D^{(r)}}:t\in \Pic(\cO_{-D})}
  \end{equation}
Does $P_D$ equidistribute to $m_{X_1}\otimes\dots\otimes m_{X_r}$? A necessary condition is of course that individual equidistribution holds in each component. 
As we explained before this example, individual equidistribution  relates to the growth  of the subgroup of $k_i$-powers in $\Pic(\cO_{-D})$, or equivalently, to the growth of the $k_i$-torsion subgroup of $\Pic(\cO_{-D})$.  Assuming a folklore conjecture (see \textcite[(1.2)]{EllebergPierceWood} and the references within) about the growth of $k$-torsion in $\Pic(\cO_{-D})$, the growth of the $k$-powers subgroup of $\Pic(\cO_{-D})$ has conjecturally the same exponent as the growth of $\Pic(\cO_{-D})$. Therefore,  conjecturally, individual equidistribution follows from the work of \textcite{MichelHarcos06} mentioned above. This folklore conjecture is holds for $k=2$, so individual equidistribution  in this case is known. 

As the work of \textcite{MichelHarcos06} handles cosets equally well (or just by changing the base point for the torsor), individual equidistribution, under the conjecture about the $k$-torsion, is true also when we consider a shifted version of $P_D$: choose arbitrary  $t_D^{i}\in \Pic(\cO_{-D})$ and consider 
 \begin{equation}
      P_D^{\mathrm{shifted}}:=\set{\pa{(t^{k_1}t_D^{1}).g_D^{(1)},\dots,(t^{k_r}t_D^{r}).g_D^{(r)}}:t\in \Pic(\cO_{-D})}.
  \end{equation}
\end{exem}
As we said above, the $X_i$'s which are of interest to us in this survey, come from the general setting explained in \S \ref{subsec:GenrQA}, and  correspond to homogeneous spaces  which are related to algebraic groups of the form $\bG_i=PB_i^\times$ for quaternion algebras $B_i$ over $\bQ$,  i.e., to  $\bQ$-forms of $\SL_2$. We will explain in \S \ref{subec:PartSolWithMainThm}, that assuming individual equidistribution (for which the shifts $t^i_D$ are irrelevant), we could  essentially show via Theorem \ref{thm:mainThm}, that equidistribution of $P_D^{\mathrm{shifted}}$, for arbitrary shifts, follows in the following cases (as usual, under congruence condition at two primes): 
\begin{enumerate}
    \item All the corresponding $\bG_i$ are non-isogeneous over $\bQ$. 
    \item For any $i\neq j$ with $\bG_i$ and $\bG_j$ being $\bQ$-isogeneous, we have $k_i\neq k_j$.
\end{enumerate}
But when the exponents are equal, we cannot expect equidistribution of arbitrary shifts even if individual equidistribution is known. Indeed, say $r=2$, $X=X_1=X_2$  and $k_1=k_2$. If one chooses $t_D^{1}=t_D^{2}$ for every $D$, the collection $P_D^{\mathrm{shifted}}$ will converge to the diagonal joining, that is, to  $\iota_{*}(m_X)$ where $\iota:X\to X\times X,x\mapsto (x,x)$. Also, if the quotient $t_D^{1}\pa{t_D^{2}}^{-1}$ does not “grow”, one cannot expect equidistribution. In order to explain what we mean by “grow”,  and state a  necessary  condition for equidistribution, we define a “size”-function $N$ on $\Pic(\cO_{-D})$. Recall that $\Pic(\cO_{-D})$ is the ideal class group of $\cO_{-D}$ which, by definition, is the set of classes of invertible fractional ideals in the field of fractions of $\cO_{-D}$ modulo principal fractional ideals. Using the Norm map $\Nr$ on ideals of $\cO_{-D}$, we define for $t\in \Pic(\cO_{-D})$  
 $$
 \mathrm{N}(t)=\min\set{\Nr(\goa):\goa\subset \cO_{-D} \text{ is an invertible ideal in the class of } t}.
 $$
With this size-function $N$, we can formulate a necessary condition for equidistribution of $P_D^{\mathrm{shifted}}$: for any $1\leq i\neq j\leq r$ with $k_i=k_j$ and $\bG_i=\bG_j$,  the shifts must satisfy
 \begin{equation}\label{eq:NormToGod}
 \mathrm{N}\pa{t_D^{i}{t_D^{j}}^{-1}}\stackrel{D\to \infty}{\longrightarrow}\infty.
  \end{equation}
  Otherwise, say, if one of these size-functions is bounded by $M$ (say, for some $i\neq j$), any limit measure will be supported on a union of the graphs of  Hecke-correspondences of level $\leq M$ of a diagonal joining (in the $i,j$-components). 
\textcite{MichelVenkateshICM} conjectured that this is the only obstruction:
\begin{conj}[Mixing conjecture of Michel and Vekatesh]\label{conj:Mixing}
  In the above setting (assuming individual equidistribution or just $k_1=\dots=k_r=1$), if \eqref{eq:NormToGod} is satisfied, then $P_D^{\mathrm{shifted}}$ equidistributes to $m_{X_1}\otimes\dots\otimes m_{X_r}$.
\end{conj}

In a celebrated work, \textcite{khayutin2019joint}  essentially proves this conjecture, under the usual assumption of congruence conditions on $D$ at two primes, and some other (arguably) mild assumptions. More precisely, Khayutin considers the case of the torsors $\cH_{-D}$ with $r=2$ and $k_1=k_2=1$ and proves that under the usual congruence conditions at two arbitrary primes (specifically, $-D$ must be  a square modulo these two arbitrary primes) and under the assumption that the Dedekind $\zeta$-function of the fields $\bQ(\sqrt{-D})$
have no exceptional Landau--Siegel zeroes, equidistribution holds. His methods should work equally well for the general case we stated here (assuming individual equidistribution holds).
This is the only application of Theorem \ref{thm:mainThm} so far, where Theorem \ref{thm:mainThm} does not rule out diagonal joinings, so other tools must be exploited. We will explain how Theorem \ref{thm:mainThm} was used by Khayutin in  \S  \ref{subec:PartSolWithMainThm}.

\begin{exem}\label{ex:RedCM}
  We return to the problem we posed in the end of \S \ref{subsec:CMSS}: is the map $(\red_p,\red_q)$ surjective for $\av{D}$  large enough? Can  this be generalized to any finite number of primes?  This arithmetic question is implicitly a joint distribution problem, and fit exactly in the setting we discuss above. It turns out that it can  answered positively under congruence conditions at two distinct primes with Theorem \ref{thm:mainThm}. See \S \ref{subec:PartSolWithMainThm} below for more details.
\end{exem}

\begin{exem}\label{exa:2in4}
Going back to Example \ref{ex:1in3}, one can consider analogues of it in higher dimensions. Consider $\bQ^n$ and fix $0<k<n$. Let $\cL_{D}$ denote the set of $k$-dimensional subspaces in $\bQ^n$ with discriminant $D$. By the discriminant of a subspace $L$ we mean the square of the  covolume of the lattice $\Lambda_L:=\bZ^n\cap L$ inside $L$. One can consider the joint equidistribution of the following three objects: the orientation of $L$, viewed via the corresponding point in the Grassmannian of $k$-dimensional subspaces of $\bQ^n$, the $k$-dimensional lattice $[\Lambda_L]$, and the $n-k$ dimensional orthogonal lattice $[\Lambda_{L^\perp}]=[L^\perp\cap\bZ^n]$  (both considered up-to rotation). This leads to a similar construction of adelic orbits and therefore to interesting joinings, but in the case where $k$ or $n-k$ are greater than~$2$, such joinings will be invariant under a group containing unipotent elements. This enables the use of tools from unipotent dynamics to classify the possible joinings. 

The remaining case, where $k=2,n=4$, turns out to be very much related to the classical problems we considered above: first, for $L\in \cL_{D}$ the lattice in $[\Lambda_L]$ and the orthogonal lattice $[\Lambda_{L^\perp}]$ are both two-dimensional lattices and as such (essentially, up-to primitivity issues) are points in $\cH_{-D}$ (e.g.~by considering the binary quadratic forms $(x_1^2+x_2^2+x_3^2+x_4^2)|_{\Lambda_L}$ and $(x_1^2+x_2^2+x_3^2+x_4^2)|_{\Lambda_{L^\perp}}$). Second, the local isomorphism between $\SO_4$ and $\SO_3\times\SO_3$ implies that elements $L\in \cL_D$ are essentially parameterized by $\bS^2(D)\times \bS^2(D)$, say by $\pa{v_1^L,v_2^L}$. Hoping for a miracle, one can also consider the orthogonal lattices $[\Lambda_{v_1^L}]:=[\pa{v_1^L}^\perp\cap \bZ^3]$ and $[\Lambda_{v_2^L}]:=[\pa{v_2^L}^\perp\cap \bZ^3]$ and ask about the joint distribution of 
\begin{equation}\label{eq:Jdfor2in4}
    J_D=\set{\pa{v_1^L,v_2^L,[\Lambda_{L}],[\Lambda_{L^\perp}], [\Lambda_{v_1^L}],[\Lambda_{v_2^L}] }:L\in \cL_D}.
\end{equation}
Does $J_D$ equidistribute to $\pa{m_{\bS^2}}^{\otimes 2}\otimes \pa{m_{\SL_2(\bZ)\backslash\bH}}^{\otimes 4}$? It turns out, that this geometric construction corresponds to a natural arithmetic problem, of the form we discussed above:  with the notation from Example \ref{ex:PowerInClassGroup}, consider  the $\Pic(\cO_{-D})$-torsors 
$$
X_1=X_2=\bS^2,\, \cG_D^{(1)}=\cG_D^{(2)}=\tfrac{1}{\sqrt{D}}\bS^2(D),\, X_i= \SL_2(\bZ)\backslash\bH,\, \cG_D^{(i)}=\cH_D,\,i=3,4,5,6.
$$
Then, roughly speaking, and  after choosing correctly base points $g_D^{(i)}\in \cG_D^{(i)}$, the set $J_D$ in \eqref{eq:Jdfor2in4} corresponds to 
\begin{equation}\label{eq:arithReform2in4}
\set{\pa{t.g_D^{(1)},s.g_D^{(2)},\pa{ts}.g_D^{(3)},\pa{ts^{-1}}.g_D^{(4)},\pa{t^2}.g_D^{(5)},\pa{s^2}.g_D^{(6)}}:(t,s)\in \pa{\Pic(\cO_{-D})}^2}.    
\end{equation}
This implies that understanding the distribution of this set, as we will explain in \S \ref{subec:PartSolWithMainThm}, is closely related to classifying joinings in  $S$-adic analogues of examples \ref{exa:SL2Squared}\eqref{exa:SL245Rotation} and  \ref{exa:SL2Squared}\eqref{exa:SL2DifferentSpeeds}. We refer the reader to \textcite[\S 8]{AEW21} for a further discussion of arithmetic distribution problems in the spirit of \eqref{eq:arithReform2in4}.
\end{exem}

Generally, and in particular in homogeneous dynamics, one expects equidistribution to hold, unless there is a “visible” obstacle. In particular, we expect the Mixing conjecture (Conjecture \ref{conj:Mixing}) to hold, the equidistribution of $J_D$ from Example \ref{ex:1in3} (or equivalently of the set in \eqref{eq:ArithStrOrtoLattc}) to hold, the equidistribution of $J_D$ from Example \ref{exa:2in4} to hold, and after reformulating it as an equidistribution problem, we expect the reduction maps from Example \ref{ex:RedCM} to “equidistribute”. As we said above, none of these question admits a complete, unconditional solution.

\subsubsection{Partial results without using Theorem \ref{thm:mainThm}}\label{subsec:WithoutConjCond}
Direct generalization of the methods used in the various analytic proofs of Duke's Theorem can yield averaged results (that is, the equidistribution of the union over $\set{D:D\leq M}$ with $M\to\infty$), see the appendix by Ruixiang Zhang in the ArXiv version of \textcite{aka2015integerArXiv}. This is related to the work of \textcite{Maass56}, who considered special cases of Example~\ref{exa:2in4} also on the average. Similar averaged results of special cases of Example~\ref{exa:2in4} were proven by \textcite{Schmidt} and recently sharpened by \textcite{horesh2020equidistribution}.

\textcite{blomer2020simultaneous} found recently an analytic approach to study joined equidistribution problems involving only two individual equidistribution problems as above,  for forms of $\SL_2$ when $r=2$. They can prove the conjectured equidistribution under the assumption of the generalized Riemann hypothesis, but without the need for auxiliary congruence conditions. This gives a  conditional solution to many of the problems considered above, and  in particular to the problems mentioned in Example~\ref{ex:1in3}, and in~\ref{ex:RedCM} for the reduction with respect to two distinct primes.  

\subsubsection{Results toward equidistibution using Theorem \ref{thm:mainThm}}\label{subec:PartSolWithMainThm}

We begin by explaining, in the spirit of \S \ref{subsec:GenrQA}, how the problems mentioned in \S \ref{subsec:ManyJointExamples} correspond to “joined” adelic orbits in adelic homogeneous spaces.
We then explain the need for auxiliary congruence conditions and how they give rise to a class-$\cA'$ homomorphism as in Theorem~\ref{thm:mainThm}, enabling the use of this theorem. Finally, we describe how Theorem~\ref{thm:mainThm} was used so far to obtain partial solutions of the problems we mentioned in \S \ref{subsec:ManyJointExamples}.

\subsubsection*{Coupling toral packets together}
Going back to the setting of \S \ref{subsec:GenrQA}, we consider $r$  quaternion algebras $B_1,\dots,B_r$ over $\bQ$ and let $\bG_i=PB_i^\times$ be the corresponding algebraic groups. 
Assume that there exists a rank one $\bQ$-anisotropic algebraic torus $\bT_n$ such that for any $1\leq i\leq r$, there is an embedding $\iota_{i,n}\colon\bT_n\to \bG_i$ (e.g., via finding corresponding vectors in $B_i^0$, see \S \ref{subsec:GenrQA}) giving rise to  a toral packet  
\begin{equation}
    \bG_i(\bQ)\iota_{i,n}(\bT_n(\bA))g_{i,n}\subset \bG_i(\bQ)\backslash\bG_i(\bA),\quad g_{i,n}\in \bG_i(\bA)
\end{equation}
with discriminant $D_n$  (as in \eqref{eq:toralPacket}). Assume that $D_n\stackrel{n\to\infty}{\longrightarrow}\infty$ or $D_n\stackrel{n\to\infty}{\longrightarrow}-\infty$ and that these toral orbits converge to a measure $\mu_i$ on $\bG_i(\bQ)\backslash\bG_i(\bA)$ when $n\to\infty$. For example, we know by Duke's Theorem (as stated in \S \ref{subsec:GenrQA}) that each $\mu_i$ must be a $\bG_i(\bA)^+$-invariant measure. 
Now, let $\bG=\prod_{i=1}^r \bG_i$,  define $\iota_n:\bT_n\to \bG$  as $\iota_n=(i_{1,n},\dots,\iota_{r,n})$, and consider the following joined orbit:
\begin{equation}
    \cJ_n:=\bG(\bQ)\iota_{n}(\bT_n(\bA))\underbrace{(g_{1,n},\dots,g_{r,n})}_{g_n:=}\subset \bG(\bQ)\backslash\bG(\bA).
\end{equation}
As in \eqref{eq:toralPacket}, we equip $\cJ_n$ with $m_{\cJ_n}$, the pushforward of the uniform probability measure on $\bT_n(\bQ)\backslash\bT_n(\bA)$ under $t\mapsto \iota_n(t)g_n$.
It follows that any weak-* limit $\eta$  of $\set{m_{\cJ_n}}$ is a probability measure projecting via the natural projections $\pi_i\colon\bG\to\bG_i$ to the measures~$\mu_i$, $i=1,\dots r$. 

To be able to use homogeneous dynamics methods, we must ask ourselves, under which elements of $\bG(\bA)$ is $\eta$ invariant. Generally, the measure $\eta$ is invariant under limits in the Chabauty-topology of the subgroups $g_n^{-1}\iota_{n}(\bT_n(\bA))g_n$. We already explained in the paragraph after \eqref{eq:goingEverywhere}, that (most) of the methods in homogeneous dynamics, essentially work only when we restrict to a \emph{fixed} place (or finitely many fixed places). We also gave intuition there to why having “large” elements in the toral packet (or equivalently, having a $\bQ_p$-\emph{split} torus $\bT_n(\bQ_p)$) is needed in order to expect equidistribution and in order to use dynamical methods. Going back to the situation at hand, we see, that in order to obtain elements at the place $p$,  leaving the measure $\eta$ invariant, we must look at limits of subgroups of the form
\begin{equation}\label{eq:TorSeqGoingToGod}
g_{n,p}^{-1}\iota_{n,p}(\bT_n(\bQ_p))g_{n,p}   
\end{equation}
where $g_{n,p}$ denotes the $p$-adic component of $g_{n}$, and similarly for $\iota_{n,p}$.  Our only chance to pinpoint a non-trivial element in the limit of these $\bQ_p$-tori is to assume that all of them split over~$\bQ_p$. Asking for this splitting, is equivalent to imposing a congruence condition on the corresponding discriminant~$D_n$ modulo~$p$ (normally, depending slightly on the exact definition of discriminant,  $-D_n$ should be congruent to a square modulo~$p$). Then, in the limit we have two options: either the groups in \eqref{eq:TorSeqGoingToGod} “degenerate” to a one-parameter unipotent subgroup, in which case, we can use methods from unipotent dynamics mentioned in \S \ref{subsec:settings} (compare also to the work of \cite{EskinMosezShah96}),  or we can find a $\bQ_p$-split torus in the limit (for the problems considered above, the latter is more plausible: if the $\bQ_p$-split tori~$\bT_n$ are the stabilizers of integral vectors, say in $\bZ^3$, their Chabauty limit will contain the stabilizer of a vector in $\bZ^3_p$). Therefore, fixing two distinct places~$p$ and~$q$, and assuming that we don't get any unipotent invariance in the limit, we get invariance under a split torus at~$p$ and at~$q$. That is, we can find elements $a_p$ and $a_q$ in the corresponding tori,  such that the homomorphism $\phi\colon\bZ^2\to\bG(\bA)$ mapping two generators of $\bZ^2$ to the elements of $\bG(\bA)$ corresponding to $a_p$ and $a_q$, will be of class-$\cA'$ and the limit measure $\eta$ will be invariant under its image. In summary, under the assumption of congruence conditions at two arbitrary fixed primes, $\eta$~is a joining for which Theorem~\ref{thm:mainThm} applies. 

Under these two congruence conditions, Theorem \ref{thm:mainThm} reduces the classification of each of the \emph{ergodic component} of $\eta$ to an algebraic problem concerning the algebraic nature of $\bG_i$ and $\iota_{i,n}$. We remark, that the latter is true, at least up to the so-called character spectrum, that is up-to the difference between $\bG(\bA)^+$-invariance and $\bG(\bA)$-invariance. Indeed, note that the measures $\mu_i$ above are known to be in general only $\bG_i(\bA)^+$-invariant. Any input we will get from Theorem \ref{thm:mainThm}, will give us information on $\eta$, up-to the so-called character-spectrum $\bG(\bA)/\bG(\bA)^+$. We will avoid giving more details here about this issue (the interested reader may consult \textcite[\S 9.1]{aka2020simultaneous} or \cite[\S 3.3]{khayutin2019joint}).  

We discuss now each of the possible algebraic situations that may occur separately, explain what Theorem~\ref{thm:mainThm} can tell us in each situation, and say which of the problems we considered in~\ref{subsec:ArithGeoJoinings} fit in each situation. Before starting, note that we formulated Theorem~\ref{thm:mainThm} in an $S$-adic setting (see \cite[Therom 1.8]{EL_Joinings2019} for an adelic statement), so let's restrict ourselves for simplicity to an $S$-adic situation where the primes~$p$ and~$q$ above are contained in~$S$.

\subsubsection*{Case I: when the groups are non-isogenous over $\bQ$}

When the groups $\bG_i$ are non-isogenous over $\bQ$, Theorem \ref{thm:mainThm} tells us that each ergodic component of $\eta$ must be invariant under a finite-index subgroup of $\bG(\bQ_S)$ where $\bG=\prod_i\bG_i$. Indeed, any stricly-contained algebraic $\bQ$-subgroup of $\bG$ projecting onto $\bG_i$, must give rise to a $\bQ$-isogeny between two factors. In other words, essentially (depending on small details as the character spectrum, for example), each ergodic component of $\eta$ must be the trivial joining, and therefore also $\eta$ itself.  Theorem \ref{thm:mainThm} has been applied so far in such situations in the following cases:
\begin{itemize}
    \item The desired equidistribution in Example \ref{ex:1in3} follows under congruence conditions at two primes, since the non-$\bQ$-isogenous groups $\SO_3$ and $\SL_2$ are considered there. See \textcite{AES2016} for more details.
    \item The equidistribution of $P_D$ or $P_D^{\mathrm{shifted}}$ from Example \ref{ex:PowerInClassGroup} follows, under the following three assumptions: congruence conditions at two primes, that  individual equidistribution is assumed/known, and that the corresponding $\bG_i$ are pairwise non-$\bQ$-isogenous. One can find a more detailed treatment in the recent survey of \textcite[\S 4]{lindenstrauss2021Survey}.
    \item Under congruence conditions at two primes~$p$ and~$q$, the desired surjectivity of the reduction map we described in Example~\ref{ex:RedCM} follows from Theorem~\ref{thm:mainThm}, with respect to finitely many primes $p_1,\dots,p_r$, different from~$p$ and~$q$, as in Example~\ref{ex:RedCM}  .  Indeed, the $\bG_i$'s here correspond to the non-$\bQ$-ismorphic quaternion algebras $B_{\infty,p_1},\dots,B_{\infty,p_r}$. See \textcite{aka2020simultaneous} for more details. This example shows how the fact that ergodic theory handles arbitrary products relatively well, enabling to achieve  strong arithmetic applications, related to arbitrary products. As far as the author knows, the  surjectivity of the above reduction maps for $r>1$ was not known before for a \emph{single} discriminant.
    \item For the six-fold product considered in \eqref{eq:Jdfor2in4}, the algebraic groups corresponding to the first two factors are identical and non-$\bQ$-isogeneous to ones corresponding to the last four identical factors. So, currently, without considering the corresponding embedding $\iota_{i,n}$,  as in the general notation above, we can only conclude from Theorem \ref{thm:mainThm} the disjointness of the first two factors from the last four factors. In this case, the structure of $\iota_{i,n}$ saves the day; we will discuss it further below.
\end{itemize}

\subsubsection*{Case II: when the groups are isogenous over $\bQ$, with a non-compatible torus action.}

To simplify notation, let $\nu$ denote an ergodic component of the limit measure $\eta$ that we are trying to classify. In all the problems we consider here, let's assume that we have congruence conditions at $p$ and $q$ giving rise to $\phi\colon\bZ^2\to \bG(\bQ_S)$, a class-$\cA'$ homomorphism, with $\nu$ being invariant under $\phi(\bZ^2)$.  Assuming that the groups are isogenous,  $\nu$ might be supported on a graph of an isogeny, which implies that the invariance group of $\nu$ will be supported on a graph of an isogeny too. This invariance group must also contain  $\phi(\bZ^2)$. But, the exact structure of $\phi$ depends on $\iota_{i,n}$. In some cases, this rules out the possibility of being supported on a graph of isogeny, implying that $\nu$ must be (essentially) the trivial joining and therefore also $\eta$. This is the case in the following examples:
\begin{itemize}
    \item Going back to the six-fold product considered in \eqref{eq:Jdfor2in4}, the structure of $\phi$ in the third and fourth factor is the exact $S$-adic analogue of Example \ref{exa:SL2Squared}.\eqref{exa:SL245Rotation}, the “$45^\circ$-rotation” example (strictly speaking, as \eqref{eq:arithReform2in4} hints at, we actually consider here a torsor for $\pa{\Pic(\cO_d)}^2$ which, together with the two congruence conditions, results in having an embedding of $\bZ^4$, rather than $\bZ^2$). This “$45^\circ$-rotation” structure of $\phi$ enables to rule out rather simply\footnote{A similar argument, as in the “two ingredients joinings theorem”- Corollary \ref{cor:differntRootDisj}, would suffice here.} any joinings supported on graph of $\bQ$-isogenies, as we explained in Example \ref{ex:Exp45RotationDisj}. So $\nu$, and  therefore also $\eta$ must be (essentially) the trivial joining. See \textcite{AEW21}  for more details.
    \item The equidistribution of $P_D$ or $P_D^{\mathrm{shifted}}$ from Example \ref{ex:PowerInClassGroup} when $\bG_i$ are isogenous to each other is trickier. For concreteness, consider $r=2$ with $\bG_1=\bG_2$ and with exponents $k_1=1$ and $k_2=2$. In this case individual equidistribution is known, and examining the corresponding to $\iota_{1,n}$ and $\iota_{2,n}$, the resulting class-$\cA'$ homomorphism $\phi$ is the exact $S$-adic analogue of Example \ref{exa:SL2Squared}.\eqref{exa:SL2DifferentSpeeds}, the “different speeds” example. Therefore, also in this case, the only possibility for $\nu$ and therefore for $\eta$ is (essentially) the trivial joining. See \textcite[\S 4]{lindenstrauss2021Survey} for more details.
\end{itemize}

\subsubsection*{Case III: when the groups are isogenous over $\bQ$, with a compatible torus action.} 

The above-mentioned results were “easy” applications of Theorem \ref{thm:mainThm}. The Mixing Conjecture (Conjecture \ref{conj:Mixing}) does not fall in the above cases: for concreteness consider again $P_D^{\mathrm{shifted}}$ from Example \ref{ex:PowerInClassGroup} with $r=2$, $\bG_1=\bG_2=\SL_2$ and with exponents $k_1=1$ and $k_2=2$, and $X:=X_1=X_2=\SL_2(\bZ)\backslash\bH$, which is a case treated in the work of \textcite{khayutin2019joint}. Let $m_X$ denote the uniform measure on $X$. As we explained above, the structure of $\phi$ in this case, does not rule out joinings which are supported on a diagonal embedding of $\SL_2$ into $\SL_2\times\SL_2$. So, as far as we know, assuming congruence conditions at two primes and applying Theorem \ref{thm:mainThm}, we only get that $\nu$,  an ergodic component of a limit measure $\eta$, can either be  the trivial joining or a diagonal joining. This is nonetheless very strong information! It implies the following corollary, which is utilize and  generalized by Khayutin in his work:
\begin{coro}
  Let $U_n\subset X$ be positive measure sets with $m_X(U_n)\stackrel{n\to\infty}{\longrightarrow}  0$  (e.g., a sequence of open neighbourhood of a point/of the cusp). If we can show that 
  \begin{equation}\label{eq:wishfulBound}
  \eta(U_n\times U_n)\ll m_X(U_n)^{1+\epsilon_0}    
  \end{equation}
 for some $\epsilon_0>0$, then almost every ergodic component of $\eta$ is the trivial joining, and so is $\eta$. 
\end{coro}\label{cor:1plusEpsilon}
\begin{proof}
We explained above that Theorem \ref{thm:mainThm} implies that almost every ergodic component $\nu$ of $\eta$ is either trivial or diagonal. If $\nu$ is the trivial joining, then $\nu(U_n\times U_n)\asymp m_X(U_n)^2$, and if $\nu$ is a diagonal joining, then $\nu(U_n\times U_n)\asymp m_X(U_n)$. Assume, for contradiction, that  a positive proportion of the ergodic components are diagonal joinings. Then it follows that  $\eta(U_n\times U_n)\gg m_X(U_n)^{1+o(1)}$ contradicting \eqref{eq:wishfulBound}.
\end{proof}
Said differently, we believe and wish to show that $\eta(U_n\times U_n)$  decays with exponent $2$. Using the joining Theorem, it is enough to establish a decay with any exponent $>1$.

We explained above that in the Mixing conjecture (Conjecture \ref{conj:Mixing}) diagonal joinings correspond to measures supported on Hecke-correspondences.  
Khayutin strengthens Corollary \ref{cor:1plusEpsilon} to consider what he calls the cross-correlation between two measures, allowing to check a condition similar to \eqref{eq:wishfulBound} for measures $\nu_D$ corresponding to $P_D^{\mathrm{shifted}}$, correlated against possible measures which are supported on Hecke-correspondences. He then uses many tools (from analytic number theory and geometric invariant theory) to analyse these cross-correlations and obtain a bound similar to \eqref{eq:wishfulBound}.    

\subsection{An application of Theorem \ref{thm:mainThm} for simple groups with high rank}\label{subsec:HighRankAppli}
There are no published applications of Theorem \ref{thm:mainThm} which are not related to forms of $\SL_2$. There might be several reasons for that: as we've seen above, many classical arithmetic objects are related to toral packets of forms of $\SL_2$. The generalization of these problems to higher dimensions normally involve invariance under a group containing unipotent elements, making Theorem \ref{thm:mainThm} either inapplicable or superfluous. 

We end this survey by sketching a new application of Theorem \ref{thm:mainThm} with $\bG_1$ and $\bG_2$ being two distinct $\bQ$-forms of $\PGL_3$. 

Let $\bG_1=\PGL_3$ and set $\Gamma_1=\PGL_3(\bZ)$. Let $\bG_2$ be the $\bQ$-algebraic group associated to the group of invertible elements of a degree three (i.e., nine-dimensional) division algebra $D$, modulo its center. Assume that $D$ splits over $\bR$, that is $D\otimes \bR=M_{3\times 3}(\bR)$. Let $\cO$ be a full order contained in $D(\bQ)$; one can show that its image in $\bG_2(\bR)$ is an arithmetic, cocompact lattice, which we denote by $\Gamma_2$. Set $G_i=\bG_i(\bR)$, $X_i=\Gamma_i\backslash G_i$, $X=X_1\times X_2$, $\bG=\bG_1\times\bG_2$, and $\Gamma=\Gamma_1\times\Gamma_2$. Let $m_{X_i}$ denote the Haar probability measure on $X_i$.

\begin{theo}\label{Thm:HighRankAppli} Let $K_n$ be a sequence of totally real cubic number fields, with $\mathrm{disc}(K_n)\to\infty$, and such that for $i=1,2$ there exist embeddings $\iota_{i,n}$ of $\bT_n:=\Res_{K_n/\bQ}\bG_m/\bG_m$ into $\bG_i$, with $\Gamma_i\iota_{i,n}(\bT_n(\bR))g_i$ for some $g_i\in G_i$ being a periodic orbit, i.e., it supports a  $g_i^{-1}\iota_{i,n}(\bT_n(\bR))g_i$-invariant probability measure $\nu_{i,n}$. We define $\iota_n=(\iota_{1,n},\iota_{2,n})\colon\bT_n\to \bG$ and consider the “joined” real orbit
\begin{equation}\label{eq:joinedSL3Packet}
    \Gamma\iota_n(\bT_n(\bR))(g_1,g_2)
\end{equation}
which is also a periodic orbit supporting a probability measure $\nu_n$. Then, any weak-* limit of $\nu_n$ has full support. More precisely, any weak-* limit of $\nu_n$ is the product of $m_{X_1}$ with a weak-* limit $\mu_2$ of $\nu_{2,n}$, and $m_{X_2}$ appears with positive proportion in the ergodic decomposition of $\mu_2$. 
\end{theo}
\begin{proof}[Sketch of proof]
 We consider first what we know about individual equidistribution. In the first factor, it is a notable result of  \textcite[Theorem 1.4]{ELMVCubic} that $\nu_{1,n}$ converges to $m_{X_1}$. For the second factor, with easier methods, it is shown in \textcite{ELMVDuke}, that $m_{X_2}$ must appear with positive proportion in the ergodic decomposition of any weak-* limit of $\nu_{2,n}$.
 Consider now a weak-* limit of $\nu_n$, and denote it by $\eta$. Then $\eta$ projects to $m_{X_1}$ in the first factor and to a weak-* limit $\eta_2$ of $\nu_{2,n}$ in the second factor. 
 As we explained above, the limit measure $\eta$  is either invariant under a unipotent element (in which case we declare victory), or under a split $\bR$-torus action, that is, $\eta$ is a joining (of $m_{X_1}$ with $\eta_2$)  with a $\bZ^2$-torus action (as the split torus has dimension two). Assume we are in the second case, we let $\mu$ denote one of the ergodic components of $\eta$ with respect to this $\bZ^2$-torus action. Also here we know that $\mu$ projects to $m_{X_1}$ in the first factor and to a weak-* limit $\mu_2$ of $\nu_{2,n}$ in the second factor. With respect to this torus action in the second factor, we start by considering  an ergodic decomposition of $\mu_2$: 
 $$
 \mu_2=\int \mu_2(\xi)d\omega(\xi).
 $$
  Since $m_{X_1}$ is ergodic with respect to the action in the first factor, and the torus action is embedded diagonally in both factors, $\mu$ admits an ergodic decomposition of the form
  $$
  \mu=\int \mu(\xi)d\omega(\xi)
  $$ 
  with $\mu(\xi)$ being a  joining of $m_{X_1}$ and $\mu_2(\xi)$. We can then proceed as follows, according to the type of  $\mu_2(\xi)$:
 \begin{itemize}
     \item If $\mu_2(\xi)= m_{X_2}$, a case occurring with positive proportion, then Theorem \ref{thm:mainThm} can be applied to $\mu(\xi)$ to conclude that $\mu(\xi)=m_{X_1}\times m_{X_2}$, since $\bG_1$ and $\bG_2$ are not isogenous over $\bQ$.
     \item If $\mu_2(\xi)\neq m_{X_2}$ it must have zero entropy (otherwise, by the results of \textcite{EKL06} it will be the Haar measure), and since $(X_1,m_{X_1})$ equipped with the torus action in the first factor is a so-called $K$-system, it must be disjoint to any zero entropy system (see for example a recent survey by \cite[Theorem 1]{delaRue2020}).  
 \end{itemize} 
 Therefore, disjointness follows in both cases and so 
 $$
 \mu=\int \mu(\xi)d\omega(\xi)=\int m_{X_1}\otimes\mu_2(\xi)d\omega(\xi)= m_{X_1}\otimes\pa{\int\mu_2(\xi)d\omega(\xi)}=m_{X_1}\otimes \mu_2,
 $$
 as we wanted to show.
\end{proof}

\printbibliography

\end{document}


%% file: Macros.tex
\usepackage{amssymb}

\newcommand{\Ga}{\Gamma}

\newcommand{\cA}{\mathcal{A}}
\newcommand{\cB}{\mathcal{B}}
\newcommand{\cC}{\mathcal{C}}

\newcommand{\cG}{\mathcal{G}}
\newcommand{\cH}{\mathcal{H}}

\newcommand{\cJ}{\mathcal{J}}

\newcommand{\cL}{\mathcal{L}}

\newcommand{\cO}{\mathcal{O}}

\newcommand{\cS}{\mathcal{S}}

\newcommand{\bA}{\mathbb{A}}

\newcommand{\bC}{\mathbb{C}}

\newcommand{\bG}{\mathbb{G}}

\newcommand{\bR}{\mathbb{R}}
\newcommand{\bZ}{\mathbb{Z}}
\newcommand{\bQ}{\mathbb{Q}}
\newcommand{\bF}{\mathbb{F}}

\newcommand{\bN}{\mathbb{N}}
\newcommand{\bH}{\mathbb{H}}
\newcommand{\bT}{\mathbb{T}}
\newcommand{\bS}{\mathbb{S}}

\newcommand{\goa}{\mathfrak{a}}

\newcommand{\gog}{\mathfrak{g}}
\newcommand{\goh}{\mathfrak{h}}
\newcommand{\goi}{\mathfrak{i}}

\newcommand{\gou}{\mathfrak{u}}
\newcommand{\gov}{\mathfrak{v}}

\newcommand{\gosl}{\mathfrak{sl}}

\newcommand{\rmx}{\mathrm{x}}

\newcommand{\SL}{\operatorname{SL}}
\newcommand{\SO}{\operatorname{SO}}

\newcommand{\PGL}{\operatorname{PGL}}

\newcommand{\GL}{\operatorname{GL}}

\newcommand{\Lie}{\operatorname{Lie}}
\newcommand{\Bin}{\operatorname{Bin}}
\newcommand{\red}{\operatorname{red}}
\newcommand{\CM}{\operatorname{CM}}
\newcommand{\Nr}{\operatorname{Nr}}

\newcommand{\gen}{\operatorname{gen}}
\newcommand{\Res}{\operatorname{Res}}
\newcommand{\Pic}{\operatorname{Pic}}

\newcommand{\Ad}{\operatorname{Ad}}

\newcommand{\norm}[1]{\left\|#1\right\|}

\newcommand{\set}[1]{\left\{#1\right\}}
\newcommand{\pa}[1]{\left(#1\right)}

\newcommand{\av}[1]{\left|#1\right|}

\newcommand{\vol}{\operatorname{vol}}
\newcommand{\supp}{\operatorname{supp}}

\newcommand{\ra}{\rightarrow}

\newcommand{\onto}{\xymatrix{\ar@{>>}[r]&}}
\newcommand{\da}[4]{\xymatrix{#1 \ar@<.5ex>[r]^{#2} \ar@<-.5ex>[r]_{#3} & #4}}

\newcommand{\ExampleOne}{Example \ref{exa:SL2Squared}\eqref{exa:SL2IdenticalWeights}}
\newcommand{\ExampleTwo}{Example \ref{exa:SL2Squared}\eqref{exa:SL2DifferentSpeeds}}
\newcommand{\WeightSet}{\set{\pm\alpha,\pm\beta}}
\newcommand{\LWM}[2]{\mu_{#1}^{[{#2}]}}
\newcommand{\UAlpha}{U^{[\alpha]}}
\newcommand{\LocalInvGp}[2]{I_{#1}^{[{#2}]}}

\newcommand{\ClassGp}[1]{\Pic(\cO_{#1})}

%% file: Bib1185-Aka.bib
@article {FurstenbergDisjoint,
    AUTHOR = {Furstenberg, Harry},
     TITLE = {Disjointness in ergodic theory, minimal sets, and a problem in
              {D}iophantine approximation},
   JOURNAL = {Math. Systems Theory},
  FJOURNAL = {Mathematical Systems Theory. An International Journal on
              Mathematical Computing Theory},
    VOLUME = {1},
      YEAR = {1967},
     PAGES = {1--49},
      ISSN = {0025-5661},
   MRCLASS = {28.70 (10.00)},
  MRNUMBER = {213508},
MRREVIEWER = {W. Parry},
       DOI = {10.1007/BF01692494},
       URL = {https://doi.org/10.1007/BF01692494},
}

@book {EWBOOK,
    AUTHOR = {Einsiedler, Manfred and Ward, Thomas},
     TITLE = {Ergodic theory with a view towards number theory},
    SERIES = {Graduate Texts in Mathematics},
    VOLUME = {259},
 PUBLISHER = {Springer-Verlag London, Ltd., London},
      YEAR = {2011},
     PAGES = {xviii+481},
      ISBN = {978-0-85729-020-5},
   MRCLASS = {37A45 (05D10 11J70 11K50 28Dxx 37-01 37D40)},
  MRNUMBER = {2723325},
MRREVIEWER = {Vitaly Bergelson},
       DOI = {10.1007/978-0-85729-021-2},
       URL = {https://doi.org/10.1007/978-0-85729-021-2},
}

@article {RudolphCounterExample,
    AUTHOR = {Rudolph, Daniel J.},
     TITLE = {An example of a measure preserving map with minimal
              self-joinings, and applications},
   JOURNAL = {J. Analyse Math.},
  FJOURNAL = {Journal d'Analyse Math\'{e}matique},
    VOLUME = {35},
      YEAR = {1979},
     PAGES = {97--122},
      ISSN = {0021-7670},
   MRCLASS = {28D05},
  MRNUMBER = {555301},
MRREVIEWER = {Andr\'{e}s del Junco},
       DOI = {10.1007/BF02791063},
       URL = {https://doi.org/10.1007/BF02791063},
}

@book {GlasnerJoiningsBook,
    AUTHOR = {Glasner, Eli},
     TITLE = {Ergodic theory via joinings},
    SERIES = {Mathematical Surveys and Monographs},
    VOLUME = {101},
 PUBLISHER = {American Mathematical Society, Providence, RI},
      YEAR = {2003},
     PAGES = {xii+384},
      ISBN = {0-8218-3372-3},
   MRCLASS = {37A15 (28Dxx 37A25 37A35 37A45 37B99 54H20)},
  MRNUMBER = {1958753},
MRREVIEWER = {Andr\'{e}s del Junco},
       DOI = {10.1090/surv/101},
       URL = {https://doi.org/10.1090/surv/101},
}

@article {EL_Joinings2019,
    AUTHOR = {Einsiedler, Manfred and Lindenstrauss, Elon},
     TITLE = {Joinings of higher rank torus actions on homogeneous spaces},
   JOURNAL = {Publ. Math. Inst. Hautes \'{E}tudes Sci.},
  FJOURNAL = {Publications Math\'{e}matiques. Institut de Hautes \'{E}tudes
              Scientifiques},
    VOLUME = {129},
      YEAR = {2019},
     PAGES = {83--127},
      ISSN = {0073-8301},
   MRCLASS = {22E40 (22D40 37A05)},
  MRNUMBER = {3949028},
MRREVIEWER = {Thomas Ward},
       DOI = {10.1007/s10240-019-00103-y},
       URL = {https://doi.org/10.1007/s10240-019-00103-y},
}

@article {EL_Joinings2007,
    AUTHOR = {Einsiedler, Manfred and Lindenstrauss, Elon},
     TITLE = {Joinings of higher-rank diagonalizable actions on locally
              homogeneous spaces},
   JOURNAL = {Duke Math. J.},
  FJOURNAL = {Duke Mathematical Journal},
    VOLUME = {138},
      YEAR = {2007},
    NUMBER = {2},
     PAGES = {203--232},
      ISSN = {0012-7094},
   MRCLASS = {37A17 (22F30 28D15)},
  MRNUMBER = {2318283},
MRREVIEWER = {J\'{e}r\^{o}me Buzzi},
       DOI = {10.1215/S0012-7094-07-13822-5},
       URL = {https://doi.org/10.1215/S0012-7094-07-13822-5},
}

@article {LindenstraussQUE,
    AUTHOR = {Lindenstrauss, Elon},
     TITLE = {Invariant measures and arithmetic quantum unique ergodicity},
   JOURNAL = {Ann. of Math. (2)},
  FJOURNAL = {Annals of Mathematics. Second Series},
    VOLUME = {163},
      YEAR = {2006},
    NUMBER = {1},
     PAGES = {165--219},
      ISSN = {0003-486X},
   MRCLASS = {11F72 (37A45 37D40)},
  MRNUMBER = {2195133},
MRREVIEWER = {Ze\'{e}v Rudnick},
       DOI = {10.4007/annals.2006.163.165},
       URL = {https://doi.org/10.4007/annals.2006.163.165},
}

@incollection {PisaNotes,
    AUTHOR = {Einsiedler, Manfred and Lindenstrauss, Elon},
     TITLE = {Diagonal actions on locally homogeneous spaces},
 BOOKTITLE = {Homogeneous flows, moduli spaces and arithmetic},
    SERIES = {Clay Math. Proc.},
    VOLUME = {10},
     PAGES = {155--241},
 PUBLISHER = {Amer. Math. Soc., Providence, RI},
      YEAR = {2010},
   MRCLASS = {22F30 (22D40 22F10 28D10 28D15 37A15 58J51)},
  MRNUMBER = {2648695},
MRREVIEWER = {Thomas Ward},
       DOI = {10.4171/owr/2010/29},
       URL = {https://doi.org/10.4171/owr/2010/29},
}

@article {EKL06,
    AUTHOR = {Einsiedler, Manfred and Katok, Anatole and Lindenstrauss,
              Elon},
     TITLE = {Invariant measures and the set of exceptions to {L}ittlewood's
              conjecture},
   JOURNAL = {Ann. of Math. (2)},
  FJOURNAL = {Annals of Mathematics. Second Series},
    VOLUME = {164},
      YEAR = {2006},
    NUMBER = {2},
     PAGES = {513--560},
      ISSN = {0003-486X},
   MRCLASS = {22F10 (11J83 28A80 28D15 37A15)},
  MRNUMBER = {2247967},
MRREVIEWER = {Thomas Ward},
       DOI = {10.4007/annals.2006.164.513},
       URL = {https://doi.org/10.4007/annals.2006.164.513},
}

@article {AES2016,
    AUTHOR = {Aka, Menny and Einsiedler, Manfred and Shapira, Uri},
     TITLE = {Integer points on spheres and their orthogonal lattices},
   JOURNAL = {Invent. Math.},
  FJOURNAL = {Inventiones Mathematicae},
    VOLUME = {206},
      YEAR = {2016},
    NUMBER = {2},
     PAGES = {379--396},
      ISSN = {0020-9910},
   MRCLASS = {37A45 (11H55)},
  MRNUMBER = {3570295},
MRREVIEWER = {Thomas Ward},
       DOI = {10.1007/s00222-016-0655-7},
       URL = {https://doi.org/10.1007/s00222-016-0655-7},
}

@article {khayutin2019joint,
    AUTHOR = {Khayutin, Ilya},
     TITLE = {Joint equidistribution of {CM} points},
   JOURNAL = {Ann. of Math. (2)},
  FJOURNAL = {Annals of Mathematics. Second Series},
    VOLUME = {189},
      YEAR = {2019},
    NUMBER = {1},
     PAGES = {145--276},
      ISSN = {0003-486X},
   MRCLASS = {11G18 (37A17)},
  MRNUMBER = {3898709},
MRREVIEWER = {Thomas Ward},
       DOI = {10.4007/annals.2019.189.1.4},
       URL = {https://doi.org/10.4007/annals.2019.189.1.4},
}

@article{blomer2020simultaneous,
  title={Simultaneous equidistribution of toric periods and fractional moments of L-functions},
  author={Blomer, Valentin and Brumley, Farrell},
  journal={arXiv preprint arXiv:2009.07093},
  year={2020}
}

@incollection {EMVpoints,
    AUTHOR = {Ellenberg, Jordan S. and Michel, Philippe and Venkatesh,
              Akshay},
     TITLE = {Linnik's ergodic method and the distribution of integer points
              on spheres},
 BOOKTITLE = {Automorphic representations and {$L$}-functions},
    SERIES = {Tata Inst. Fundam. Res. Stud. Math.},
    VOLUME = {22},
     PAGES = {119--185},
 PUBLISHER = {Tata Inst. Fund. Res., Mumbai},
      YEAR = {2013},
   MRCLASS = {11K36 (11E12)},
  MRNUMBER = {3156852},
MRREVIEWER = {Vladimir S. Anashin},
}

@article {Ratner91,
    AUTHOR = {Ratner, Marina},
     TITLE = {On {R}aghunathan's measure conjecture},
   JOURNAL = {Ann. of Math. (2)},
  FJOURNAL = {Annals of Mathematics. Second Series},
    VOLUME = {134},
      YEAR = {1991},
    NUMBER = {3},
     PAGES = {545--607},
      ISSN = {0003-486X},
   MRCLASS = {22E40 (58F11 58F17)},
  MRNUMBER = {1135878},
MRREVIEWER = {S. G. Dani},
       DOI = {10.2307/2944357},
       URL = {https://doi.org/10.2307/2944357},
}

@article {RatnerSadic95,
    AUTHOR = {Ratner, Marina},
     TITLE = {Raghunathan's conjectures for {C}artesian products of real and
              {$p$}-adic {L}ie groups},
   JOURNAL = {Duke Math. J.},
  FJOURNAL = {Duke Mathematical Journal},
    VOLUME = {77},
      YEAR = {1995},
    NUMBER = {2},
     PAGES = {275--382},
      ISSN = {0012-7094},
   MRCLASS = {22E40 (22E50)},
  MRNUMBER = {1321062},
MRREVIEWER = {Dave Witte Morris},
       DOI = {10.1215/S0012-7094-95-07710-2},
       URL = {https://doi.org/10.1215/S0012-7094-95-07710-2},
}

@article {MargulisTomanov94,
    AUTHOR = {Margulis, Gregory A. and Tomanov, George M.},
     TITLE = {Invariant measures for actions of unipotent groups over local
              fields on homogeneous spaces},
   JOURNAL = {Invent. Math.},
  FJOURNAL = {Inventiones Mathematicae},
    VOLUME = {116},
      YEAR = {1994},
    NUMBER = {1-3},
     PAGES = {347--392},
      ISSN = {0020-9910},
   MRCLASS = {22E40 (11E57 20G25 22D40)},
  MRNUMBER = {1253197},
MRREVIEWER = {Nimish A. Shah},
       DOI = {10.1007/BF01231565},
       URL = {https://doi.org/10.1007/BF01231565},
}

@incollection {EinsiedlerKatok,
    AUTHOR = {Einsiedler, Manfred and Katok, Anatole},
     TITLE = {Rigidity of measures---the high entropy case and non-commuting
              foliations},
      NOTE = {Probability in mathematics},
   JOURNAL = {Israel J. Math.},
  FJOURNAL = {Israel Journal of Mathematics},
    VOLUME = {148},
      YEAR = {2005},
     PAGES = {169--238},
      ISSN = {0021-2172},
   MRCLASS = {37C85 (37A15 37C40 37D30)},
  MRNUMBER = {2191228},
MRREVIEWER = {S. G. Dani},
       DOI = {10.1007/BF02775436},
       URL = {https://doi.org/10.1007/BF02775436},
}

@article{AEW21,
  title={Planes in four space and four associated CM points},
  author={Aka, Menny and Einsiedler, Manfred and Wieser, Andreas},
  JOURNAL = {Duke Math. J.},
  FJOURNAL = {Duke Mathematical Journal},
  year={2021},
  note = {(in press)}
}

@book {ZimmerBook,
    AUTHOR = {Zimmer, Robert J.},
     TITLE = {Ergodic theory and semisimple groups},
    SERIES = {Monographs in Mathematics},
    VOLUME = {81},
 PUBLISHER = {Birkh\"{a}user Verlag, Basel},
      YEAR = {1984},
     PAGES = {x+209},
      ISBN = {3-7643-3184-4},
   MRCLASS = {22E40 (22D40 28D15)},
  MRNUMBER = {776417},
MRREVIEWER = {S. G. Dani},
       DOI = {10.1007/978-1-4684-9488-4},
       URL = {https://doi.org/10.1007/978-1-4684-9488-4},
}

@article{AkaMussoWieser,
  title={Equidistribution of rational subspaces and their shapes},
  author={Aka, Menny and Musso, Andrea and Wieser, Andreas},
  journal={arXiv preprint arXiv:2103.05163},
  year={2021}
}

@book {PR94,
    AUTHOR = {Platonov, Vladimir and Rapinchuk, Andrei},
     TITLE = {Algebraic groups and number theory},
    SERIES = {Pure and Applied Mathematics},
    VOLUME = {139},
      NOTE = {Translated from the 1991 Russian original by Rachel Rowen},
 PUBLISHER = {Academic Press, Inc., Boston, MA},
      YEAR = {1994},
     PAGES = {xii+614},
      ISBN = {0-12-558180-7},
   MRCLASS = {11E57 (11-02 20Gxx)},
  MRNUMBER = {1278263},
}

@book {Gauss,
    AUTHOR = {Gauss, Carl Friedrich},
     TITLE = {Disquisitiones arithmeticae},
      NOTE = {Translated and with a preface by Arthur A. Clarke,
              Revised by William C. Waterhouse, Cornelius Greither and A. W.
              Grootendorst and with a preface by Waterhouse},
 PUBLISHER = {Springer-Verlag, New York},
      YEAR = {1986},
     PAGES = {xx+472},
      ISBN = {0-387-96254-9},
   MRCLASS = {01A75 (01A55)},
  MRNUMBER = {837656},
}

@book {LinnikBook,
    AUTHOR = {Linnik, Yuri V.},
     TITLE = {Ergodic properties of algebraic fields},
    SERIES = {Ergebnisse der Mathematik und ihrer Grenzgebiete, Band 45},
      NOTE = {Translated from the Russian by M. S. Keane},
 PUBLISHER = {Springer-Verlag New York Inc., New York},
      YEAR = {1968},
     PAGES = {ix+192},
   MRCLASS = {10.65},
  MRNUMBER = {0238801},
}

@article {Duke88,
    AUTHOR = {Duke, William},
     TITLE = {Hyperbolic distribution problems and half-integral weight
              {M}aass forms},
   JOURNAL = {Invent. Math.},
  FJOURNAL = {Inventiones Mathematicae},
    VOLUME = {92},
      YEAR = {1988},
    NUMBER = {1},
     PAGES = {73--90},
      ISSN = {0020-9910},
   MRCLASS = {11F11 (11E32 11F30 11F37)},
  MRNUMBER = {931205},
MRREVIEWER = {Mark Sheingorn},
       DOI = {10.1007/BF01393993},
       URL = {https://doi.org/10.1007/BF01393993},
}

@article {IwaniecBreak,
    AUTHOR = {Iwaniec, Henryk},
     TITLE = {Fourier coefficients of modular forms of half-integral weight},
   JOURNAL = {Invent. Math.},
  FJOURNAL = {Inventiones Mathematicae},
    VOLUME = {87},
      YEAR = {1987},
    NUMBER = {2},
     PAGES = {385--401},
      ISSN = {0020-9910},
   MRCLASS = {11F37 (11F30)},
  MRNUMBER = {870736},
MRREVIEWER = {Marvin I. Knopp},
       DOI = {10.1007/BF01389423},
       URL = {https://doi.org/10.1007/BF01389423},
}

@article {ELMVDuke,
    AUTHOR = {Einsiedler, Manfred and Lindenstrauss, Elon and Michel,
              Philippe and Venkatesh, Akshay},
     TITLE = {The distribution of closed geodesics on the modular surface,
              and {D}uke's theorem},
   JOURNAL = {Enseign. Math. (2)},
  FJOURNAL = {L'Enseignement Math\'{e}matique. Revue Internationale. 2e S\'{e}rie},
    VOLUME = {58},
      YEAR = {2012},
    NUMBER = {3-4},
     PAGES = {249--313},
      ISSN = {0013-8584},
   MRCLASS = {11F11 (11E16 11F37 14C22 37A45 53C22)},
  MRNUMBER = {3058601},
MRREVIEWER = {Thomas R. Shemanske},
       DOI = {10.4171/LEM/58-3-2},
       URL = {https://doi.org/10.4171/LEM/58-3-2},
}

@article {WieserLinnik,
    AUTHOR = {Wieser, Andreas},
     TITLE = {Linnik's problems and maximal entropy methods},
   JOURNAL = {Monatsh. Math.},
  FJOURNAL = {Monatshefte f\"{u}r Mathematik},
    VOLUME = {190},
      YEAR = {2019},
    NUMBER = {1},
     PAGES = {153--208},
      ISSN = {0026-9255},
   MRCLASS = {37A17 (11E12 11K06)},
  MRNUMBER = {3998337},
MRREVIEWER = {Thomas Ward},
       DOI = {10.1007/s00605-019-01320-7},
       URL = {https://doi.org/10.1007/s00605-019-01320-7},
}

@incollection {MichelVenkateshICM,
    AUTHOR = {Michel, Philippe and Venkatesh, Akshay},
     TITLE = {Equidistribution, {$L$}-functions and ergodic theory: on some
              problems of {Y}u. {L}innik},
 BOOKTITLE = {International {C}ongress of {M}athematicians. {V}ol. {II}},
     PAGES = {421--457},
 PUBLISHER = {Eur. Math. Soc., Z\"{u}rich},
      YEAR = {2006},
   MRCLASS = {11F66 (11F67 11M41 37A45)},
  MRNUMBER = {2275604},
MRREVIEWER = {Gergely Harcos},
}

@article {DeuringSurj,
    AUTHOR = {Deuring, Max},
     TITLE = {Die {T}ypen der {M}ultiplikatorenringe elliptischer
              {F}unktionenk\"{o}rper},
   JOURNAL = {Abh. Math. Sem. Hansischen Univ.},
  FJOURNAL = {Abhandlungen aus dem Mathematischen Seminar der Hansischen
              Universit\"{a}t},
    VOLUME = {14},
      YEAR = {1941},
     PAGES = {197--272},
      ISSN = {0025-5858},
   MRCLASS = {09.1X},
  MRNUMBER = {5125},
MRREVIEWER = {Saunders Mac Lane},
       DOI = {10.1007/BF02940746},
       URL = {https://doi.org/10.1007/BF02940746},
}

@article {Michel04,
    AUTHOR = {Michel, Philippe},
     TITLE = {The subconvexity problem for {R}ankin-{S}elberg
              {$L$}-functions and equidistribution of {H}eegner points},
   JOURNAL = {Ann. of Math. (2)},
  FJOURNAL = {Annals of Mathematics. Second Series},
    VOLUME = {160},
      YEAR = {2004},
    NUMBER = {1},
     PAGES = {185--236},
      ISSN = {0003-486X},
   MRCLASS = {11F66 (11G18)},
  MRNUMBER = {2119720},
MRREVIEWER = {Gergely Harcos},
       DOI = {10.4007/annals.2004.160.185},
       URL = {https://doi.org/10.4007/annals.2004.160.185},
}

@incollection {Gross87,
    AUTHOR = {Gross, Benedict H.},
     TITLE = {Heights and the special values of {$L$}-series},
 BOOKTITLE = {Number theory ({M}ontreal, {Q}ue., 1985)},
    SERIES = {CMS Conf. Proc.},
    VOLUME = {7},
     PAGES = {115--187},
 PUBLISHER = {Amer. Math. Soc., Providence, RI},
      YEAR = {1987},
   MRCLASS = {11F67 (11G40 11R52)},
  MRNUMBER = {894322},
MRREVIEWER = {Lawrence Washington},
}

@article{aka2020simultaneous,
  title={Simultaneous supersingular reductions of CM elliptic curves},
  author={Aka, Menny and Luethi, Manuel and Michel, Philippe and Wieser, Andreas},
  journal={arXiv preprint arXiv:2005.01537},
  year={2020}
}

@article{lindenstrauss2021Survey,
  title={Recent progress on rigity properties of higher rank diagonalizable actions and applications},
  author={Lindenstrauss, Elon},
  journal={arXiv preprint arXiv:2101.11114},
  year={2021}
}

@article {ELMVCubic,
    AUTHOR = {Einsiedler, Manfred and Lindenstrauss, Elon and Michel,
              Philippe and Venkatesh, Akshay},
     TITLE = {Distribution of periodic torus orbits and {D}uke's theorem for
              cubic fields},
   JOURNAL = {Ann. of Math. (2)},
  FJOURNAL = {Annals of Mathematics. Second Series},
    VOLUME = {173},
      YEAR = {2011},
    NUMBER = {2},
     PAGES = {815--885},
      ISSN = {0003-486X},
   MRCLASS = {37A17 (11E57 11M36 22E40 37A45)},
  MRNUMBER = {2776363},
MRREVIEWER = {Valentin Blomer},
       DOI = {10.4007/annals.2011.173.2.5},
       URL = {https://doi.org/10.4007/annals.2011.173.2.5},
}

@article {EV08LocalGlobal,
    AUTHOR = {Ellenberg, Jordan S. and Venkatesh, Akshay},
     TITLE = {Local-global principles for representations of quadratic
              forms},
   JOURNAL = {Invent. Math.},
  FJOURNAL = {Inventiones Mathematicae},
    VOLUME = {171},
      YEAR = {2008},
    NUMBER = {2},
     PAGES = {257--279},
      ISSN = {0020-9910},
   MRCLASS = {11E12 (11E57 11E95)},
  MRNUMBER = {2367020},
MRREVIEWER = {Marek Szyjewski},
       DOI = {10.1007/s00222-007-0077-7},
       URL = {https://doi.org/10.1007/s00222-007-0077-7},
}

@article {MozesShah95,
    AUTHOR = {Mozes, Shahar and Shah, Nimish},
     TITLE = {On the space of ergodic invariant measures of unipotent flows},
   JOURNAL = {Ergodic Theory Dynam. Systems},
  FJOURNAL = {Ergodic Theory and Dynamical Systems},
    VOLUME = {15},
      YEAR = {1995},
    NUMBER = {1},
     PAGES = {149--159},
      ISSN = {0143-3857},
   MRCLASS = {58F11 (22E40 28D10)},
  MRNUMBER = {1314973},
MRREVIEWER = {Garrett Stuck},
       DOI = {10.1017/S0143385700008282},
       URL = {https://doi.org/10.1017/S0143385700008282},
}

@article {GorodnikOh11,
    AUTHOR = {Gorodnik, Alex and Oh, Hee},
     TITLE = {Rational points on homogeneous varieties and equidistribution
              of adelic periods},
      NOTE = {With an appendix by Mikhail Borovoi},
   JOURNAL = {Geom. Funct. Anal.},
  FJOURNAL = {Geometric and Functional Analysis},
    VOLUME = {21},
      YEAR = {2011},
    NUMBER = {2},
     PAGES = {319--392},
      ISSN = {1016-443X},
   MRCLASS = {11G50 (11D45 11E72 11G35 28D15 37A17 37A45)},
  MRNUMBER = {2795511},
MRREVIEWER = {B. Sury},
       DOI = {10.1007/s00039-011-0113-z},
       URL = {https://doi.org/10.1007/s00039-011-0113-z},
}

@article{aka2015integerArXiv,
  title={Integer points and their orthogonal lattices},
  author={Aka, Menny and Einsiedler, Manfred and Shapira, Uri},
  journal={arXiv preprint arXiv:1502.04209},
  year={2015}
}

@article {MichelHarcos06,
    AUTHOR = {Harcos, Gergely and Michel, Philippe},
     TITLE = {The subconvexity problem for {R}ankin-{S}elberg
              {$L$}-functions and equidistribution of {H}eegner points.
              {II}},
   JOURNAL = {Invent. Math.},
  FJOURNAL = {Inventiones Mathematicae},
    VOLUME = {163},
      YEAR = {2006},
    NUMBER = {3},
     PAGES = {581--655},
      ISSN = {0020-9910},
   MRCLASS = {11F66 (11F67 11M41)},
  MRNUMBER = {2207235},
MRREVIEWER = {K. Soundararajan},
       DOI = {10.1007/s00222-005-0468-6},
       URL = {https://doi.org/10.1007/s00222-005-0468-6},
}

@incollection {DukeSurvey2007,
    AUTHOR = {Duke, William},
     TITLE = {An introduction to the {L}innik problems},
 BOOKTITLE = {Equidistribution in number theory, an introduction},
    SERIES = {NATO Sci. Ser. II Math. Phys. Chem.},
    VOLUME = {237},
     PAGES = {197--216},
 PUBLISHER = {Springer, Dordrecht},
      YEAR = {2007},
   MRCLASS = {11E45 (11E25 11F30 11F37 11F46 11F67 11L05)},
  MRNUMBER = {2290500},
MRREVIEWER = {Rainer Schulze-Pillot},
       DOI = {10.1007/978-1-4020-5404-4\_10},
       URL = {https://doi.org/10.1007/978-1-4020-5404-4_10},
}

@article {McMullenUniformly,
    AUTHOR = {McMullen, Curtis T.},
     TITLE = {Uniformly {D}iophantine numbers in a fixed real quadratic
              field},
   JOURNAL = {Compos. Math.},
  FJOURNAL = {Compositio Mathematica},
    VOLUME = {145},
      YEAR = {2009},
    NUMBER = {4},
     PAGES = {827--844},
      ISSN = {0010-437X},
   MRCLASS = {11J17 (11J70 37A45)},
  MRNUMBER = {2521246},
MRREVIEWER = {Jeffrey O. Shallit},
       DOI = {10.1112/S0010437X09004102},
       URL = {https://doi.org/10.1112/S0010437X09004102},
}

@article {ELMVDukePeriodic,
    AUTHOR = {Einsiedler, Manfred and Lindenstrauss, Elon and Michel,
              Philippe and Venkatesh, Akshay},
     TITLE = {Distribution of periodic torus orbits on homogeneous spaces},
   JOURNAL = {Duke Math. J.},
  FJOURNAL = {Duke Mathematical Journal},
    VOLUME = {148},
      YEAR = {2009},
    NUMBER = {1},
     PAGES = {119--174},
      ISSN = {0012-7094},
   MRCLASS = {37A17 (11E57 37A45)},
  MRNUMBER = {2515103},
MRREVIEWER = {Dmitry Y. Kleinbock},
       DOI = {10.1215/00127094-2009-023},
       URL = {https://doi.org/10.1215/00127094-2009-023},
}

@article {SubCollecDuke,
    AUTHOR = {Aka, Menny and Einsiedler, Manfred},
     TITLE = {Duke's theorem for subcollections},
   JOURNAL = {Ergodic Theory Dynam. Systems},
  FJOURNAL = {Ergodic Theory and Dynamical Systems},
    VOLUME = {36},
      YEAR = {2016},
    NUMBER = {2},
     PAGES = {335--342},
      ISSN = {0143-3857},
   MRCLASS = {37A17 (11J17 37A45 37D40 53C22)},
  MRNUMBER = {3503026},
MRREVIEWER = {Felipe Alberto Ram\'{\i}rez},
       DOI = {10.1017/etds.2014.68},
       URL = {https://doi.org/10.1017/etds.2014.68},
}

@article{horesh2020equidistribution,
  title={Equidistribution of primitive lattices in ${\mathbb R}^n$},
  author={Horesh, Tal and Karasik, Yakov},
  journal={arXiv preprint arXiv:2012.04508},
  year={2020}
}

@article {Maass56,
    AUTHOR = {Maass, Hans},
     TITLE = {Spherical functions and quadratic forms},
   JOURNAL = {J. Indian Math. Soc. (N.S.)},
  FJOURNAL = {The Journal of the Indian Mathematical Society. New Series},
    VOLUME = {20},
      YEAR = {1956},
     PAGES = {117--162},
      ISSN = {0019-5839},
   MRCLASS = {10.1X},
  MRNUMBER = {86837},
MRREVIEWER = {B. W. Jones},
}

@article {Schmidt,
    AUTHOR = {Schmidt, Wolfgang M.},
     TITLE = {The distribution of sublattices of {${\bf Z}^m$}},
   JOURNAL = {Monatsh. Math.},
  FJOURNAL = {Monatshefte f\"{u}r Mathematik},
    VOLUME = {125},
      YEAR = {1998},
    NUMBER = {1},
     PAGES = {37--81},
      ISSN = {0026-9255},
   MRCLASS = {11H06 (11H55)},
  MRNUMBER = {1485976},
MRREVIEWER = {Jeffrey Lin Thunder},
       DOI = {10.1007/BF01489457},
       URL = {https://doi.org/10.1007/BF01489457},
}

@article {EskinMosezShah96,
    AUTHOR = {Eskin, Alex and Mozes, Shahar and Shah, Nimish},
     TITLE = {Unipotent flows and counting lattice points on homogeneous
              varieties},
   JOURNAL = {Ann. of Math. (2)},
  FJOURNAL = {Annals of Mathematics. Second Series},
    VOLUME = {143},
      YEAR = {1996},
    NUMBER = {2},
     PAGES = {253--299},
      ISSN = {0003-486X},
   MRCLASS = {22E40 (11E57 11P21)},
  MRNUMBER = {1381987},
MRREVIEWER = {Alexander Starkov},
       DOI = {10.2307/2118644},
       URL = {https://doi.org/10.2307/2118644},
}

@Inbook{delaRue2020,
author="{de la Rue}, Thierry",
editor="Meyers, Robert A.",
title="Joinings in Ergodic Theory",
bookTitle="Encyclopedia of Complexity and Systems Science",
year="2020",
publisher="Springer Berlin Heidelberg",
address="Berlin, Heidelberg",
pages="1--20",
isbn="978-3-642-27737-5",
doi="10.1007/978-3-642-27737-5_300-2",
url="https://doi.org/10.1007/978-3-642-27737-5_300-2",
sortname="Rue"
}

@article {EllebergPierceWood,
    AUTHOR = {Ellenberg, Jordan S. and Pierce, Lillian B. and Wood, Melanie
              Matchett},
     TITLE = {On {$\ell$}-torsion in class groups of number fields},
   JOURNAL = {Algebra Number Theory},
  FJOURNAL = {Algebra \& Number Theory},
    VOLUME = {11},
      YEAR = {2017},
    NUMBER = {8},
     PAGES = {1739--1778},
      ISSN = {1937-0652},
   MRCLASS = {11R29 (11N36 11R45)},
  MRNUMBER = {3720931},
MRREVIEWER = {Frank Henry Thorne},
       DOI = {10.2140/ant.2017.11.1739},
       URL = {https://doi.org/10.2140/ant.2017.11.1739},
}

@article {LedYoungII,
    AUTHOR = {Ledrappier, François and Young, Lai-Sang},
     TITLE = {The metric entropy of diffeomorphisms. {II}. {R}elations
              between entropy, exponents and dimension},
   JOURNAL = {Ann. of Math. (2)},
  FJOURNAL = {Annals of Mathematics. Second Series},
    VOLUME = {122},
      YEAR = {1985},
    NUMBER = {3},
     PAGES = {540--574},
      ISSN = {0003-486X},
   MRCLASS = {58F11 (58F15)},
  MRNUMBER = {819557},
MRREVIEWER = {D. Newton},
       DOI = {10.2307/1971329},
       URL = {https://doi.org/10.2307/1971329},
}

@article {LedYoungI,
    AUTHOR = {Ledrappier, François and Young, Lai-Sang},
     TITLE = {The metric entropy of diffeomorphisms. {I}. {C}haracterization
              of measures satisfying {P}esin's entropy formula},
   JOURNAL = {Ann. of Math. (2)},
  FJOURNAL = {Annals of Mathematics. Second Series},
    VOLUME = {122},
      YEAR = {1985},
    NUMBER = {3},
     PAGES = {509--539},
      ISSN = {0003-486X},
   MRCLASS = {58F11 (58F15)},
  MRNUMBER = {819556},
MRREVIEWER = {D. Newton},
       DOI = {10.2307/1971328},
       URL = {https://doi.org/10.2307/1971328},
}

@book {LusinREF,
    AUTHOR = {Federer, Herbert},
     TITLE = {Geometric measure theory},
    SERIES = {Die Grundlehren der mathematischen Wissenschaften, Band 153},
 PUBLISHER = {Springer-Verlag New York Inc., New York},
      YEAR = {1969},
     PAGES = {xiv+676},
   MRCLASS = {28.80 (26.00)},
  MRNUMBER = {0257325},
MRREVIEWER = {J. E. Brothers},
}

@article{Siegel1935,
author = {Siegel, Carl},
journal = {Acta Arithmetica},
language = {ger},
number = {1},
pages = {83-86},
title = {Über die Classenzahl quadratischer Zahlkörper},
url = {http://eudml.org/doc/205054},
volume = {1},
year = {1935},
}
